\documentclass[aap, preprint]{imsart}

\RequirePackage[OT1]{fontenc}
\RequirePackage{amsthm,amsmath}
\RequirePackage[numbers]{natbib}
\RequirePackage[colorlinks,citecolor=blue,urlcolor=blue]{hyperref}

\arxiv{arXiv:1605.00669}

\startlocaldefs
\numberwithin{equation}{section}
\theoremstyle{plain}
\newtheorem{thm}{\protect\theoremname}[section]
  \theoremstyle{plain}
  \newtheorem{prop}[thm]{\protect\propositionname}
  \theoremstyle{plain}
  \newtheorem{cor}[thm]{\protect\corollaryname}
  \theoremstyle{definition}
  \newtheorem{defn}[thm]{\protect\definitionname}
  \theoremstyle{remark}
  \newtheorem{rem}[thm]{\protect\remarkname}
  \theoremstyle{plain}
  \newtheorem{lem}[thm]{\protect\lemmaname}
  \theoremstyle{definition}
  
 \theoremstyle{definition}
  \newtheorem{ass}[thm]{\protect\assumptionname}
\makeatother

\usepackage[english]{babel}
  \providecommand{\corollaryname}{Corollary}
  \providecommand{\definitionname}{Definition}
  \providecommand{\lemmaname}{Lemma}
  \providecommand{\propositionname}{Proposition}
  \providecommand{\remarkname}{Remark}
  \providecommand{\examplename}{Example}
  \providecommand{\assumptionname}{Assumption}
\providecommand{\theoremname}{Theorem}

\newcommand{\I}[2]{\ensuremath{\left[#1,#2\right]}}  

\newcommand{\timeInt}{\I{0}{T}}

\newcommand{\R}{\ensuremath{\mathbb{R}}}
\newcommand{\E}{\ensuremath{\mathbf{E}}}
\renewcommand{\P}{\ensuremath{\mathbf{P}}}
\newcommand{\Q}{\ensuremath{\mathbf{Q}}}
\newcommand{\N}{\ensuremath{\mathbb{N}}}
\newcommand{\1}{\ensuremath{\mathbf{1}}}

\renewcommand{\S}{\ensuremath{\mathscr{S}}}
\newcommand{\Sdual}{\ensuremath{\mathscr{S'}}}
\newcommand{\cI}{\ensuremath{\mathcal{I}}}
\newcommand{\cadlag}{c\`adl\`ag}

\newcommand{\M}{\ensuremath{\mathrm{M1}}}

\newcommand{\e}{\ensuremath{\varepsilon}}

\newcommand{\C}{\ensuremath{C^{\mathrm{test}}}}

\newcommand{\dx}{\ensuremath{\partial_{x}}}

\newcommand{\dxx}{\ensuremath{\partial_{xx}}}

\newcommand{\anti}{\ensuremath{\partial_{x}^{-1}}}

\newcommand{\nub}{\ensuremath{\bar{\nu}}}

\newcommand{\weak}{\ensuremath{\Rightarrow}}

\newcommand{\ith}{\ensuremath{i^{\mathrm{th}}}}

\usepackage{amssymb}
\usepackage{enumerate}
\usepackage{mathrsfs}
\usepackage{graphicx}

\begin{document}

\begin{frontmatter}
\title{A stochastic McKean--Vlasov equation for absorbing diffusions 
on the half-line}
\runtitle{A stochastic McKean--Vlasov equation on the half-line}

\begin{aug}
\author{\fnms{Ben} \snm{Hambly}\thanksref{t1,m1}\ead[label=e1]{hambly@maths.ox.ac.uk}}
\and
\author{\fnms{Sean} \snm{Ledger}\thanksref{t2,m2}\ead[label=e2]{sean.ledger@bristol.ac.uk}}

\thankstext{t1}{First supporter of the project}
\thankstext{t2}{Supported by the Heilbronn Institute for Mathematical Research}

\affiliation{University of Oxford\thanksmark{m1} and Heilbronn Institute, University of Bristol\thanksmark{m2}}

\address{Mathematical Institute\\
University of Oxford\\
Woodstock Road\\
Oxford\\
OX2 6GG\\
\printead{e1}}

\address{Heilbronn Institute for Mathematical Research\\
University of Bristol\\
Howard House\\
Bristol\\
BS8 1SN\\
\printead{e2}}
\end{aug}

\begin{abstract}
\ We study a finite system of diffusions on the half-line, absorbed when they hit zero, with a correlation effect that is controlled 
by the proportion of the processes that have been absorbed. As the number of processes in the system becomes large, the 
empirical measure of the population converges to the solution of a non-linear stochastic heat equation with Dirichlet boundary
condition. The diffusion 
coefficients are allowed to have finitely many discontinuities (piecewise Lipschitz) and we prove pathwise uniqueness 
of solutions to the limiting stochastic PDE. As a corollary we obtain a representation of the limit as the unique solution 
to a stochastic McKean--Vlasov problem. Our techniques involve energy estimation in the dual of the first Sobolev space, 
which connects the regularity of solutions to their boundary behaviour, and tightness calculations in the Skorokhod M1 
topology defined for distribution-valued processes, which exploits the monotonicity of the loss process $L$. The 
motivation for this model comes from the analysis of large portfolio credit problems in finance.
\end{abstract}

\begin{keyword}[class=MSC]
\kwd{60H15}
\kwd{60H30}
\kwd{60F15}
\end{keyword}

\begin{keyword}
\kwd{McKean--Vlasov problem}
\kwd{non-linear SPDE}
\kwd{Skorokhod M1 topology}
\end{keyword}

\end{frontmatter}


\section{Introduction} 
\label{Sect_Intro}


\subsection*{Motivation and framework}

We prove the weak convergence of a system of interacting diffusions to the unique solution of a non-linear stochastic PDE 
on the half-line. In our model the diffusions are absorbed at the origin and the proportion of absorbed particles influences the 
diffusion coefficients, which leads to a description of the limiting system as the solution to a stochastic McKean--Vlasov problem. 
The motivation for studying the model in this paper is to extend the mathematical framework of \cite{bush11} for the pricing 
of large portfolio credit derivatives to include processes whose dynamics are driven by statistics of the entire population.  
With more complicated interaction terms, the methods in \cite{bush11} are no longer tractable and so we require new techniques.
In particular, it is very difficult to analyse the correlation between pairs of particles in our model (an essential ingredient of
\cite{bush11}) and, from a practical perspective, it is desirable to allow the coefficients of the diffusions to be discontinuous, 
which presents a further complication. 

Portfolio credit derivatives (such as the collateralised debt obligation --- CDO)  have a pay-off structure which depends on the 
total notional value of the loss due to default of entities in the portfolio across the lifetime of the product, after a process of 
partial asset recovery takes place. We will not explore the financial details of these contracts (see \cite{schonbucher2003credit}), 
but two important effects the modeller must capture are the intensity of defaults and the tendency for defaults to occur simultaneously. Common modelling approaches include copula-based models, in which the joint probability of default 
over a fixed time period is modelled directly, and reduced-form models, in which the default rates are modelled as correlated
stochastic processes. The model we will consider is a \emph{structural model}: default times are represented as the threshold 
hitting times of a collection of correlated stochastic processes. These models were introduced in the context of portfolio derivatives 
by \cite{hull2001} and \cite{zhou2001}, and their origins trace back to \cite{black1976} and \cite{merton1974} for single-name
derivatives. 

Our general framework is as follows. Suppose we have a collection of $N \geq 1$ defaultable entities and a fixed finite 
time horizon $T>0$. Assign the \ith\ entity a risk process, $X^{i,N}$, called the \emph{distance-to-default} process, with
$\{X^{i,N}_{0}\}_{1 \leq i \leq N}$ chosen to be positive independent random variables supported on $(0,\infty)$ with 
common law $\nu_{0}$. Default of the \ith\ entity is modelled as the first hitting time of zero of the distance-to-default process:
\begin{equation}
\label{Intro_Eq_Tau}
\tau^{i,N} := \inf\{ t > 0 : X^{i,N}_{t} \leq 0\}.
\end{equation}
The \emph{empirical} and \emph{loss processes} then track the spatial evolution of the surviving particles and the proportion 
of killed particles; defined respectively as
\begin{equation}
\label{Intro_Eq_NuL}
\nu^{N}_{t} := \frac{1}{N} \sum_{i=1}^{N} \1_{t < \tau^{i,N}} \delta_{X^{i,N}_{t}},
	\qquad\qquad
	L^{N}_{t} := \frac{1}{N} \sum_{i=1}^{N} \1_{ \tau^{i,N} \leq t}.
\end{equation}
Here, $\delta_{x}$ denotes the Dirac delta measure of the point $x \in \R$. The empirical process takes values in the 
sub-probability measures on \R\ and the loss process takes values in \R. For $S \subseteq \R$, $\nu^{N}_{t}(S)$ is simply 
the proportion of the diffusions that take values in $S$ at time $t$ that have not yet hit the origin by time $t$:
\[
\nu^{N}_{t}(S) = \frac{\#\{1 \leq i \leq N : X^{i,N}_{t} \in S \textrm{ and } t < \tau^{i,N}\}}{N},
\]
hence we have the relationship
\[
L^{N}_{t} = 1 - \nu^{N}_{t}(0,\infty).
\]

In practice, once the dynamics of $X^{i,N}$ have been specified, the model could be used to generate realisations of 
$L^{N}$ from which portfolio credit derivatives (options on $L^{N}$) could be priced using Monte Carlo routines. 
Instead, we will approximate $L^{N}$ by its limit as $N \to \infty$.  This is known as a large portfolio approximation, 
an idea first introduced in \cite{vasicek1991} and now found in several modern frameworks for copula-based models 
\cite{cherubini2004copula, frey2003, merino2002} and reduced-form models \cite{ding2009, errais2010, mortensen2006}. 
We will return to the question of how this approximation is generated in practice after a precise description of the limiting 
objects and mode of convergence.

\subsection*{Model specification}

We will model the processes $\{X^{i,N}\}_{1 \leq i \leq N}$ as correlated diffusions with parameters that are functions 
of the current proportional loss:
\begin{multline}
\label{eq:Intro_Coefficients}
X^{i,N}_{t} 
	= X^{i}_{0}
	+ \int^{t}_{0} \mu(s, X^{i,N}_{s}, L^{N}_{s}) ds
	+ \int^{t}_{0} \sigma(s, X^{i,N}_{s}) \rho(s, L^{N}_{s}) dW_{s}
	\\+ \int^{t}_{0} \sigma(s, X^{i,N}_{s}) (1 - \rho(s, L^{N}_{s})^{2})^{\frac{1}{2}} dW^{i}_{s}.
\end{multline}
Here, $W, W^{1}, W^{2}, \dots$ are independent standard Brownian motions and the precise conditions on the coefficients 
are given in Assumption~\ref{Not_Def_Coefficients}. In particular we assume $\rho$ is  piecewise Lipschitz with finitely many
discontinuities in the loss variable $\ell \mapsto \rho(s, \ell)$. (It is easy, but perhaps not immediate, to show that this collection 
of processes exists, see Remark \ref{Not_Rem_WellDefined}.)

In \cite{bush11} this model is analysed for the case when the coefficients are constants and it is shown that the sequence of 
empirical process, $(\nu^{N})_{N \geq 1}$, converges to a stochastic limit which can be characterised as the unique solution 
to a heat equation with constant coefficients and a random transport term driven by the systemic Brownian motion $W$ 
\cite[Thm.~1.1]{bush11}. However, numerical experiments show that the constant coefficient model is too simple to adequately
capture the traded prices of CDOs across all tranches simultaneously \cite[Sct.~5]{bush11}. This problem is common for 
Gaussian models --- the tails of the risk processes are too light to produce large losses and so a large  correlation parameter 
is required to generate scenarios in which many defaults occur over a given time horizon \cite{finger2005, schonbucher2003credit}.
Consequently, different products on the same underlying portfolio may produce different correlation parameters when calibrated 
to market prices. This phenomenon is known as \emph{correlation skew} (see Figure~\ref{Intro_Fig_Skew}). 

There is a large literature addressing the drawbacks of Gaussian credit models. Examples include the addition of jump processes 
and heavy-tailed distributions \cite{fang2010, lindskog2003, WuYang2013}, stochastic parameters and inhomogeneity 
\cite{andersen2005, burtschell2007} and contagion effects \cite{daipra2009, Giesecke2015, giesecke2006}. Close relatives to 
our framework include \cite{Bujok2012}, in which a jump process is added to the systemic factor, but in a discretised version of 
the system, and \cite{jin2010}, in which the particles are taken to be general diffusions. In \cite{Ahmad2016} the constant 
coefficient model is studied on the unit interval with absorbing boundaries at 0 and 1 and with an additional multiplicative killing 
rate as a model for mortgage pools. 

\begin{figure}
\begin{center}
\includegraphics[width=8cm]{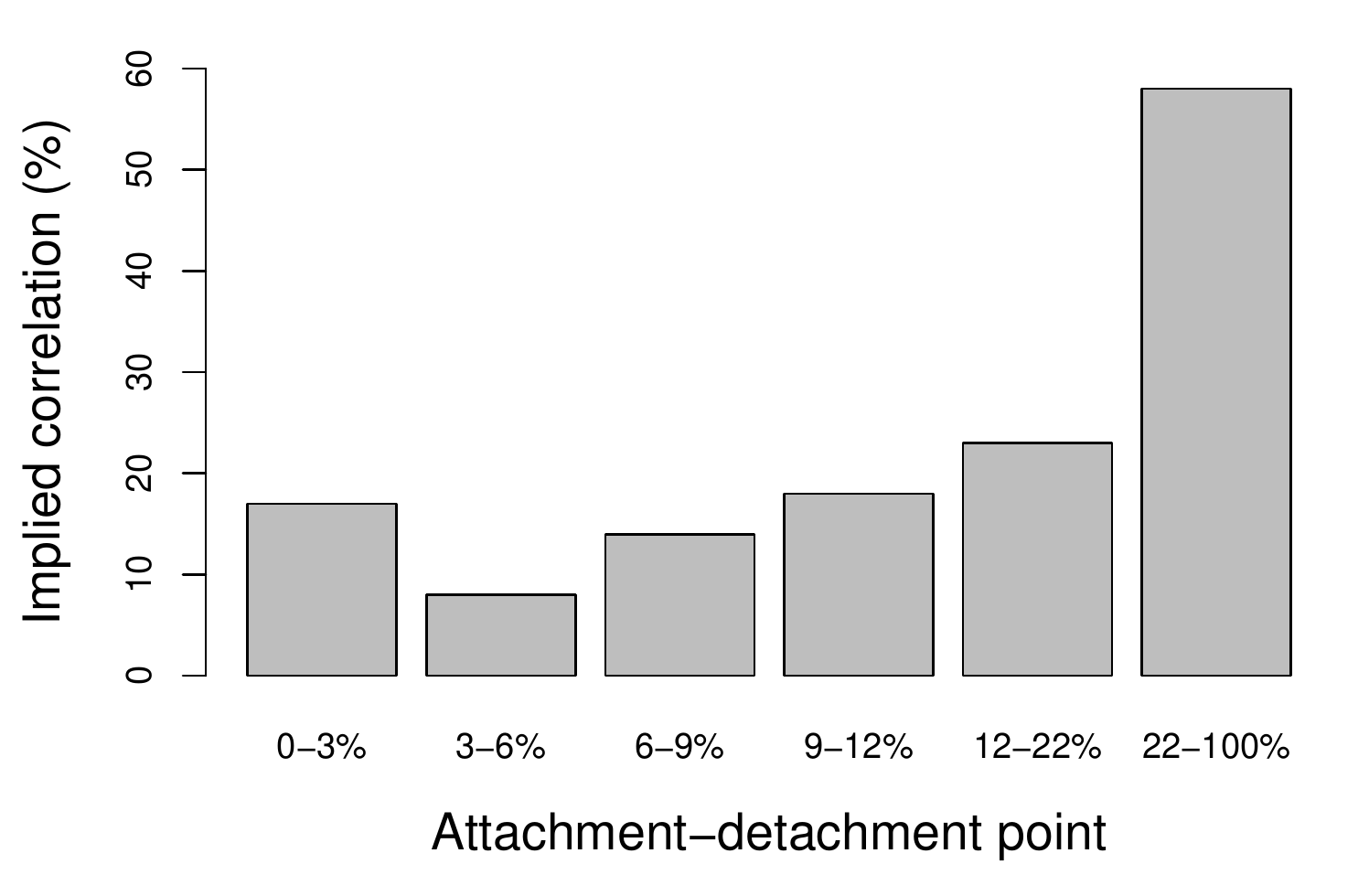}
\par\end{center}
\vspace{-0.4cm}
\caption{{\footnotesize  \label{Intro_Fig_Skew} Implied correlation for each tranche for the data set from 
\cite[Figure 2, 7 year maturity]{bush11}. With $\nu_{0}$, $\mu$ and $\sigma$ fixed in the constant coefficient model, 
the implied correlation for a given tranche is the value of the correlation parameter required to give a model spread 
equal to the market spread for that tranche. This is an example of correlation skew.}}
\end{figure}

Our present approach is inspired by Figure~\ref{Intro_Fig_Skew}. Suppose $\mu$ and $\sigma$ are fixed constants and $\rho$ is 
only a function of $\ell$. If $\ell \mapsto \rho(\ell)$ was piecewise constant across intervals corresponding to the CDO tranches 
in Figure~\ref{Intro_Fig_Skew}, then an obvious strategy for calibrating $\rho$ to the market prices is to calibrate the first level 
of $\rho$ to the traded spread of the most junior tranche, fix this value, repeat the calibration procedure for the next most junior
tranche spread and continue for all tranches. It is therefore a natural assumption to allow the diffusion coefficients in 
(\ref{eq:Intro_Coefficients}) to have finitely many discontinuities. Piecewise Lipschitz coefficients encompass this class of models 
whilst giving an analytically tractable system.

\subsection*{Main results}

The dynamics of an individual distance-to-default process, $X^{i,N}$, are controlled by the population behaviour, hence we 
have an example of a McKean--Vlasov system --- see \cite{sznitman1991} for an overview. Some applications of these systems
include the modelling of large collections of neurons and threshold hitting times for membrane potential levels in mathematical 
neuroscience \cite{deMasi2015, lucon2014}, the modelling of a large number of non-cooperative agents in mean-field games 
\cite{carmona2013, chassagneux2015}, filtering theory \cite{BainCrisan2009, crisan2010} and mathematical genetics 
\cite{dawson2014}. Examples in portfolio credit modelling include \cite{daipra2009, Spiliopoulos2014} in which systems with 
contagion effects are analysed under their large population limits.  

As $N \to \infty$, we will find that the influence of the idiosyncratic Brownian drivers, $W^{1}, W^{2}, \dots$, averages-out to 
a deterministic effect, but that the randomness due to the systemic Brownian motion, $W$, remains present. Hence the system 
must be characterised as the pair $(\nu^{N}, W)$, and we will follow an established strategy to demonstrate the  convergence 
in law of this pair and to characterise the limiting law:
\begin{enumerate}[(i)] 
\item Prove tightness of $(\nu^{N}, W)_{N \geq 1}$ (in a suitable topology), 
\item Characterise the limit points as weak solutions of a non-linear evolution equation,
\item Prove uniqueness of solutions for this equation,
\item Conclude all limiting laws agree, and hence that we have convergence in law.
\end{enumerate} 

The mathematical challenge comes from the interaction of the individuals through the boundary behaviour of the population 
and the discontinuities in the diffusion coefficients. A similar model has recently been studied where the particles interact 
through the quantiles of the empirical measure \cite{CrisanKurtzLee2014}, however there is no general uniqueness theory for 
this problem. For a model without systemic noise there is a uniqueness theory in \cite{kolokoltsov2013}. Discontinuous 
coefficients have been considered in \cite{chiang1994}, but only on the whole space and in the deterministic setting. In our 
model, parameter discontinuities are allowed because the limiting realisations of the loss process are strictly increasing 
(Proposition~\ref{Prob_Prop_LossInc}). This implies the infinite system spends a null set of time at points where the 
discontinuities in the coefficients prevent the application of the continuous mapping theorem (Corollary~\ref{Tight_Cor_Fiddly}).
Stochastic PDEs of McKean--Vlasov type are popular tools in the analysis of mean-field games with common noise 
\cite{carmona2015, kolokoltsov2015}. In \cite{dirt_annalsAP_2015, dirt_SPA_2015} a system of diffusions on the half-line is 
studied in which each particle undergoes a proportional jump towards zero whenever any of the particles hits the absorbing 
boundary at zero. The purpose of the model is to describe the self-excitatory behaviour of a large collection of neurons. 
For small values of the feedback parameter, existence and uniqueness theorems hold for the limiting system. It is shown 
in \cite{carrillo2011}, however, that for large values of the feedback parameter the limiting system must blow-up (in the 
sense that no continuous solutions exist) and a complete existence and uniqueness theory in this case remains a challenge. 

The topology we will use for establishing tightness of the sequence of laws of $(\nu^{N}, W)_{N \geq 1}$ is the product topology 
$(D_{\Sdual}, \M) \times (C_{\R}, \mathrm{U})$, where $(D_{\Sdual}, \M)$ is the \M\ topological space of distribution-valued 
\cadlag\ processes on $[0,T]$, introduced in \cite{ledger2015}, and $(C_{\R}, \mathrm{U})$ is the space of real-valued 
continuous functions on $[0,T]$ with the topology of uniform convergence. (Throughout, \S\ denotes the space of rapidly 
decreasing functions and \Sdual\ the space of tempered distributions.) It will not be necessary to explain the full details of 
the construction of $(D_{\Sdual}, \M)$, as the proof Theorem~\ref{Intro_Thm_Exist} uses only Theorem~3.2 and 
Proposition~2.7 of \cite{ledger2015}, together with facts about the classical M1 toplogy on $D_{\R}$. The \M\ topology is 
helpful because monotone real-valued processes are automatically tight in $(D_{\R}, \M)$, a fact which has been exploited in 
many other applications (see \cite{ledger2015} for references). In our infinite-dimensional setting, the decomposition trick in 
\cite[Prop.~4.2]{ledger2015} enables us to exploit the monotonicity of the loss process in proving tightness of the empirical 
process. Tightness on the product space implies the existence of subsequential limit points, whereby we recover:

\begin{thm}[Existence]
\label{Intro_Thm_Exist}
Let $(\nu, W)$ realise a limiting law of the sequence $(\nu^{N}, W)_{N \geq 1}$. Then $\nu$ is a continuous process taking values 
in the sub-probability measures and satisfies the regularity conditions of Assumption~\ref{Not_Def_RegCond} and the 
\emph{limit SPDE}:
\begin{align*}
\nu_{t}(\phi) 
	&= \nu_{0}(\phi)
	+ \int^{t}_{0} \nu_{s}( \mu(s, \cdot, L_{s}) \dx \phi ) ds
	+ \frac{1}{2} \int^{t}_{0} \nu_{s}( \sigma^{2}(s, \cdot) \dxx \phi ) ds \\  
	&\quad + \int_{0}^{t} \nu_{s}( \sigma(s, \cdot)\rho(s, L_{s}) \dx \phi ) dW_{s},
	\qquad \textrm{with \ \ } L_{t} = 1 - \nu_{t}(\1_{(0,\infty)}),
\end{align*}
for every $t \in [0,T]$ and $\phi \in \C := \{ \phi \in \S : \phi(0) = 0 \}$, with probability 1. Furthermore, if the limit point is 
attained along the subsequence $(\nu^{N_{k}}, W)_{k \geq 1}$, then $(L^{N_{k}}, W)_{k \geq 1}$ converges in law to 
$(L, W)$ on the product space $(D_{\R}, \M) \times (C_{\R}, \mathrm{U})$.
\end{thm}

\begin{figure}
\vspace{0.4cm}
\begin{center}
\includegraphics[bb=0bp 50bp 576bp 320bp,clip,width=11cm]{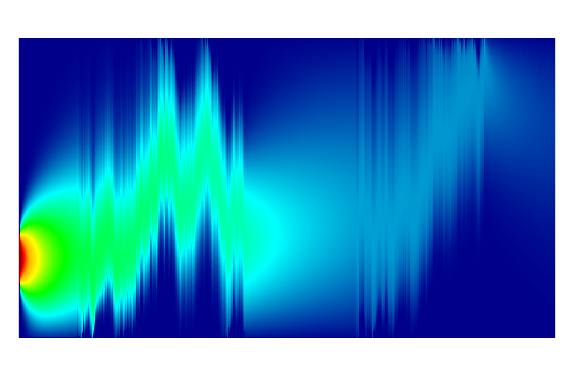}
\par\end{center}
\vspace{-0cm}
{\scriptsize 
\[
\qquad\; L^{-1}(1/5)\;\quad\qquad\quad L^{-1}(2/5)\qquad\quad L^{-1}(3/5)\qquad\quad L^{-1}(4/5)\qquad
\]
}{\scriptsize \par}
{\footnotesize 
\[
\rho\left(t, \ell\right)=\begin{cases}
0 & \textrm{if }\ell\in\left[0,\frac{1}{5}\right)\cup\left[\frac{2}{5},\frac{3}{5}\right)\cup\left[\frac{4}{5},1\right]\\
\frac{9}{10} & \textrm{if }\ell\in\left[\frac{1}{5},\frac{2}{5}\right)\cup\left[\frac{3}{5},\frac{4}{5}\right).
\end{cases}
\]
}{\footnotesize \par}
\vspace{-0.2cm}
\caption{{\footnotesize \label{Intro_Fig_LossDepHeat} Heat plot for the solution, $\nu$, of the limit SPDE for a fixed sample 
path of $W$. Time is plotted on the horizontal axis, space on the vertical axis and the value of a pixel represents the (scaled) 
intensity of $\nu$ at that space-time point (blue for level zero increasing to dark red for maximal value). The initial condition is 
a step function, $\mu = 0$, $\sigma = 1$ and $\rho$ is given above. Markers are added to show the times at which the loss 
process, $L$, reaches levels $1/5$, $2/5$, $3/5$ and $4/5$. Notice the corresponding three periods of smooth heat flow 
between the two periods of highly correlated motion. (Figure produced using the algorithm outlined in Section~\ref{Sect_Open}.)}}
\end{figure}

The limit SPDE is a non-linear heat equation with stochastic transport term driven by the systemic Brownian motion (see 
Figure~\ref{Intro_Fig_LossDepHeat} for an example with an exaggerated correlation change), and the space of test functions, 
\C, encodes the Dirichlet boundary conditions. In the limit, the idiosyncratic noise averages-out to produce the diffusive evolution 
equation. The intuition for this effect is explained easily in Section~\ref{Sect_Dynamics}, however a full proof of 
Theorem~\ref{Intro_Thm_Exist} requires more technical details and is given in Section~\ref{Sect_Tightness}. 
Several estimates involving purely probabilistic arguments are presented in Section~\ref{Sect_ProbEstimates}, where a key 
result is Proposition~\ref{Prob_Prop_LossInc} which shows (in an asymptotic sense) that over any non-zero time interval the 
system must lose a non-zero proportion of mass, and hence any limiting loss process is strictly increasing. 

With Theorem~\ref{Intro_Thm_Exist} established, demonstrating the full weak convergence of $(\nu^{N}, W)_{N \geq 1}$ 
is a matter of proving uniqueness of solutions to the limit SPDE: 

\begin{thm}[Uniqueness/Law of large numbers]
\label{Intro_Thm_Unique}
Let $\nu_{0}$ satisfy Assumption~\ref{Not_Def_Coefficients}. Suppose that $(\nu, W)$ realises a limiting law of 
$(\nu^{N}, W)_{N \geq 1}$ and that $\tilde{\nu}$ satisfies Assumption~\ref{Not_Def_RegCond}.  
If $\nu$ and $\tilde{\nu}$ solve the limit SPDE in Theorem~\ref{Intro_Thm_Exist} with respect to $W$ and $\nu_{0}$, 
then with probability 1
\[
\nu_{t}(S) = \tilde{\nu}_{t}(S),
	\qquad
	\textrm{for every } t \in [0,t] \textrm{ and Borel measurable } S \subseteq \R.
\]
Hence there exists a unique law of a solution to the limit SPDE on $(D_{\Sdual}, \M) \times (C_{\R}, \mathrm{U})$ and 
$(\nu^{N}, W)_{N \geq 1}$ converges weakly to this law. Furthermore, if $(\nu, W)$ realises the unique law, then 
$(L^{N}, W)_{N \geq 1}$ converges in law to $(L, W)$ on $(D_{\R}, \M) \times (C_{\R}, \mathrm{U})$, where 
$L_{t} = 1 - \nu_{t}(0, \infty)$.
\end{thm}

\begin{rem}[Strong solutions]
\label{Intro_Rem_StrongSolutions}
Theorem~\ref{Intro_Thm_Unique} shows that all weak solutions realise limiting laws, and amongst limiting laws we have 
pathwise uniqueness. Following~\cite[Cor.~5.3.23]{karatzas1991brownian}, we deduce that strong solutions exist on a 
sufficiently rich probability space, whereby $\nu$ (and hence $L$) is adapted to the filtration generated by $W$. 
\end{rem}

\begin{rem} (Density) In Corollary~\ref{Unique_Cor_L2reg} we show that $\nu$ has a density process $V_t\in L^2(0,\infty)$ 
such that $\nu_t(\phi) = \int_0^{\infty} \phi(x) V_t(x) dx$ for all $\phi \in L^{2}(0,\infty)$ and $t \in [0,T]$. It is then instructive 
to write the limit SPDE formally as
\begin{align*}
V_t(x) 
 	&=  V_0(x) - \int_{0}^{t} \partial_x(\mu(s,\cdot,L_s)V_s(\cdot)) ds 
 	+ \frac{1}{2} \int_0^t \partial_{xx} (\sigma^2(s,\cdot) V_s(\cdot)) ds \\ 
	&\qquad- \int_0^t \rho(s,L_s) \partial_x (\sigma(s,\cdot)V_s(\cdot)) dW_s,
	\qquad\qquad \textrm{with } V_t(0)=0.
\end{align*}
\end{rem}

To prove Theorem~\ref{Intro_Thm_Unique} (Section~\ref{Sect_Uniqueness}) we use the kernel smoothing method from
\cite{bush11}, which is a technique for mollifying potentially exotic solutions to the limit SPDE in order to work with smooth 
tractable objects, at the expense of a small approximation error. The technique was used on the whole space 
in \cite{kotelenez1995, KurtzXiong1999}. In \cite{bush11} the approximation error is controlled in the space 
$L^{2}(0,\infty)$ and there the key quantity to control is the second moment of the mass near the origin: $\E \nu_{t}(0,\e)^{2}$, 
for a candidate solution $\nu$. This approach succeeds because the quantity can be written in terms of the law of a two-dimensional
Brownian motion in a wedge, for which explicit formulae are available. In that case the kernel smoothing method can be used to 
give a precise description of the regularity of the solution \cite{ledger2014}. As the particle interactions in our model are more
complicated, however, these explicit formula are no longer available. Although we are able to show that the unique solution to 
the limit SPDE has a density in $L^{2}$ (Corollary~\ref{Unique_Cor_L2reg}), which is an auxiliary result towards 
Theorem~\ref{Intro_Thm_Unique}, that method cannot be used to fully establish uniqueness as it relies on a crude upper bound 
for $\nu$ which neglects the effect of the absorbing boundary (Remark~\ref{Unique_Rem_Crude}). Our solution to this problem is 
to adapt the kernel smoothing method to the dual of the first Sobolev space, which then only requires us to control the first 
moment $\E \nu_{t}(0,\e)$ (Section~\ref{Sect_Kernel}). This is an easier quantity to estimate as only individual particles need 
to be studied and not pairs of particles, hence we do not need to consider the complicated correlation between particles 
(see Propositions~\ref{Prob_Prop_Boundary} and~\ref{Tight_Prop_RegCond}).

We must also deal with discontinuities in the coefficients of the limit SPDE and here the strict monotonicity of the limiting 
loss processes is again important. Our strategy is to prove uniqueness up to the first time the level of the loss reaches a 
discontinuity point of the coefficients, whereby continuity allows us to propagate the argument onto the next such time 
interval. With a strictly increasing loss process and only finitely many discontinuities, this argument terminates after finitely 
many iterations, whereby we have uniqueness on the whole time horizon $[0,T]$.

\begin{rem}[Pathological $\rho$]
\label{Ch5_Remark_BadRho}
We cannot choose $\rho$ arbitrarily and expect Theorem~\ref{Intro_Thm_Unique} to hold. As an example, 
let $\mu = 0$, $\sigma = 1$ and 
\[
\rho(t,\ell)
	= \begin{cases}
q^{-1}, & \textrm{if }\ell=kq^{-n}\textrm{ for some prime } q \textrm{, }n\in\mathbb{N}
  \textrm{ and }1\!\leq\! k\!\leq\! q^{n}\!-\!1\\
0, & \textrm{otherwise}.
\end{cases}
\]
For $N=q^{n}$, $L^{N}$ is supported on $\left\{ k q^{-n}\right\} _{0\leq k\leq q^{n}}$, hence $\nu^{N}$ behaves as the 
basic constant correlation system with $\rho=q^{-1}$, which we denote $\nu|_{\rho = q^{-1}}$. Therefore 
$(\nu^{q^{n}})_{n \geq 1}$ converges weakly to $\nu|_{\rho = q^{-1}}$ as $n \to \infty$, hence there is a distinct limit 
point for every prime, so weak convergence fails for this example.
\end{rem}

In Section~\ref{Sect_MV} we recast our results as a stochastic McKean--Vlasov problem (with randomness from $W$) and 
this shows that $\nu$ can be written as the conditional law of a single tagged particle:

\begin{thm}[Stochastic McKean--Vlasov problem]
\label{Intro_Thm_MV}
Let $(\nu, W)$ be a strong solution to the limit SPDE (Remark~\ref{Intro_Rem_StrongSolutions}). For any independent 
Brownian motion, $W^{\perp}$, there exists a continuous real-valued process, $X$, satisfying
\[
\begin{cases}
X_{t} 
	= X_{0}
	+ \int_{0}^{t} \mu(s,X_{s}, L_{s}) ds
	+ \int_{0}^{t} \sigma(s,X_{s})\rho(s,L_{s}) dW_{s} \\
	\qquad \qquad \qquad + \int_{0}^{t} \sigma(s,X_{s})(1-\rho(s,L_{s})^{2})^{\frac{1}{2}} dW^{\perp}_{s}, \\
\tau=\inf \{ t>0 : X_{t} \leq 0 \},\\
\nu_{t}(\phi) 
	= \mathbf{E} [ \phi (X_{t}) \mathbf{1}_{t<\tau}|W], \\
L_{t} = \P(\tau \leq t | W).
\end{cases}
\]
(Here, $X_{0}$ has law $\nu_{0}$ and is independent of all other random variables.) Furthermore, the law of 
$\left(X,W\right)$ is unique.
\end{thm}
 
Returning to the question of applying our model, regarding a portfolio credit derivative as an option on the loss process, $L$, 
with some payoff function,  $\Psi : D_{\R} \to \R$, the main practical question is how to accurately estimate $\E\Psi(L)$. This 
comes in two parts: we must first generate an approximation to $L$ (through $\nu$) to a given level of precision for a fixed 
Brownian trajectory and then we must combine such estimates into a random sample. In Section~\ref{Sect_Open} we give 
an outline of a discrete-time algorithm for approximating the system and some potential variance reduction techniques. We 
leave the tasks of checking the benefits and correctness of these methods as open problems. A number of potential 
modifications to the model are also stated, along with their corresponding mathematical challenges.

\subsection*{Overview}

In Section~\ref{Sect_Not} we state the main technical assumptions on the model parameters and review their purpose. 
In Section~\ref{Sect_Dynamics} we derive the evolution equation satisfied by the empirical measure of the finite system,
which gives a heuristic explanation for arriving at the limit SPDE in Theorem~\ref{Intro_Thm_Exist}. In 
Section~\ref{Sect_ProbEstimates} several probabilistic estimates are derived for the finite system and these are applied in 
Section~\ref{Sect_Tightness} to give a proof of Theorem~\ref{Intro_Thm_Exist}. In Section~\ref{Sect_Kernel} we describe 
the kernel smoothing method, which is the main tool for the proof of Theorem~\ref{Intro_Thm_Unique} in 
Section~\ref{Sect_Uniqueness}. In Section~\ref{Sect_Lemmas} several technical lemmas are presented which are used to 
in Section~\ref{Sect_Uniqueness}, but which are deferred for readability. In Section~\ref{Sect_MV} we use our results to 
give a short proof of Theorem~\ref{Intro_Thm_MV}. In Section~\ref{Sect_Open} we outline an algorithm for simulating the 
solution to the limit SPDE and discuss open problems relating to this and to potential model extensions. 


\section{Notation and assumptions}
\label{Sect_Not}

The purpose of this section is to lay out the technical definitions omitted in the introduction and to explain their purpose. 

\begin{ass}[Coefficient assumptions]
\label{Not_Def_Coefficients}
Let $\mu : [0,T] \times \R \times [0,1] \to \R$, $\sigma : [0,T] \times \R \to [0,\infty)$ and $\rho : [0,T] \times [0,1] \to [0,1)$ be 
the coefficients in (\ref{eq:Intro_Coefficients}) and $\nu_{0}$ be the common law of the initial values of the distance-to-default
processes introduced above (\ref{Intro_Eq_Tau}). We assume that we have a sufficient large constant, $C \in (1, \infty)$, such 
that all the following hold:
\begin{enumerate}[(i)]

\item (Initial condition) \label{Not_Def_Coefficients_1} The probability measure $\nu_{0}$ is supported on $(0,\infty)$, 
has a density $V_{0} \in L^{2}(0,\infty)$ and satisfies
\[
\nu_{0}(\lambda, \infty) = o(\exp\{-\alpha \lambda\}),
	\qquad \textrm{as } \lambda \to +\infty
\]
for every $\alpha > 0$. (Note: $V_{0} \in L^{2}(0,\infty)$ implies $\nu_{0}(0, \e) = O(\e^{1/2}) = o(1)$ as $\e \to 0$.)

\item \label{Not_Def_Coefficients_Smooth} (Spatial regularity) For all fixed $t \in [0,T]$ and 
$\ell \in [0,1]$, $\mu(t, \cdot, \ell), \sigma(t, \cdot) \in C^{2}(\R)$ with 
\[
|\dx^{n} \mu(t,x,\ell)|, |\dx^{n} \sigma(t,x)|
	\leq C
\]
for all $t \in [0,T]$, $x \in \R$, $\ell \in  [0,1]$ and $n = 0, 1, 2$,

\item (Non-degeneracy)  \label{Not_Def_Coefficients_3} For all  $t \in [0,T]$, $x \in \R$, $\ell \in  [0,1]$
\[
\sigma(t, x) \geq C^{-1} > 0,
	\qquad
	0 \leq \rho(t, \ell) \leq 1 - C^{-1} < 1,
\]

\item (Piecewise Lipschitz in loss)  \label{Not_Def_Coefficients_4} There exists 
$0 = \theta_{0} < \theta_{1} < \cdots < \theta_{k} = 1$ such that
\[
|\mu(t,x,\ell) - \mu(t,x,\bar{\ell})|,
	|\rho(t, \ell) - \rho(t, \bar{\ell}) |
	\leq C |\ell - \bar{\ell}|,
\]
whenever $t \in [0,T]$, $x \in \R$ and both $\ell,\bar{\ell} \in [\theta_{i - 1}, \theta_{i})$ for some $i \in \{1,2,\dots,k\}$,

\item (Integral constraint)  \label{Not_Def_Coefficients_5} $\sup_{s \in [0,T]} \int^{\infty}_{0} |\partial_{t} \sigma(s, y)| dy 
< \infty$.
\end{enumerate}
\end{ass}

\begin{rem}[$X^{i,N}$ well defined]
\label{Not_Rem_WellDefined}
To see that we can find $\{ X^{i,N} \}_{1 \leq i \leq N}$ satisfying (\ref{eq:Intro_Coefficients}) notice that initially $L = 0$, 
so we can find $N$ diffusions satisfying (\ref{eq:Intro_Coefficients}) up to the first time one of the diffusions hits the origin 
(i.e.\ with coefficients of the form $g(t,x,0)$) --- notice that the coefficients are globally Lipschitz by
(\ref{Not_Def_Coefficients_Smooth}) of Assumption~\ref{Not_Def_Coefficients}, so standard diffusion theory applies. At this 
stopping time $L^{N} = 1/N$, and so the process can be restarted as a diffusion with coefficients $g(t,x,1/N)$. This gives 
a solution up to the first time two particles have hit the origin. Repeating this argument gives the construction of 
$\{X^{i,N}\}_{1 \leq i \leq N}$. 
\end{rem}

Condition~(\ref{Not_Def_Coefficients_1}) ensures that limiting realisations of the system satisfy the regularity conditions in
Assumption~\ref{Not_Def_RegCond}, as required for Theorem~\ref{Intro_Thm_Exist}. The tail assumption and boundary 
behaviour of $\nu_{0}$ are used in Proposition~\ref{Prob_Prop_Boundary} and~\ref{Prob_Prop_Tail} to show that $\nu^{N}$ 
inherits the corresponding properties at times $t > 0$, and this is transferred to limit points by 
Proposition~\ref{Tight_Prop_RegCond}. 

The boundedness assumption on the coefficients, given by the case $n = 0$ in condition~(\ref{Not_Def_Coefficients_Smooth}), 
is used many times throughout this paper. The cases $n = 1$ and $2$ are used in Lemma~\ref{Prob_Lem_Scale} 
and~\ref{Prob_Lem_Remove} to relate the law of $X^{1, N}$ to that of a standard Brownian motion, and in 
Lemma~\ref{Lemmas_Lem_Switch1stOrder} and~\ref{Lemmas_Lem_Switch2ndOrder} to interchange coefficients and measures 
in the proof of Theorem~\ref{Intro_Thm_Unique}. 

Condition~(\ref{Not_Def_Coefficients_3}) implies that there is always a diffusive effect acting on the system, and this ensures 
that the limiting system does not become degenerate. If $\sigma = 0$ or $\rho = 1$ then the particles are completely dependent 
and move according to a drift term given by $\mu$ and $W$. The assumption that $\rho$ is bounded away from 1 is used directly 
in the proof of Theorem~\ref{Intro_Thm_Unique} in~(\ref{eq:Unique_RhoLess1}) and~(\ref{eq:Unique_Anti5}). 

Condition (\ref{Not_Def_Coefficients_4}) is the main motivating assumption, which we have discussed at length in 
Section~\ref{Sect_Intro}.

Condition (\ref{Not_Def_Coefficients_5}) is purely a technical assumption to ensure that the drift coefficient, $D$, in 
Lemma~\ref{Prob_Lem_Scale} is uniformly bounded by a deterministic constant.

Finally, we will remark on the specific form of $\sigma = \sigma(t, x)$ and $\rho = \rho(t, \ell)$. From (\ref{eq:Intro_Coefficients}) 
we can write the dynamics of a single particle as
\[
dX^{i,N}_{t} 
	= \mu(t, X^{i,N}_{t}, L^{N}_{t}) dt
	+ \sigma(t, X^{i,N}_{t}) dB^{i}_{t},
\]
where $B^{i}$ is a Brownian motion. Although the $\{B^{i}\}_{i}$ are coupled through $L^{N}$, this representation allows 
us to relate the law of an \emph{individual} particle to a standard Brownian motion as in Lemmas~\ref{Prob_Lem_Scale} 
and~\ref{Prob_Lem_Remove}, since $\mu$ is bounded and $\sigma$ is independent of $L^{N}$. A second advantage of the 
taking $\sigma$ and $\rho$ in this form is that the pairwise correlation between particles is purely a function of 
$\rho(t, L^{N}_{t})$, and so is the same for all pairs. This is explicitly made use of in the construction of the time-change defined 
in (\ref{eq:Def_v}), and there it is again important that the correlation function is bounded strictly away from 1, so that the 
system can be compared to a standard multi-dimensional Brownian motion.
%

Below are the constraints we place on solutions to the limit SPDE in Theorem~\ref{Intro_Thm_Unique} to ensure that we 
have uniqueness. As Theorem~\ref{Intro_Thm_Exist} indicates, these conditions are natural in the sense that all limit points 
of the finite system satisfy them. 

\begin{ass}[Regularity conditions]
\label{Not_Def_RegCond}
Let $\nu$ be a \cadlag\ process taking values in the space of sub-probability measures on \R. The regularity conditions on $\nu$ are
\begin{enumerate}[(i)]

\item \label{Not_Def_RegCond1} (Loss function) The process defined by $L_{t} := 1 - \nu_{t}(0,\infty)$ is non-decreasing 
at all times and is strictly increasing when $L_{t} < 1$, 

\item \label{Not_Def_RegCond2} (Support) For every $t \in [0,T]$, $\nu_{t}$ is supported on $[0,\infty)$,

\item \label{Not_Def_RegCond3} (Exponential tails) For every $\alpha > 0$
\[
\E \int^{T}_{0} \nu_{t}(\lambda, +\infty) dt = o(e^{-\alpha \lambda}),
	\qquad \textrm{as } \lambda \to \infty,
	\
\]

\item \label{Not_Def_RegCond4} (Boundary decay) There exists $\beta > 0$ such that

\[
\E \int^{T}_{0} \nu_{t}(0, \e) dt = O(\e^{1 + \beta}),
	\qquad \textrm{as } \e \to 0,
\]

\item \label{Not_Def_RegCond5} (Spatial concentration) There exists $C > 0$ and $\delta > 0$ such that
\[
\E \int^{T}_{0} |\nu_{t}(a,b)|^{2} dt \leq C |b-a|^{\delta},
	\qquad \textrm{for all } a < b.
\]
\end{enumerate}
\end{ass}

It is essential that limit points satisfy condition~(\ref{Not_Def_RegCond1}) in order to apply the continuous mapping theorem 
to recover the limit SPDE for limit points (Corollary~\ref{Tight_Cor_Fiddly}). There, strict monotonicity ensures that there are 
only finitely many $t$ such that $L_{t} = \theta_{i}$ for some $i$, and hence that this set of times is negligible in the limit. 
Knowing that $L$ is monotone also allows us to split $[0,T]$ into consecutive intervals such that in the $i^{\textrm{th}}$ 
interval $L_{t} \in [\theta_{i}, \theta_{i +1})$, and this argument is used in the uniqueness proof in 
Section~\ref{Sect_Uniqueness}~(Case 2). 

Condition~(\ref{Not_Def_RegCond2}) is natural since $\nu^{N}$ is supported on $[0,\infty)$ by construction. However, it is 
also convenient to take our test functions, \C, to be supported on \R, hence (\ref{Not_Def_RegCond2}) is needed to rule 
out pathological solutions that have support on the negative half-line and that would otherwise break the uniqueness claim. 

Condition~(\ref{Not_Def_RegCond3}) is used several times throughout Section~\ref{Sect_Lemmas} to check various 
integrability requirements. It is also used in Lemma~\ref{Lemmas_Lem_LossInH1} to relate $\nu$ and $L$ via the $H^{-1}$ norm. 

Condition~(\ref{Not_Def_RegCond4}) is the key boundary estimate discussed in Section~\ref{Sect_Intro}. Its main use is 
in Lemma~\ref{Unique_Lem_Boundary}. 

Condition~(\ref{Not_Def_RegCond5}) guarantees that solutions cannot become too concentrated in spatial locations. 
This is used to interchange coefficients and measures in Lemma~\ref{Lemmas_Lem_Switch1stOrder} 
and~\ref{Lemmas_Lem_Switch2ndOrder}.


\section{Dynamics of the finite particle system}
\label{Sect_Dynamics}

This section introduces the empirical process approximation to the limit SPDE from Theorem~\ref{Intro_Thm_Exist} and 
explains the intuition behind the convergence of $(\nu^{N})_{N \geq 1}$. Throughout, we will drop the dependence of 
the coefficients on the time, space and loss variables and use the following short-hand when it is safe to do so:

\begin{rem}[Short-hand notation]
\label{Finite_Rem_ShortHand}
For fixed $N$, when there is no confusion, we may use the functional notation
\[
\mu_{t} = \mu(t, \cdot \, , L^{N}_{t}),
	\qquad 
	\sigma_{t} = \sigma(t, \cdot ),
	\qquad 
	\rho_{t} = \rho(t, L^{N}_{t}).
\]
\end{rem}

\begin{prop}[Finite evolution equation]
\label{Finite_Prop_Evolution}
For every $N \geq 1$, $t \in \timeInt$ and $\phi \in \C$
\begin{align*}
\nu^{N}_{t}(\phi)
  &= \nu^{N}_{0}(\phi)
  + \int_{0}^{t} \nu^{N}_{s}(\mu_{s} \dx\phi) ds
  + \frac{1}{2} \int_{0}^{t} \nu^{N}_{s}(\sigma^{2}_{s} \dxx \phi) ds \\
  &\qquad+ \int_{0}^{t} \nu^{N}_{s}(\sigma_{s} \rho_{s} \dx\phi) dW_{s}
  + I^{N}_{t}(\phi),
\end{align*}
where we have the \emph{idiosyncratic driver}
\[
I^{N}_{t}(\phi) 
  := \frac{1}{N} \sum_{i=1}^{N} \int_{0}^{t} \sigma(s, X^{i,N}_{s}) (1 - \rho(s ,L^{N}_{s})^{2})^{\frac{1}{2}} \dx 
  \phi(X^{i, N}_{s}) \1_{s < \tau^{i,N}} dW^{i}_{s}.
\]
\end{prop}

\begin{proof}
Apply It\^o's formula to $\phi(X^{i,N})$ to obtain
\begin{align*}
\phi(X^{i,N}_{t \wedge \tau^{i,N}})
  &= \phi(X^{i,N}_{0}) 
  + \int_{0}^{t} (\mu_{s} \dx \phi)(X^{i,N}_{s}) \1_{s < \tau^{i,N}} ds \\
  &\qquad + \frac{1}{2} \int_{0}^{t} (\sigma_{s}^{2} \dxx \phi) (X^{i,N}_{s}) 1_{s < \tau^{i,N}} ds   \\
   &\qquad + \int_{0}^{t} (\sigma_{s} \rho_{s} \dx \phi)(X^{i,N}_{s}) \1_{s < \tau^{i,N}} dW_{s}  \\
  &\qquad + \int_{0}^{t} (\sigma_{s} (1-\rho_{s}^{2})^{\frac{1}{2}} \dx \phi) (X^{i,N}_{s}) \1_{s < \tau^{i,N}} dW^{i}_{s}.
\end{align*}
If $\phi \in \C$, then 
\begin{equation}
\label{eq:WhyCtest}
\phi(X^{i,N}_{t \wedge \tau^{i,N}}) = \phi(X^{i,N}_{t}) \1_{t < \tau^{i,N}}\end{equation}
Substituting this expression into the left-hand side above, summing over $i \in \{1,2, \dots N \}$ and multiplying by $N^{-1}$ 
gives the result.
\end{proof}

\begin{rem}
We need to ensure that our test functions satisfy $\phi(0) = 0$ so that equation (\ref{eq:WhyCtest}) is valid. 
\end{rem}

Since the idiosyncratic noise, $I^{N}$, is a sum of martingales with zero covariation, the process converges to zero in 
the limit as $N \rightarrow \infty$. This explains why we arrive at the limit SPDE in Theorem~\ref{Intro_Thm_Exist}. 

\begin{prop}[Vanishing idiosyncratic noise]
\label{Finite_Prop_Vanish}
For every $\phi \in \C$
\[
\E \sup_{t \in \timeInt} |I^{N}_{t}(\phi)|^{2} 
  = \Vert \dx \phi \Vert_{\infty}^{2} \cdot O(N^{-1}),
  \qquad \textrm{as } N \to \infty.
\]
\end{prop}

\begin{proof}
Since $\sigma$ and $\dx \phi$ are bounded, the result follows from Doob's martingale inequality and the fact that
\[
[I^{N}_{\,\cdot}(\phi)]_{t}
  = \frac{1}{N^{2}} \sum_{i=1}^{N} \int_{0}^{t} \sigma(s, X^{i,N}_{s})^{2}(1 - \rho(s,L^{N}_{s})^{2}) \dx 
  \phi( X^{i,N}_{s})^{2}ds.
\]
\end{proof}

\subsection*{The whole space process}

In the proceeding sections it will be useful to work with the process defined by 
\begin{equation}
\label{eq:Finite_nub}
\nub^{N}_{t} := \frac{1}{N} \sum_{i=1}^{N} \delta_{X^{i,N}_{t}},
\end{equation}
which is a probability-measure valued processes on the whole of \R. Clearly it is the case that 
\begin{equation}
\label{eq:Finite_SimpleBound}
\nu^{N}_{t}(S) \leq \nub^{N}_{t}(S),
  \qquad \textrm{for all } N \geq 1, t \in \timeInt \textrm{ and } S \subseteq \R. 
\end{equation}

Since $\nub^{N}$ is not affected by the absorbing boundary, from the work in Proposition~\ref{Finite_Prop_Evolution} it 
follows that $\nub^{N}$ satisfies the same evolution equation as $\nu^{N}$, but on the whole space. This is encoded through 
the test functions:

\begin{prop}[Evolution of $\nub^{N}$]
\label{Finite_Prop_EvoNub}
For every $N \geq 1$, $t \in [0,T]$ and $\phi \in \S$
\begin{align*}
\nub^{N}_{t}(\phi)
  &= \nu^{N}_{0}(\phi)
  + \int_{0}^{t} \nub^{N}_{s}(\mu_{s} \dx\phi) ds
  + \frac{1}{2} \int_{0}^{t} \nub^{N}_{s}(\sigma^{2}_{s} \dxx \phi) ds \\
  & \qquad + \int_{0}^{t} \nub^{N}_{s}(\sigma_{s}\rho_{s} \dx\phi) dW_{s}
  + \bar{I}^{N}_{t}(\phi),
\end{align*}
where 
\[
\bar{I}^{N}_{t}(\phi) 
  := \frac{1}{N} \sum_{i=1}^{N} \int_{0}^{t} \sigma(s, X^{i,N}_{s}) (1 - \rho(s, L^{N}_{s})^{2})^{\frac{1}{2}} \dx 
  \phi(X^{i, N}_{s}) dW^{i}_{s}.
\]
\end{prop}


\section{Probabilistic estimates}
\label{Sect_ProbEstimates}

Here we collect the main probabilistic estimates used in later proofs. The reader may wish to skip this section and use it only 
as a reference.  We begin by noting the following simple result, which is just a consequence of the fact that $\{X^{i,N}\}_{i}$ 
are identically distributed: for any measurable $S \subseteq \R$, $N \geq 1$ and $t \in [0,T]$
\begin{equation}
\label{eq:Prob_Linear}
\E \nu^{N}_{t}(S)
	= \frac{1}{N} \sum_{i=1}^{N} \E [\1_{X^{i,N}_{t} \in S; t < \tau^{i,N}}] 
	= \P(X^{1,N}_{t} \in S; t < \tau^{1,N}).
\end{equation}
Under \P, $X^{1,N}$ is a diffusion and with Lemmas~\ref{Prob_Lem_Scale} and~\ref{Prob_Lem_Remove} we 
are able to estimate (\ref{eq:Prob_Linear}) for relevant choices of $S$ by relating the law of $X^{1,N}$ to that 
of standard Brownian motion. Specifically, in Corollary~\ref{Prob_Cor_SpatialConc} and Propositions~\ref{Prob_Prop_Boundary}
and~\ref{Prob_Prop_Tail} we show that $\nu^{N}$ satisfies the corresponding estimates to those in 
Assumption~\ref{Not_Def_RegCond}~(\ref{Not_Def_RegCond3}),~(\ref{Not_Def_RegCond4}) and~(\ref{Not_Def_RegCond5}), 
which is of direct use in Proposition~\ref{Tight_Prop_RegCond} when we take a limit as $N \to \infty$. In 
Propositions~\ref{Prob_Prop_LossInc} and~\ref{Prob_Prop_LowerLoss} we prove two estimates for which (\ref{eq:Prob_Linear}) 
is not helpful. These results require us to express the quantities of interest in terms of independent particles to show that certain
events concerning the increments in the loss process are asymptotically negligible.

\begin{lem}[Scale transformation]
\label{Prob_Lem_Scale}
Define $\zeta : [0,T] \times \R \to \R$ by
\[
\zeta(t, x) := \int^{x}_{0} \frac{dy}{\sigma(t, y)}
\]
and $Z_{t} := \zeta(t, X^{1,N}_{t})$. Then $\mathrm{sgn}(Z_{t}) = \mathrm{sgn}(X^{1,N}_{t})$ and 
$dZ_{t} = D_{t} dt + dB_{t}$ where $B$ is the Brownian motion
\[
B_{t} = \int^{t}_{0} \rho(s,L^{N}_{s}) dW_{s} + \int^{t}_{0} (1 - \rho(s,L^{N}_{s})^{2})^{\frac{1}{2}} dW^{1}_{s}
\]
and the drift coefficient, $D$, is given by
\[
D_{t}
	= (\frac{\mu}{\sigma} - \partial_{x} \sigma)(t, X^{1,N}_{t}, L^{N}_{t}) 
	- \int_{0}^{X^{1,N}_{t}} \frac{\partial_{t}\sigma}{\sigma^{2}}(t,y)dy,
\]
which is uniformly bounded (in $N$ and $t$).
\end{lem}

\begin{proof}
Straightforward application of It\^o's formula coupled with Assumption~\ref{Not_Def_Coefficients}).
\end{proof}

\begin{lem}[Removing drift]
\label{Prob_Lem_Remove} 
For every $\delta \in (0,1)$, there exists $c_{\delta} > 0$ such that
\[
\P(X^{1,N}_{t} \in S; t < \tau^{1,N})
	\leq c_{\delta} F_{t}(\zeta(t,S))^{\delta},
	\qquad \textrm{for every measurable } S \subseteq \R,
\]
where $F_{t}$ is the marginal law of a killed Brownian motion at time $t$ with initial distribution 
$\nu_{0} \circ \zeta(0, \cdot)^{-1}$ and $\zeta$ is as defined in Lemma~\ref{Prob_Lem_Scale}. Likewise, 
if $\bar{F}$ is the marginal law of the Brownian motion without killing at the origin and with the same initial distribution 
\[
\P(X^{1,N}_{t} \in S)
	\leq c_{\delta} \bar{F}_{t}(\zeta(t,S))^{\delta},
	\qquad \textrm{for every measurable } S \subseteq \R.
\]
\end{lem}

\begin{proof}
Let $Z$ be as in Lemma~\ref{Prob_Lem_Scale}, then $\tau^{1,N}$ is also the first hitting time, $\tau^{Z}$, of 0 by $Z$ so
\begin{equation}
\label{eq:Prob_Calc1a}
\P(X^{1,N}_{t} \in S; t < \tau^{1,N})
	= \P(Z_{t} \in \zeta(t,S); t < \tau^{Z}). 
\end{equation}
Apply Girsanov's Theorem with the change of measure 
\[
\left.\frac{d\mathbf{Q}}{d\mathbf{P}}\right|_{\mathcal{F}_{t}}\!
  = \exp\left\{ -\int_{0}^{t}D_{s}dB_{s}-\frac{1}{2}\int_{0}^{t}D_{s}^{2}ds\right\}
  =: \Xi_{t},
\]
then under \Q, $Z$ is a standard Brownian motion with $Z_{0} = \zeta(0, X^{1,N}_{0})$, and, for any $E \in \mathcal{F}_{t}$ 
and $p^{-1}+q^{-1}=1$, H\"older's inequality gives
\begin{multline*}
\P(E) 
  = \mathbf{E_{Q}} [\Xi_{t}^{-1}\mathbf{1}_{E}] 
  \leq \mathbf{E_{Q}} [\Xi_{t}^{-p}]^{\frac{1}{p}} \mathbf{Q}(E)^{\frac{1}{q}} \\
  =  \mathbf{E_{P}}[\Xi_{t}^{1-p}]^{\frac{1}{p}} \mathbf{Q}(E)^{\frac{1}{q}} 
   =  \exp\Bigl\{ C_{p}\int_{0}^{t} D_{s}^{2} ds\Bigr\} \mathbf{Q}(E)^{\frac{1}{q}}  
   \leq C_{q} \mathbf{Q}(E)^{\frac{1}{q}}, 
\end{multline*}
for some constant $C_{q} > 0$ as $D$ is uniformly bounded. Applying this bound to (\ref{eq:Prob_Calc1a}) gives
\[
\P(X^{1,N}_{t} \in S; t < \tau^{1,N})
  \leq C_{q} \Q(Z_{t} \in \zeta(t,S); t < \tau^{Z})^{\frac{1}{q}}
  = C_{q} F_{t}(\zeta(t, S))^{\frac{1}{q}}.
\]
The result is then complete by taking $\delta = q^{-1}$. The case involving $\bar{F}$ follows by dropping the dependence 
on $\{t < \tau^{1,N}\}$.
\end{proof}

The following result is a simple consequence of Lemma~\ref{Prob_Lem_Remove} and controls the expected mass concentrated 
in an interval.

\begin{cor}[Spatial concentration]
\label{Prob_Cor_SpatialConc}
For every $\delta \in (0,1)$ there exists $c_{\delta} > 0$ such that 
\[
\E \int^{T}_{0} \nu^{N}_{t}(a, b)dt
	\leq \E \int^{T}_{0} \nub^{N}_{t}(a, b)dt
	\leq c_{\delta} (b - a)^{\delta},
\]
for all $a < b$ and $N \geq 1$.
\end{cor}

\begin{proof}
Notice that $\zeta(t, (a, b)) \subseteq [\zeta(t,a), \zeta(t,b)]$, so with $\bar{F}$ as in Lemma~\ref{Prob_Lem_Remove}
\begin{align*}
\bar{F}_{t}(\zeta(t, (a,b)))
	&\leq \int^{\infty}_{0} \int^{\zeta(t,b)}_{\zeta(t,a)} \frac{1}{\sqrt{2 \pi t}} \exp \Big\{ - \frac{(x - \zeta(0,x_{0}))^{2}}{2t} \Big\} dx \nu_{0}(dx_{0}) \\
	&\leq (2\pi t)^{-1/2} (\zeta(t,b) - \zeta(t,a)) \\
	&= (2\pi t)^{-1/2} \int^{b}_{a} \frac{dy}{\sigma(t, y)}
	\leq (2\pi t)^{-1/2} \cdot C \cdot (b - a),
\end{align*}
and then the result is immediate from Lemma~\ref{Prob_Lem_Remove} since $t \mapsto t^{-\delta / 2}$ is integrable at the origin. 
\end{proof}


\subsection*{Boundary estimate}

A sharper application of Lemma~\ref{Prob_Lem_Remove} gives control of the concentration of mass near the origin. 
Notice the stronger rate of convergence due to the absorption at the boundary:

\begin{prop}[Boundary estimate]
\label{Prob_Prop_Boundary}
There exists $\beta > 0$ and $\delta \in (0,1)$ such that as $\e \to 0$
\[
\E\nu^{N}_{t}(0,\e)
	= t^{-\frac{\delta}{2}}O(\e^{1 + \beta})
	\qquad\textrm{and}\qquad
	\E\int^{T}_{0} \nu^{N}_{t}(0,\e) dt = O(\e^{1 + \beta})
\]
where the $O$'s are uniform in $t \in [0,T]$ and $N \geq 1$. 
\end{prop}

\begin{proof}
Let $F$ be as in Lemma~\ref{Prob_Lem_Remove}. The heat kernel for a Brownian motion absorbed at the origin is
\begin{equation}
G_{t}(x_{0},x) 
  = (2 \pi t)^{-\frac{1}{2}} \Bigl[ \exp \Bigl\{-\frac{(x - x_{0})^{2}}{2t} \Bigr\} - \exp \Bigl\{-\frac{(x + x_{0})^{2}}{2t} 
  \Bigr\} \Bigr],
\label{eq:hkabsorbing}
\end{equation}
for $x_{0}, x, t > 0$. By using the bounds $G_{t}(x_{0},x) \leq (2 \pi t)^{-1/2}$ and 
\[
G_{t}(x_{0},x)
  \leq \frac{2 x_{0} x}{\sqrt{2 \pi t^{3}}} \exp \Bigl\{-\frac{(x - x_{0})^{2}}{2t} \Bigr\},
\]
which follows from the simple estimate $1-e^{-z} \leq z$, for an arbitrary function $f = f(\e)$ we have, writing 
$\pi_{0} := \nu_{0} \circ \zeta(0, \cdot)^{-1}$, that
\begin{align*}
F_{t}((0,\e))
	&\leq c_{1} t^{-\frac{1}{2}} \int^{\e}_{0} \int^{\e + f(\e)}_{0} \pi_{0}(dx_{0}) dx\\
	& \qquad + c_{1}t^{-\frac{3}{2}} \int^{\e}_{0} \int^{\infty}_{\e + f(\e)} x x_{0} \exp \Bigl\{-\frac{(x - x_{0})^{2}}{2t} \Bigr\} 
	\pi_{0}(dx_{0}) dx  \\
	&\leq c_{1}t^{-\frac{1}{2}} \e \pi_{0}(0, \e + f(\e)) \\
		& \qquad+ c_{1} t^{-\frac{3}{2}} \exp \Big\{ - \frac{f(\e)^{2}}{2t}  \Big\} \cdot \int^{\e}_{0} x dx \cdot \int_{0}^{\infty} 
		x_{0}  \pi_{0}(dx_{0})
\end{align*}
where $c_{1} > 0$ is a numerical constant. By Assumption~\ref{Not_Def_Coefficients}(i) we have a constant $c_{2} > 0$ such that
\[
F_{t}((0,\e))
	\leq c_{1}t^{-\frac{1}{2}} \e \nu_{0}(0, c_{2}(\e + f(\e))) 
	+ c_{2} t^{-\frac{3}{2}} \e^{2}\exp\{-f(\e)^{2} / 2t\}.
\]
Since the function 
\[
u \mapsto u^{-\alpha}\exp\{- \beta / u \},	\qquad \textrm{for } u > 0, \alpha, \beta > 0 
\]
is maximised at $u = \beta/\alpha$, we have the bound
\[
F_{t}((0,\e))
	\leq c_{3} t^{-\frac{1}{2}} \e \{ \nu_{0}(0, c_{2}(\e + f(\e))) + \e f(\e)^{-2} \}.
\]
Taking $f(\e) = \e^{1/3}$ gives 
\[
F_{t}((0,\e)) = t^{-\frac{1}{2}}O(\e^{1 + \frac{1}{6}})
\]
since $\nu_{0}(0,x) = O(x^{1/2})$ as $x \to 0$ (recall 
Assumption~\ref{Not_Def_Coefficients}(\ref{Not_Def_Coefficients_1})). The result is complete by applying Lemma~\ref{Prob_Lem_Remove} and noting that $\zeta(t, (0,\e)) \subseteq [0, \zeta(t,\e)] \subseteq [0,C\e]$. 
\end{proof}


\subsection*{Tail estimate}

A similar analysis applies for the decay of the mass that escapes to infinity.

\begin{prop}[Tail estimate]
\label{Prob_Prop_Tail}
For every $\alpha > 0$, as $\lambda \to +\infty$ 
\[
\E \nu^{N}_{t}(\lambda,\infty)= o(\exp\{-\alpha \lambda\}),
\qquad \textrm{uniformly in } N \geq 1 \textrm{ and } t \in [0,T].
\]
\end{prop}

\begin{proof}
Working with $\bar{F}$ from Lemma~\ref{Prob_Lem_Remove} and splitting the range of integration at $\lambda / 2$ gives
\begin{align*}
\bar{F}_{t}((\lambda, \infty))
	&= \int_{0}^{\infty} \P(B_{t} > \lambda|B_0=x) \pi_{0}(dx) \\
	&\leq c_{1} t^{-\frac{1}{2}} \exp\Big\{ -\frac{\lambda^{2}}{8t} \Big\} 
	+ \pi_{0}(\lambda / 2, \infty),
\end{align*}
where $\pi_{0} = \nu_{0} \circ \zeta(0, \cdot)^{-1}$. By the conditions of Assumption~\ref{Not_Def_Coefficients}, 
$\pi_{0}(\lambda / 2, \infty) = o(e^{-\alpha \lambda})$, so
\[
\bar{F}_{t}((\lambda, \infty))
	\leq c_{1} t^{-\frac{1}{2}}e^{-2\lambda^{2} / t} + o(e^{-\alpha \lambda})
	\leq c_{1} \{t^{-\frac{1}{2}} e^{-\lambda^{2} / t}\} e^{-\lambda^{2} / T} + o(e^{-\alpha \lambda}),
\]
as $\lambda \to \infty$, for every $\alpha > 0$. The result follows since $t \mapsto t^{-\frac{1}{2}} e^{-\lambda^{2} / t}$ is uniformly bounded 
for $\lambda \geq 1$, and using Lemma~\ref{Prob_Lem_Remove} with the fact that $\zeta(t, (\lambda, \infty)) \subseteq [\zeta(t, \lambda), \infty) \subseteq [C^{-1}\lambda, \infty)$.
\end{proof}


\subsection*{Loss increment estimate}

So far the probabilistic estimates we have seen are consequences of the behaviour of the first moment of the diffusion processes. 
The next two estimates require knowledge of the correlation between particles and so are harder to prove. Heuristically, the 
first result shows that over any non-zero time interval a non-zero proportion of particles hit the absorbing boundary. Later 
in Proposition~\ref{Tight_Prop_RegCond} this result will directly imply that limiting loss functions are strictly increasing whenever
there is a non-zero proportion of mass remaining in the system. 

\begin{prop}[Asymptotic loss increment]
\label{Prob_Prop_LossInc}
For all $t \in [0,T)$, $h > 0$ (such that $t + h \in [0,T]$) and $r < 1$
\[
\lim_{\delta \to 0} \limsup_{N \to \infty} \P(L^{N}_{t+ h} - L^{N}_{t} < \delta, L^{N}_{t} < r) =0.
\]
\end{prop}

\begin{proof}
Begin by noticing that, for any $a,b > 0$, if $L^{N}_{t} < r$ and $\nu^{N}_{t}(a, \infty) \leq b$, then $\nu^{N}_{t}(0, a) > 1 - r - b$. 
By applying Markov's inequality and Proposition~\ref{Prob_Prop_Tail} we get the bound
\begin{align*}
\P(L^{N}_{t+h} - L^{N}_{t} < \delta, L^{N}_{t} < r)
	&\leq \P(L^{N}_{t+h} - L^{N}_{t} < \delta, \nu^{N}_{t}(0, a) > 1 - r - b) \\
	& \qquad + \P(\nu^{N}_{t}(a, \infty) > b)  \\
	&\leq \P(L^{N}_{t+h} - L^{N}_{t} < \delta, \nu^{N}_{t}(0, a) > 1 - r - b)  \\
	& \qquad + o(e^{-a}).
\end{align*}
Therefore fix $b = 1 - r - c_{0}$, for $c_{0} = \frac{1}{2}(1 - r)$, to arrive at
\begin{equation}
\label{eq:Prob_LossInc0}
\P(L^{N}_{t+h} - L^{N}_{t} < \delta, L^{N}_{t} < r) \leq \P(L^{N}_{t+h} - L^{N}_{t} < \delta, \nu^{N}_{t}(0, a) > c_{0}) 
+ o(e^{-a}).
\end{equation}

We now concentrate on the first term in the right-hand side above with $N$, $t$ and $a$ fixed. Let $\cI$ denote the random 
set of indices
\[
\cI := \{ 1 \leq i \leq N : X^{i, N}_{t} < a \textrm{ and } \tau^{i,N} > t \}.
\]
If $\nu^{N}_{t}(0, a) > c_{0}$, then $\# \cI \geq Nc_{0}$, so by conditioning on $\cI$ (which is $\mathcal{F}_{t}$-measurable)
\begin{multline}
\label{eq:Prob_LossInc1}
\P(L^{N}_{t+h} - L^{N}_{t} \leq \delta, \nu^{N}_{t}(0,a) > c_{0}) \\
	\leq \sum_{\cI_{0} : \# \cI_{0} \geq Nc_{0}} \P(L^{N}_{t+h} - L^{N}_{t} < \delta | \cI = \cI_{0}) \P(\cI = \cI_{0})
\end{multline}
and 
\begin{equation}
\label{eq:Prob_LossInc2}
\P(L^{N}_{t+h} - L^{N}_{t} < \delta | \cI = \cI_{0})
	\leq \P(\#\{ i \in \cI_{0} : \inf_{u \leq h} X^{i, N}_{t + u} \leq 0 \} < N\delta | \cI = \cI_{0})
\end{equation}

To estimate the right-hand side of (\ref{eq:Prob_LossInc2}) take $\zeta$ as in Lemma~\ref{Prob_Lem_Scale} and define 
$Z^{i}_{t} := \zeta(t, X^{i,N}_{t})$ for $1 \leq i \leq N$. By Assumption~\ref{Not_Def_Coefficients}, there exists a 
constant $c_{1} > 0$ such that $|D^{i}_{t}| \leq c_{1}$ for all $t$. Returning to (\ref{eq:Prob_LossInc2}), since 
$\zeta(t, x) \leq 0$ if and only if $x \leq 0$, we have 
\[
\P(L^{N}_{t+h} - L^{N}_{t} < \delta | \cI = \cI_{0})
	\leq \P(\#\{ i \in \cI_{0} : \inf_{u \leq h} Z^{i}_{t + u} \leq 0 \} < N\delta | \cI = \cI_{0}).
\]
From the bound $Z^{i}_{t + u} \leq Z^{i}_{t} + c_{1} h + Y^{i}_{u}$, for $0\leq u \leq h$, where
\[
Y^{i}_{u}
	:= I_{u} + J^{i}_{u} := \int^{t + u}_{t} \rho(s, L^{N}_{s}) dW_{s}
		+ \int^{t + u}_{t} \sqrt{1 - \rho(s, L^{N}_{s})^{2}} dW^{i}_{s},
\]
we obtain
\begin{multline*}
\P(L^{N}_{t+h} - L^{N}_{t} < \delta | \cI = \cI_{0}) \\
	\leq \P(\#\{ i \in \cI_{0} : \inf_{u \leq h} Y^{i}_{u} \leq -Z^{i}_{t} - c_{1} h  \} < N\delta | \cI = \cI_{0}). 
\end{multline*}
From Assumption~\ref{Not_Def_Coefficients} $|Z^{i}_{t}| = O(|X^{i,N}_{t}|)$, so we have $c_{2} > 0$ such that
\begin{multline}
\label{eq:Prob_LossInc3}
\P(L^{N}_{t+h} - L^{N}_{t} < \delta | \cI = \cI_{0}) \\
	\leq \P(\#\{ i \in \cI_{0} : \inf_{u \leq h} Y^{i}_{u} \leq -c_{2} a - c_{2}  \} < N\delta | \cI = \cI_{0}).
\end{multline}

Our next step is to remove the dependence on the process $I$ in (\ref{eq:Prob_LossInc3}). To do this we split the probability 
on the event $\{\sup_{u \leq h} |I_{u}| \geq c_{2}a\}$ to get 
\begin{align*}
\P(L^{N}_{t+h} - L^{N}_{t} < \delta &| \cI = \cI_{0}) \\
	&\leq \P(\#\{ i \in \cI_{0} : \inf_{u \leq h} J^{i}_{u} \leq -2c_{2} a - c_{2}  \} < N\delta | \cI = \cI_{0}) \\
	&\qquad + \P(\sup_{u \leq h} |I_{u}| \geq c_{2}a | \cI = \cI_{0}).
\end{align*}
Since $I$ is a martingale, this final probability is $o(1)$ as $a \to \infty$, by Doob's maximal inequality. 

We have reduced the problem far enough to apply a time-change in order to extract the independence between the particles. 
To this end, conditioned on the event $\cI = \cI_{0}$, define 
\begin{equation}
\label{eq:Def_v}
v(s)
	:= \inf \{ u > 0 : \int^{t + u}_{t} (1 - \rho(u_{0}, L^{N}_{u_{0}})^{2}) du_{0} = s \}, 
\end{equation}
then $B$, where $B^{i} := J^{i}_{v(\cdot)}$, is an $\R^{\# \cI_{0}}$-valued standard Brownian motion, therefore 
\begin{multline*}
\P(L^{N}_{t+h} - L^{N}_{t} < \delta | \cI = \cI_{0}) \\
	\leq \P(\#\{ i \in \cI_{0} : \inf_{v(u) \in [0,h]} B^{i}_{u} \leq -2c_{2} a - c_{2}  \} < N\delta | \cI = \cI_{0})
	+ o(1).
\end{multline*}
By Assumption~\ref{Not_Def_Coefficients}, $c_{3}u \leq |v(u)|$, hence 
\begin{align*}
\P(L^{N}_{t+h}&- L^{N}_{t} < \delta | \cI = \cI_{0}) \\
	&\leq \P(\#\{ i \in \cI_{0} : \inf_{u \in [0,h/c_{3}]} B^{i}_{u} \leq -2c_{2} a - c_{2}  \} < N\delta | \cI = \cI_{0})
	+ o(1)   \\
	&\leq \P(\#\{ i \in \cI_{0} :  B^{i}_{h/c_{3}} \leq -2c_{2} a - c_{2}  \} < N\delta | \cI = \cI_{0})
	+ o(1)  \\
	&\leq \P\Big(\frac{1}{N} \sum_{i \in \cI_{0}} \1_{\xi^{i} \leq -c_{4} (a +1)} < \delta\Big)
	+ o(1),
\end{align*}
where $\{\xi^{i}\}_{1 \leq i \leq  N}$ is a collection of i.i.d.\ standard normal random variables and $c_{3}, c_{4} > 0$ 
are further numerical constants. By symmetry, this final probability depends only on $\#\cI_{0}$, hence 
\[
\P(L^{N}_{t+h} - L^{N}_{t} < \delta | \cI = \cI_{0})
	\leq \P\Big(\frac{1}{N} \sum_{i = 1}^{\#\cI_{0}} \1_{\xi^{i} \leq -c_{4} (a +1)} < \delta\Big) + o(1).
\]

Returning to (\ref{eq:Prob_LossInc1}) we now have
\begin{align*}
\P(L^{N}_{t+h} &- L^{N}_{t} < \delta, \nu^{N}_{t}(0,a) > c_{0}) \\
	&\leq \sum_{S_{0} : \# \cI_{0} \geq Nc_{0}} \!\!\! \P\Big(\frac{1}{N} \sum_{i = 1}^{\#\cI_{0}} \1_{\xi^{i} 
	\leq -c_{4} (a +1)} < \delta\Big) \P(\cI = \cI_{0}) 	+ o(1) \\
	&\leq \P\Big(\frac{1}{N} \sum_{i = 1}^{Nc_{0}} \1_{\xi^{i} \leq -c_{4} (a +1)} < \delta\Big) + o(1),
\end{align*}
so the law of large numbers gives
\begin{equation}
\label{eq:Prob_LossInc4}
\limsup_{n \to \infty}\P(L^{N}_{t+h} - L^{N}_{t} < \delta, L^{N}_{t} < r)
	 \leq \1_{c_{0} p(a) \leq \delta} + o(1),
\end{equation}
where $p(a) := \P(\xi^{1} \leq -c_{4} (a +1))$ and where we have substituted back into (\ref{eq:Prob_LossInc0}). This inequality
holds for all $a$ and $\delta$, with the $o(1)$ term denoting convergence as $a \to \infty$. We now choose the free parameter 
$a$ to be a function of $\delta$, specifically
\[
a(\delta) := (2 \log \log (1 / \delta))^{\frac{1}{2}}.
\]
This guarantees that $a(\delta) \to \infty$ as $\delta \to 0$, but also 
\begin{align*}
\delta^{-1} p(a(\delta)) 
	&\geq \frac{1}{2} \delta^{-1} a(\delta)^{-1} e^{-a(\delta)^{2} / 2} \\
	&= \frac{1}{\sqrt{2}} \delta^{-1} (\log(1 / \delta))^{-1} (\log \log(1 / \delta))^{1/2}
		\to \infty
\end{align*}
as $\delta \to 0$, where we have used the well-known Gaussian estimate $\Phi(-x) \geq (x^{-1} - x^{-3})\phi(x) \geq 
\frac{1}{2}x^{-1}\phi(x)$, for $\Phi$ and $\phi$ the c.d.f.\ and p.d.f.\ of the standard normal distribution. Using this 
choice of $a(\delta)$ in (\ref{eq:Prob_LossInc4}) completes the result.
\end{proof}

The following is a partial converse of the previous result in that it shows that the system cannot lose a large amount of mass 
in a short period of time. It will be used in Proposition~\ref{Tight_prop_tight} to verify a sufficient condition for the tightness 
of $(\nu^{N}, W)_{N \geq 1}$. 

\begin{prop}
\label{Prob_Prop_LowerLoss}
For every $t \in [0,T]$ and $\eta > 0$
\[
\lim_{\delta \to 0} \limsup_{N \to \infty} \P(L^{N}_{t + \delta} - L^{N}_{t} \geq \eta)
	= 0.
\]
\end{prop}

\begin{proof}
With $\e > 0$ fixed, we have 
\begin{align}
\label{eq:LossLowerI}
\P(L^{N}_{t + \delta} &- L^{N}_{t} \geq \eta) \\
	&\leq \P(\nu^{N}_{t}(0,\e) \geq \eta / 2) \nonumber
	+ \P(L^{N}_{t + \delta} - L^{N}_{t} \geq \eta, \nu^{N}_{t}(0,\e) < \eta / 2)  \\
	&\leq 2\eta^{-1} \P(X^{1,N}_{t} \in (0, \e))
	+ \P(L^{N}_{t + \delta} - L^{N}_{t} \geq \eta, \nu^{N}_{t}(0,\e) < \eta / 2), \nonumber \\
	&\leq \P(L^{N}_{t + \delta} - L^{N}_{t} \geq \eta, \nu^{N}_{t}(0,\e) < \eta / 2) + o(1), \qquad \textrm{as }\e \to 0, \nonumber
\end{align}
where the second line uses Markov's inequality and (\ref{eq:Prob_Linear}) and the third line uses 
Proposition~\ref{Prob_Prop_Boundary} for $t> 0$ and Assumption~\ref{Not_Def_Coefficients} (\ref{Not_Def_Coefficients_1}) 
for $t = 0$.  Define $\cI$ to be the random set of indices
\[
\cI := \{ 1 \leq i \leq N : X^{i,N}_{t} \geq \e \},
\]
then conditioning on $\cI$ gives
\begin{multline}
\label{eq:LossLowerIII}
\P(L^{N}_{t + \delta} - L^{N}_{t} \geq \eta, \nu^{N}_{t}(0,\e) < \eta / 2) \\
	\leq \sum_{\cI_{0} : \# \cI_{0} \geq N(1-\eta / 2)} \P(L^{N}_{t + \delta} - L^{N}_{t} \geq \eta | \cI = \cI_{0}) \P(\cI = \cI_{0}).
\end{multline}
The conditional expectation in the summand can be bounded by
\begin{align*}
\P(L^{N}_{t + \delta} &- L^{N}_{t} \geq \eta | \cI = \cI_{0}) \\
	&\leq \P(\#\{i \in \cI_{0} : \inf_{s \in [t, t+ \delta]} X^{i,N}_{s} \leq 0 \} \geq \frac{N\eta}{2}  | \cI = \cI_{0})   \\
	&\leq \P(\#\{i \in \cI_{0} : \inf_{s \in [t, t+ \delta]} ( X^{i,N}_{s} - X^{i,N}_{t} ) \leq -\e \} \geq \frac{N\eta}{2}  | \cI = \cI_{0}).
\end{align*}

With $t$ fixed, define the process $U^{i}_{s} := \zeta(t + s, X^{i,N}_{t + s} - X^{i,N}_{t})$, then 
\[
\P(L^{N}_{t + \delta} - L^{N}_{t} \geq \eta | \cI = \cI_{0})
	\leq \P(\#\{i \in \cI_{0} : \inf_{s \in [0, \delta]} U^{i}_{s} \leq -c_{5}\e \} \geq \frac{N\eta}{2}  | \cI = \cI_{0})
\]
for $c_{5} > 0$ a numerical constant. As for $Z$ in Lemma~\ref{Prob_Lem_Scale}, we have 
\begin{align*}
dU^{i}_{s} 
	&= E^{i}_{s} ds 
	+ \rho(t + s, L^{N}_{t + s}) dW_{t + s}
	+ (1 - \rho(t + s, L^{N}_{t + s})^{2})^{1/2} dW^{i}_{t + s} \\
	&=: E^{i}_{s}ds + dI_{s} + dJ^{i}_{s},
\end{align*}
where $E^{i}_{s}$ is uniformly bounded by Assumption~\ref{Not_Def_Coefficients}, therefore we can find $c_{6} > 0$ such that
\begin{align*}
\P(L^{N}_{t + \delta} &- L^{N}_{t} \geq \eta | \cI = \cI_{0}) \\
	&\leq \P(\#\{i \in \cI_{0} : \inf_{s \in [0, \delta]} J^{i}_{s} \leq -c_{6}(\e -\delta - a) \} \geq \frac{N\eta}{2}  | \cI = \cI_{0})  \\
	&\qquad + \P(\sup_{s \in [0,\delta]}|I_{s}| \geq a | \cI = \cI_{0} ).
\end{align*}
By applying the time-change argument from (\ref{eq:Def_v}) and using Markov and Doob's maximal inequality we have
\begin{multline*}
\P(L^{N}_{t + \delta} - L^{N}_{t} \geq \eta | \cI = \cI_{0}) \\
	\leq \P(\#\{i \in \cI_{0} : \inf_{s \in [0, \delta]} B^{i}_{s} \leq -c_{7}(\e -\delta - a) \} \geq \frac{N\eta}{2} ) + O(\delta a^{-2}),
\end{multline*}
where $B^{i}$ are independent standard Brownian motions, $a > 0$ and $c_{7} > 0$ is a numerical constant. 

Returning to (\ref{eq:LossLowerIII}) and noticing the the right-hand side above is maximised when $\cI_{0} = \{1, 2, \dots, N\}$
\begin{multline*}
\P(L^{N}_{t + \delta} - L^{N}_{t} \geq \eta, \nu^{N}_{t}(0,\e) < \eta / 2) \\
	\leq \P \Big( \frac{1}{N} \sum_{i=1}^{N} \1_{\inf_{s \in [0, \delta]} B^{i}_{s} \leq -c_{7}(\e -\delta - a) \}} \geq \eta / 2 \Big)
	 + O(\delta a^{-2}).
\end{multline*}
The law of large numbers and the distribution of the minimum of Brownian motion gives 
\begin{multline}
\label{eq:LowerLossIV}
\limsup_{N \to \infty} \P(L^{N}_{t + \delta} - L^{N}_{t} \geq \eta, \nu^{N}_{t}(0,\e) < \eta / 2) \\
	\leq \1_{\Phi(-c_{7}\delta^{-1/2}(\e - \delta - a)) \geq \eta / 2} 
	+ O(\delta a^{-2}),
\end{multline}
provided $\e - \delta - a > 0$, where $\Phi$ is the normal c.d.f. We now make the choice
\[
\e(\delta) = \delta^{1/2} \log(1/\delta)
	\qquad \textrm{and} \qquad
	a(\delta) = \delta^{1/2} \log \log (1/\delta),
\] 
which guarantees
\[
\e(\delta) \to 0,
	\qquad \delta^{-1/2}(\e(\delta) - \delta - a(\delta)) \to \infty
	\qquad \textrm{and} \qquad
	\delta a(\delta)^{-2} \to 0,
\]
as $\delta \to 0$. Hence the result follows from (\ref{eq:LossLowerI}), (\ref{eq:LossLowerIII}) and (\ref{eq:LowerLossIV}).
\end{proof}


\section{Tightness of the system and existence of solutions; Proof of Theorem \ref{Intro_Thm_Exist}}
\label{Sect_Tightness}

We will now use the results from Section~\ref{Sect_ProbEstimates} to prove Theorem~\ref{Intro_Thm_Exist}, which 
follows directly from the combination of Propositions~\ref{Tight_Prop_ConvLoss},~\ref{Tight_Prop_RegCond}
and~\ref{Tight_Prop_EvoEqn}. We first establish tightness of the sequence of the laws of $(\nu^{N}, W)_{N \geq 1}$
(Proposition~\ref{Tight_prop_tight}) using the framework of \cite{ledger2015}. The reader is referred to that article 
for the technical definitions of the topological spaces used in this section. Once we have tightness we can then extract 
limit points of the sequence $(\nu^{N}, W)_{N \geq 1}$, and Propositions~\ref{Tight_Prop_Measure},~\ref{Tight_Prop_ConvLoss}
and~\ref{Tight_Prop_RegCond} are devoted to recovering the properties of the limiting laws from the probabilistic properties 
of the finite system. Finally, the limit points are shown to satisfy the evolution equation in Theorem~\ref{Intro_Thm_Exist} 
via a martingale argument (Proposition~\ref{Tight_Prop_EvoEqn}) and care needs to be taken over the discontinuities in 
the coefficients of the limit SPDE (Corollary~\ref{Tight_Cor_Fiddly}). 

\begin{prop}[Tightness]
\label{Tight_prop_tight}
The sequence $(\nu^{N})_{N \geq 1}$ is tight on the space $(D_{\Sdual}, \M)$, hence $(\nu^{N}, W)_{N \geq 1}$ is tight 
on the space $(D_{\Sdual}, \M) \times (C_{\mathbb{R}}, \mathrm{U})$, where $(C_{\mathbb{R}}, \mathrm{U})$ is the 
space of real-valued continuous paths with the topology of uniform convergence.  
\end{prop}

\begin{rem}
We note that a version of this result is given in \cite[Thm.~4.3]{ledger2015} for the case $\mu = 0$, $\sigma = 1$. 
\end{rem}

\begin{proof}
The second statement follows from the first and the fact that joint tightness is implied by marginal tightness. 

By \cite[Thm.~3.2]{ledger2015} it suffices to show that $(\nu^{N}(\phi))_{N \geq 1}$ is tight on $(D_{\R}, \M)$ for every 
$\phi \in \S$. To prove this we verify the conditions of \cite[Thm.~12.12.2]{whitt2002}, the first of which is trivial because 
$\nu^{N}$ is a sub-probability measure so $|\nu^{N}_{t}(\phi)| \leq \Vert \phi \Vert_{\infty}$. Hence we concentrate on 
condition (ii), which is implied by \cite[Prop.~4.1]{ledger2015}, therefore we are done if we can find $a,b,c > 0$ such that
\begin{equation}
\label{eq:Tight1}
\P(H_{\R}(\nu^{N}_{t_{1}}(\phi),  \nu^{N}_{t_{2}}(\phi), \nu^{N}_{t_{3}}(\phi)) \geq \eta)
	\leq c\eta^{-a}|t_{3} - t_{1}|^{1 + b},
\end{equation}
for all $N \geq 1$, $\eta > 0$ and $0 \leq t_{1} < t_{2} < t_{3} \leq T$, where 
\[
H_{\R}(x_{1}, x_{2}, x_{3}) 
	:= \inf_{\lambda \in (0,1)} |x_{2} - (1-\lambda)x_{1} - \lambda x_{3}|
	\qquad \textrm{for } x_{1}, x_{2}, x_{3} \in \R, 
\]
and if
\begin{equation}
\label{eq:Tight2}
\lim_{N \to \infty} \P(\sup_{t \in (0, \delta)} |\nu^{N}_{t}(\phi) - \nu^{N}_{0}(\phi)| + \sup_{t \in (T - \delta, T)} 
|\nu^{N}_{T}(\phi) - \nu^{N}_{t}(\phi)| \geq \eta ) = 0,
\end{equation}
for every $\eta > 0$. 

With $\nub^{N}$ as defined in (\ref{eq:Finite_nub}), the decomposition in \cite[Prop.~4.2]{ledger2015} and 
Markov's inequality give
\begin{align*}
\P(H_{\R}(\nu^{N}_{t_{1}}(\phi), & \nu^{N}_{t_{2}}(\phi), \nu^{N}_{t_{3}}(\phi)) \geq \eta) \\
	&\leq \eta^{-4} \E[(|\nub^{N}_{t_{1}}(\phi)-\nub^{N}_{t_{2}}(\phi)|
	+ |\nub^{N}_{t_{2}}(\phi)-\nub^{N}_{t_{3}}(\phi)|)^{4}]  \\
	&\leq 8 \eta^{-4} (\E|\nub^{N}_{t_{1}}(\phi)-\nub^{N}_{t_{2}}(\phi)|^{4} + 
	\E|\nub^{N}_{t_{2}}(\phi)-\nub^{N}_{t_{3}}(\phi)|^{4}).
\end{align*}
For any $t, s \in [0,T]$, from H\"older's inequality we obtain
\begin{align*}
\E|\nub^{N}_{t}(\phi) - \nub^{N}_{s}(\phi)|^{4}
	&\leq \frac{1}{N} \sum_{i=1}^{N} \E|\phi(X^{i,N}_{t \wedge \tau^{i,N}}) - \phi(X^{i,N}_{s \wedge \tau^{i,N}})|^{4} \\
	&\leq \Vert \phi \Vert_{\mathrm{lip}}^{4} \E|X^{i,N}_{t \wedge \tau^{i,N}} - X^{i,N}_{s \wedge \tau^{i,N}}|^{4},
\end{align*}
where $\Vert \phi \Vert_{\mathrm{lip}}$ is the Lipschitz constant of $\phi$. By Assumption~\ref{Not_Def_Coefficients} and 
the Burkholder--Davis--Gundy inequality \cite[Thm.~IV.42.1]{rogerswilliams2}, the final expectation above is $O(|t - s|^{2})$
uniformly in $N$. Therefore we have (\ref{eq:Tight1}) with $a = 4$ and $b = 1$. 

Now consider the first supremum in (\ref{eq:Tight2}). By again using the decomposition from \cite[Prop.~4.2]{ledger2015}, 
that is $\nu^{N}_{t}(\phi) = \nub^{N}_{t}(\phi) - \phi(0)L^{N}_{t}$, we have 
\begin{multline*}
\P(\sup_{t \in (0,\delta)} |\nu^{N}_{t}(\phi) - \nu^{N}_{0}(\phi)| \geq \eta) \\
	\leq \P(\sup_{t \in (0,\delta)} |\nub^{N}_{t}(\phi) - \nub^{N}_{0}(\phi)| \geq \eta / 2)
	+ \P(|\phi(0)|L^{N}_{\delta} \geq \eta / 2).
\end{multline*}
The first term on the right-hand side vanishes as $\delta \to 0$ by the same work as for (\ref{eq:Tight1}) and the second 
term vanishes by Proposition~\ref{Prob_Prop_LowerLoss}. Therefore  
\[
\P(\sup_{t \in (0,\delta)} |\nu^{N}_{t}(\phi) - \nu^{N}_{0}(\phi)| \geq \eta) \to 0, \qquad \textrm{as } \delta \to 0,
\]
and likewise for $\P(\sup_{t \in (T-\delta,T)} |\nu^{N}_{T}(\phi) - \nu^{N}_{t}(\phi)| \geq \eta)$, so we have (\ref{eq:Tight2}), 
which completes the proof.
\end{proof}


\subsection*{Limit points}

Tightness of $(\nu^{N}, W)_{N \geq 1}$ ensures that the sequence is relatively compact \cite[Thm.~3.2]{ledger2015}, 
hence every subsequence of $(\nu^{N}, W)_{N \geq 1}$ has a further subsequence which converges in law. To avoid 
possible confusion about multiple distinct limit points, we will denote by $(\nu^{*}, W)$ any pair of processes that realises 
one of these limiting laws. Using $\Rightarrow$ to denote convergence in law, we have
\[
(\nu^{N_{k}}, W) \Rightarrow (\nu^{*}, W),
	\qquad \textrm{on } (D_{\Sdual}, \M) \times (C_{\R}, \mathrm{U}),
\]
as $k \to \infty$, for some subsequence $(N_{k})_{k \geq 1}$. Establishing full weak convergence is equivalent to showing 
that there is exactly one limiting law. 

So far we have that any limiting empirical process, $\nu^{*}$, is an element of $D_{\Sdual}$. The following result 
recovers $\nu^{*}$ as a probability-measure-valued process:

\begin{prop}
\label{Tight_Prop_Measure}
Let $(\nu^{*}, W)$ realise a limiting law. Then $\nu^{*}_{t}$ is a sub-probability measure supported on $[0,\infty)$ 
for every $t \in [0,T]$, with probability 1. 
\end{prop}

\begin{rem}
Technically, what we will show is that, for every $t$, $\nu^{*}_{t}$ agrees with a sub-probability measure on $\S$ and 
from now on we associate $\nu^{*}_{t}$ with this measure.
\end{rem}

\begin{proof}[Proof of Proposition~\ref{Tight_Prop_Measure}]
Take $(\nu^{N_{k}}, W) \Rightarrow (\nu^{*}, W)$. Fix $\phi \in \S$, then by \cite[Prop.~2.7 (i)]{ledger2015} 
$\nu^{N_{k}}(\phi) \Rightarrow \nu^{*}(\phi)$ on $(D_{\R}, \M)$. Lemma~13.4.1 of \cite{whitt2002} gives
\[
\sup_{t \in [0,T]} |\nu^{N_{k}}_{t}(\phi)|
	\Rightarrow
	\sup_{t \in [0,T]} |\nu^{*}_{t}(\phi)|,
	\qquad \textrm{on } \R,
\]
therefore the portmanteau theorem~\cite[Thm.~2.1]{billingsley1999} gives
\[
\P(\sup_{t \in [0,T]} |\nu^{*}_{t}(\phi)| > \Vert \phi \Vert_{\infty})
	\leq \liminf_{k \to \infty} \P(\sup_{t \in [0,T]} |\nu^{N_{k}}_{t}(\phi)| > \Vert\phi \Vert_{\infty})
	= 0,
\]
with the final equality due to $\nu^{N}_{t}$ being a sub-probability measure. (The supremum over $t$ ensures that 
the following argument holds for all $t$ simultaneously.) By a similar analysis we have that $\nu^{*}_{t}(\phi)$ is 
non-negative when $\phi$ is non-negative and $\nu^{*}_{t}(\phi) = 0$ when $\phi$ is supported on $(-\infty, 0)$. 
Hence, $\nu^{*}_{t}$ is a positive linear functional on \S, so extends to a positive linear functional, $\xi_{t}$,  on 
the space, $C_{0}$, of continuous and compactly support function on \R\ with the uniform topology. The Riesz 
representation theorem~\cite[Thm.~2.14]{rudin} then implies that, for every $t$, there exists a regular Borel 
measure, $\zeta_{t}$, such that
\[
\xi_{t}(\phi) = \int_{\R} \phi(x) \zeta_{t}(dx)
	\qquad \textrm{for every } \phi \in C_{0}.
\]
Associating $\zeta$ and $\nu^{*}$ gives the result. 
\end{proof}

Now that it is safe to regard a limit point, $\nu^{N_{k}} \Rightarrow \nu^{*}$, as taking values in the sub-probability 
measures, it makes sense to introduce the \emph{limit loss process} as $L^{*}_{t} := 1 - \nu^{*}_{t}(0, \infty)$. 
Of course we would like to know that $L^{N_{k}} \Rightarrow L^{*}$ on $(D_{\R}, \M)$, however the function 
$x \mapsto 1$ is not an element of \S, so \cite[Prop.~2.7]{ledger2015} does not allow us to deduce this fact from the 
continuous mapping theorem. To remedy this we must work slightly harder:

\begin{prop}[Convergence of the loss process]
\label{Tight_Prop_ConvLoss}
Suppose that $(\nu^{N_{k}}, W)_{k \geq 1}$ converges weakly to $(\nu^{*}, W)$ and that 
$L_{t}^{*} := 1 - \nu^{*}_{t}(0, \infty)$. Then $(L^{N_{k}}, W)_{k \geq 1}$ converges weakly 
to $(L^{*}, W)$ on $(D_{\R}, \M) \times (C_{\R}, \mathrm{U})$. 
\end{prop}

\begin{proof}
For a contradiction suppose that the weak convergence does not hold. Since $t \mapsto L^{N}_{t}$ is increasing, 
$L^{N}_{t} \in [0,1]$ and we have Proposition~\ref{Prob_Prop_LowerLoss}, the conditions of 
\cite[Thm.~12.12.2]{whitt2002} are satisfied and so $(L^{N})_{N \geq 1}$ is tight on $(D_{\R}, \M)$, and 
because marginal tightness implies joint tightness, $(L^{N}, W)_{N \geq 1}$ is also tight. By taking a further 
subsequence if needed, assume that $(L^{N_{k}}, W)_{k \geq 1} \Rightarrow (L^{\dagger}, W)_{k \geq 1}$ 
for some $L^{\dagger} \in D_{\R}$. 

Notice from \cite[Thm.~12.4.1]{whitt2002} that the canonical time projection from $(D_{\R}, \M)$ to \R\ is only 
continuous at times for which its argument does not jump. That is, for every $t$, $\pi_{t}(x) := x_{t}$ is continuous 
at $x \in D_{\R}$ if and only if $x_{t-} = x_{t}$. To this end, define $\mathrm{cont}(L^{\dagger}) = 
\{s \in [0,T] : \P(L^{\dagger}_{s-} = L^{\dagger}_{s}) = 1\}$, which we know by \cite[Sec.~13]{billingsley1999} 
is cocountable in $[0,T]$. For $\lambda \in \N$ define $\phi_{\lambda} \in \S$ to be any function satisfying 
$\phi_{\lambda} = 1$ on $[-\lambda, \lambda]$, $\phi_{\lambda} = 0$ on $(-\infty, -2\lambda) \cup (2\lambda, \infty)$ 
and $\phi_{\lambda} \in (0,1)$ otherwise. By \cite[Prop.~2.7(i)]{ledger2015} $\nu^{N_{k}}(\phi_{\lambda}) 
\Rightarrow \nu^{*}(\phi_{\lambda})$, and define $\mathrm{cont}(\nu^{*}(\phi_{\lambda})) = \{s \in [0,T] : 
\P(\nu^{*}_{s-}(\phi_{\lambda}) = \nu^{*}_{s}(\phi_{\lambda})) = 1\}$. Take 
\[
\mathbb{T} := \mathrm{cont}(L^{\dagger})
	\cap \bigcap_{\lambda = 1}^{\infty} \mathrm{cont}(\nu^{*}(\phi_{\lambda})),
\]
which is cocountable (since it is the countable intersection of cocountable sets) and so is dense in $[0,T]$. 

Since $(L^{N_{k}}, W) \Rightarrow (L^{\dagger}, W)$ and $\mathbb{T}$ is dense in $[0,T]$, if $(L^{*}, W)$ 
and $(L^{\dagger}, W)$ are not equal in law on $(D_{\R}, \M)$, then it must be the case that not all of the 
finite-dimensional marginals of $L^{*}$ and $L^{\dagger}$ on $\mathbb{T}$ are equal in law. It is no loss of generality 
to assume that there exists $\e > 0$, $m \in \N$, $f_i, g_i : \R \to \R$ bounded and Lipschitz and $t_1,\dots,t_m \in \mathbb{T}$ such that
\[
\E \prod_{i=1}^{m} f_i(L^{*}_{t_i})g_i(W_{t_i}) + \e \leq  \limsup_{k \to \infty} \E \prod_{i=1}^{m} f_i(L^{N_{k}}_{t_i})g_i(W_{t_i}).
\]
By Proposition~\ref{Prob_Prop_Tail}
\[
\E|L^{N_{k}}_{t} - (1 - \nu^{N_{k}}_{t}(\phi_{\lambda}))|
	= O(e^{-\lambda}),
	\qquad \textrm{uniformly in } t \textrm{ and } N_{k},
\]
as $\lambda \to \infty$, therefore the Lipschitz property of $f_i$ gives
\[
\E \prod_{i=1}^{m} f_i(L^{*}_{t_i})g_i(W_{t_i}) + \e
	\leq \limsup_{k \to \infty} \E \prod_{i=1}^{m} f_i(1 - \nu^{N_{k}}_{t_i}(\phi_{\lambda}))g_i(W_{t_i})
	+ O(e^{-\lambda}),
\]
but $t_i \in \mathrm{cont}(\nu^{*}(\phi_{\lambda}))$, so
\[
\E \prod_{i=1}^{m} f_i(L^{*}_{t_i})g_i(W_{t_i}) + \e
	\leq \E \prod_{i=1}^{m} f_i(1 - \nu^{*}_{t_i}(\phi_{\lambda}))g_i(W_{t_i})
	+ O(e^{-\lambda}).
\]
Since $\nu^{*}_{t}$ is a probability measure $\nu^{*}_{t}(\phi_{\lambda}) \to \nu^{*}_{t}(\R) = 1 - L^{*}_{t}$ (recall 
from Proposition~\ref{Tight_Prop_Measure} that $\nu^{*}_{t}$ is supported on $[0,\infty)$), so taking $\lambda \to \infty$ 
gives the required contradiction. 
\end{proof}

We are now in a position to verify the first half of Theorem~\ref{Intro_Thm_Exist}, which is that any limit point must 
satisfy the regularity conditions from Assumption~\ref{Not_Def_RegCond}:

\begin{prop}[Regularity conditions]
\label{Tight_Prop_RegCond}
If $(\nu^{*}, W)$ realises a limiting law of $(\nu^{N}, W)_{N \geq 1}$, then $\nu^{*}$ 
satisfies Assumption~\ref{Not_Def_RegCond}.
\end{prop}

\begin{proof}
Firstly, $\nu^{*}$ takes values in the sub-probability measures by Proposition~\ref{Tight_Prop_Measure}, and that result
also gives Assumption~\ref{Not_Def_RegCond}~(\ref{Not_Def_RegCond2}). 

For conditions~(\ref{Not_Def_RegCond4}) and~(\ref{Not_Def_RegCond5}) of Assumption~\ref{Not_Def_RegCond}, 
let $I = (x, y) \subseteq \R$ be any finite open interval. For $\delta > 0$, take any $\phi_{\delta} \in \S$ satisfying 
$\phi_{\delta} = 1$ on $I$, $\phi_{\delta} = 0$ on $(-\infty, x- \delta)\cup(y + \delta, \infty)$ and $\phi_{\delta} \in (0,1)$ 
otherwise. Taking $(\nu^{N_{k}}, W) \Rightarrow (\nu^{*}, W)$ and noting that 
$\int^{t}_{0} \nu^{N_{k}}_{s}(\phi_{\lambda}) ds \Rightarrow \int^{t}_{0} \nu^{*}_{s}(\phi_{\lambda}) ds$ in \R\ 
by \cite[Thm.~11.5.1]{whitt2002} and that these integrals are uniformly bounded (by 
$T\Vert\phi_{\lambda} \Vert_{\infty} = T$), we have
\[
\E \int^{T}_{0} \nu^{*}_{t}(I)dt 
	\leq \E \int^{T}_{0} \nu^{*}_{t}(\phi_{\delta})dt
	= \lim_{k \to \infty} \E \int^{T}_{0} \nu^{N_{k}}_{t}(\phi_{\delta})dt.
\]
For both conditions~(\ref{Not_Def_RegCond4}) and~(\ref{Not_Def_RegCond5}) we have bounds on the right-hand 
side which are independent of $N_{k}$ (Propositions~\ref{Prob_Cor_SpatialConc} and~\ref{Prob_Prop_Boundary}), 
and then the conditions hold by sending $\delta \to 0$. For condition~(\ref{Not_Def_RegCond3}) we have $y = \infty$, 
so $\phi_{\delta} \notin \S$. However, for $I = (\lambda, \eta)$ with $\eta > 0$, the above work gives
\begin{align*}
\E \int^{T}_{0} \nu^{*}_{t}(\lambda, \eta)dt 
	&\leq \lim_{k \to \infty} \E \int^{T}_{0} \nu^{*}_{t}(\phi_{\lambda}) dt \\
	&\leq \liminf_{k \to \infty} \E \int^{T}_{0} \nu^{N_{k}}_{t}(\lambda - \delta, \eta + \delta)dt \\
	&= o(e^{-\alpha(\lambda - \delta)}),
\end{align*}
so sending $\delta \to 0$ and $\eta \to \infty$ (using the dominated convergence theorem) gives the result. 

It remains to show~(\ref{Not_Def_RegCond1}) of Assumption~\ref{Not_Def_RegCond}. First we prove that $L^{*}$ is 
non-decreasing. By \cite[Sec.~13]{billingsley1999} there is a (deterministic) cocountable set, $\mathbb{T}$, on 
which $(L^{N_{k}}_{t}, L^{N_{k}}_{s}) \Rightarrow (L^{*}_{t}, L^{*}_{s})$ in $\R \times \R$. So for $s < t$ in 
$\mathbb{T}$ \cite[Thm.~2.1]{billingsley1999} implies
\[
\P(L_{t}^{*} - L_{s}^{*} < 0) 
	\leq \liminf_{k \to \infty} \P(L_{t}^{N_{k}} - L_{s}^{N_{k}} < 0)
	= 0,
\]
and hence $L^{*}$ is non-decreasing on $\mathbb{T}$. But $\mathbb{T}$ is dense in $[0,T]$ and $L^{*}$ \cadlag, 
so we conclude $L^{*}$ is non-decreasing on $[0,T]$. To deduce the strict monotonicity, Proposition~\ref{Prob_Prop_LossInc} 
implies
\begin{multline*}
\P(L^{*}_{t} - L^{*}_{s} = 0, L^{*}_{s} < r)
	= \lim_{\delta} \P(L^{*}_{t} - L^{*}_{s} < \delta, L^{*}_{s} < r)
	\\\leq \limsup_{\delta \to 0} \limsup_{k \to \infty} \P(L^{N_{k}}_{t} - L^{N_{k}}_{s} < \delta, L^{N_{k}}_{s} < r) = 0,
\end{multline*}
whenever $r < 1$ and sending $r \uparrow 1$ gives the required result. 
\end{proof}

So far we have seen no reason why it is important $L^{*}$ should be strictly increasing whenever the mass in the system 
is not completely depleted ($L^{*} < 1$). The following result is such an example and shows why this condition is needed 
to pass to a weak limit. The result will be applied directly in the next subsection.

\begin{cor}[Weak convergence of integrals]
\label{Tight_Cor_Fiddly}
Fix $t \in [0,T]$ and $\phi \in \S$. Let $g = g(t, x, \ell)$ be equal to either $\mu(t, x, \ell)$, $\sigma(t,x)^{2}$ or 
$\sigma(t, x, \ell)\rho(t, \ell)$. Define $A$ to be all elements in $D_{\Sdual}$ that take values in the sub-probability 
measures and let $B = D_{[0,1]}\subseteq D_{\R}$. Then the map
\[
(\xi, \ell) \in A \times B
	\mapsto \int^{t}_{0} \xi_{s}(g(s, \cdot \, , \ell_{s}) \phi(\cdot)) ds \in \R
\]
is continuous (with respect to the product topology on $(D_{\Sdual}, \M) \times (D_{[0,1]}, \M)$) at all point $(\xi, \ell)$ 
which satisfy the conditions of Assumption~\ref{Not_Def_RegCond}. Consequently, if 
$(\nu^{N_{k}}, W) \Rightarrow (\nu^{*}, W)$ then  
\[
\int^{t}_{0} \nu^{N_{k}}(g(s, \cdot \, , L^{N_{k}}_{s})\phi(\cdot))ds
	\Rightarrow 
	\int^{t}_{0} \nu^{*}(g(s, \cdot \, , L^{*}_{s})\phi(\cdot))ds
	\qquad \textrm{on } \R.
\]
\end{cor}

\begin{proof}
For short-hand we will denote this map $\Psi: A \times B \to \R$. Suppose that $(\bar{\xi}, \bar{\ell}) \to (\xi, \ell)$ in 
$A \times B$, then 
\begin{multline}
\label{eq:FiddlyCor1}
|\Psi(\bar{\xi}, \bar{\ell}) - \Psi(\xi,\ell)|
	\leq \Big| \int^{t}_{0} \bar{\xi}_{s}(g(s, \cdot, \ell_{s}) \phi) ds - \int^{t}_{0} \xi_{s}(g(s, \cdot, \ell_{s})\phi) ds \Big|
	\\+ \int^{t}_{0} |\bar{\xi}_{s}(g(s, \cdot, \ell_{s})\phi - g(s, \cdot, \bar{\ell}_{s})\phi) | ds =: I + J.
\end{multline}
We will control $I$ and $J$ separately.

Begin by fixing $\e > 0$ and $\delta > 0$. Take $k = k(\delta) > 0$ sufficiently large so that $|g(s,x,\ell)\phi(x)| < \delta$ 
for all $s \in [0,T]$, $x \in \R \smallsetminus [-k, k]$ and $\ell \in [0,1]$, which is possible because $g$ is bounded and 
$\phi$ is rapidly decreasing. Let $\psi_{\e}$ be a mollifier and set $g^{\e}(s, x , \ell) := (g(s, \cdot, \ell) * \psi_{\e})(x) 
\in C^{\infty}(\R)$, then we have
\begin{multline*}
I \leq \Big| \int^{t}_{0} \bar{\xi}_{s}(g^{\e}(s, \cdot, \ell_{s}) \phi) ds - \int^{t}_{0} \xi_{s}(g^{\e}(s, \cdot, \ell_{s})\phi) ds \Big|  \\
	+ 2\int^{t}_{0} \sup_{x \in \R} |\phi(x)||g^{\e}(s,x,\ell_{s}) - g(s,x,\ell_{s})| ds.
\end{multline*}
Since $g^{\e}(s, \cdot, \ell) \in C^{\infty}(\R)$ and $\phi \in \S$, $g^{\e}(s, \cdot, \ell) \phi(\cdot) \in \S$ so the first term 
vanishes as $\bar{\xi} \to \xi$. We can then split the second term as
\begin{align*}
\limsup_{\bar{\xi} \to \xi}I 
	&\leq 2\Vert \phi \Vert_{\infty}\int^{t}_{0} \sup_{x \in [-2k, 2k]} |g^{\e}(s,x,\ell_{s}) - g(s,x,\ell_{s})| ds \\
& \qquad + 2c \int^{t}_{0} \sup_{x \in \R \smallsetminus [-2k,2k]} |\phi(x)| ds,
\end{align*}
and here the first term vanishes as $\e \to 0$ by \cite[App.~C, Thm.~6]{evans2010partial} since $[-k,k]$ is compact and 
the second term can be guaranteed to be less than $2\delta$ for $k$ sufficiently large. Taking $\delta \to 0$ gives $\limsup I = 0$. 

To deal with $J$ in (\ref{eq:FiddlyCor1}), first notice that since $\bar{\xi} \in A$ 
\[
J \leq \Vert \phi \Vert_{\infty} \int^{t}_{0} \sup_{x \in \R} |g(s, x, \ell_{s}) - g(s, x, \bar{\ell}_{s})| ds.
\]
Define $\mathbb{T}_{0} := \{ s \in [0,t] : \ell_{s} = \theta_{i} \textrm{ for some } i \in \{0,1,\dots,k\} \}$, where we 
recall Assumption~\ref{Not_Def_Coefficients} condition~\ref{Not_Def_Coefficients_5}. For $\delta > 0$, let
$\mathbb{T}_{0}^{\delta} :=\{ s \in [0,t] : \min_{0 \leq i \leq k}|\theta_{i} - \ell_s| < \delta\}$. Define $\mathbb{T}_{1}$ 
to be all $s \in [0,t]$ such that $\ell_{s} = \ell_{s-}$, which we know is a cocountable set~\cite[Cor.~12.2.1]{whitt2002}. 
For $s \in \mathbb{T}_{1}$, $\bar{\ell}_{s} \to \ell_{s}$ in \R, so if $s \in \mathbb{T}_{1} \smallsetminus
\mathbb{T}_{0}^{\delta}$ then eventually $\ell_{s}, \bar{\ell}_{s} \in [\theta_{i-1}, \theta_{i})$ for some $i \in \{1,2,\dots,k\}$,
whence $\sup_{x \in \R} |g(s, x, \ell_{s}) - g(s, x, \bar{\ell}_{s})| \to 0$ by Assumption~\ref{Not_Def_Coefficients} 
condition~(\ref{Not_Def_Coefficients_4}). We conclude
\[
\limsup_{\bar{\xi} \to \xi} J
	\leq c_{1} \int_{([0,T] \smallsetminus \mathbb{T}_{1}) \cup  \mathbb{T}_{0}^{\delta}} ds
	\leq c_{1} k \delta,
	\qquad \textrm{for every } \delta > 0,
\]
where $c_{1} > 0$ is a numerical constant due to Assumption~\ref{Not_Def_Coefficients}. This completes the result. 
\end{proof}

\subsection*{Martingale approach}

We complete this section and the proof of Theorem~\ref{Intro_Thm_Exist} by showing that the limit SPDE holds 
for a general limit point. For this we will use a martingale argument and we introduce three processes:

\begin{defn}[Martingale components]
\label{Tight_Def_MGComp}
For a fixed test function $\phi \in \C$, define the maps:
\begin{enumerate}[(i)]
\item $M^{\phi} : D_{\Sdual} \times D_{[0,1]} \to D_{\R}$,
\begin{multline*}
M^{\phi}(\xi, \ell)(t) 
  := \xi_{t}(\phi) 
  - \nu_{0}(\phi)
  - \int^{t}_{0} \xi_{s}(\mu(s, \cdot, \ell_{s})\dx\phi) ds \\
   \qquad - \frac{1}{2} \int^{t}_{0} \xi_{s}(\sigma^{2}(s, \cdot)\dxx\phi) ds,
\end{multline*}

\item $S^{\phi} : D_{\Sdual} \times D_{[0,1]} \to D_{\R}$,
\[
S^{\phi}(\xi, \ell)(t)
  := M^{\phi}(\xi, \ell)(t)^{2}
  - \int^{t}_{0} \xi_{s}(\sigma(s, \cdot) \rho(s, \ell_{s})\dx\phi)^{2} ds,
\]

\item $C^{\phi} : D_{\Sdual} \times D_{\R} \times C_{\R} \to D_{\R}$,
\[
C^{\phi}(\xi,\ell,w)(t)
  := M^{\phi}(\xi, \ell)(t) \cdot w(t)
  - \int^{t}_{0} \xi_{s}(\sigma(s, \cdot, \ell_{s})\rho(s, \ell_{s}) \dx\phi) ds
\]
\end{enumerate}
\end{defn}
These processes capture the dynamics of the limit SPDE:

\begin{lem}[Martingale approach]
\label{Tight_Lem_MGapp}
Let $W$ be a standard Brownian motion and let $\xi$ and $L_{t} = 1 -\xi_{t}(0,\infty)$ be random processes 
satisfying the conditions of Assumption~\ref{Not_Def_RegCond}. If 
\[
M^{\phi}(\xi, L), \qquad S^{\phi}(\xi, L) \qquad \textrm{and} \qquad C^{\phi}(\xi, L, W)
\]
are martingales for every $\phi \in \C$, then $\xi$, $L$ and $W$ satisfy the limit SPDE from Theorem~\ref{Intro_Thm_Exist}. 
\end{lem}

\begin{proof}
The hypothesis gives
\begin{gather*}
[M^{\phi}(\xi,L)]_{t} = \int_{0}^{t} \xi_{s}(\sigma(s, \cdot, L_{s}) \rho(s, L_{s}) \dx\phi)^{2} ds,
  \\
  [M^{\phi}(\xi,L),W]_{t} = \int^{t}_{0} \xi_{s}(\sigma(s, \cdot, L_{s}) \rho(s, L_{s})\dx \phi)ds,
\end{gather*}
hence 
\[
[M^{\phi}(\xi, L) - \int^{\cdot}_{0} \xi_{s}(\sigma(s, \cdot, L_{s}) \rho(s, L_{s})\dx\phi)dW_{s}]_{t} = 0,
\]
for every $t \in \timeInt$, which completes the proof.
\end{proof}

Our strategy is to take a limit in Proposition~\ref{Finite_Prop_Evolution} and apply weak convergence. First notice that we have:

\begin{lem}
\label{Tight_Lem_WeakConvHelper}
For every fixed $\phi \in \C$, there exists a deterministic cocountable subset of $[0,T]$ on which
\begin{gather*}
M^{\phi}(\nu^{N_{k}},L^{N_{k}})(t) \weak M^{\phi}(\nu^{*},L^{*})(t),
  \qquad
S^{\phi}(\nu^{N_{k}},L^{N_{k}})(t) \weak S^{\phi}(\nu^{*},L^{*})(t),  \\
C^{\phi}(\nu^{N_{k}},L^{N_{k}}, W)(t) \weak C^{\phi}(\nu^{*},L^{*}, W)(t) \qquad \textrm{in } \R.
\end{gather*}
Furthermore, these sequences are uniformly bounded (for fixed $\phi$).
\end{lem}

\begin{proof}
Note that all the above processes are uniformly bounded (for fixed $\phi$) since $\nu^{N}$ is a probability measure. 
The result then follows by Corollary~\ref{Tight_Cor_Fiddly}.
\end{proof}

\begin{prop}[Evolution equation]
\label{Tight_Prop_EvoEqn}
Suppose $(\nu^{N_{k}}, W) \Rightarrow (\nu^{*}, W)$. Then, for every $\phi \in \C$, the processes 
$M^{\phi}(\nu^{*}, L^{*})$, $S^{\phi}(\nu^{*}, L^{*})$ and $C^{\phi}(\nu^{*}, L^{*}, W)$ from
Definition~\ref{Tight_Def_MGComp} are martingales. Hence $\nu^{*}$ and $W$ satisfy the evolution 
equation from Theorem~\ref{Intro_Thm_Exist}. Furthermore, $\nu^{*}$ is continuous. 
\end{prop}

\begin{proof}
Fix $\phi \in \C$ and let $\mathbb{T}$ be the cocountable set of times on which we have the conclusion of 
Lemma~\ref{Tight_Lem_WeakConvHelper}. To show that $M^{\phi}(\nu^{*},L^{*})$ is a martingale, it is 
enough to show that, for any arbitrary $k \geq 1$, $s,t \in \mathbb{T}$, $s_{1}, \dots, s_{k} \in [0,s] \cap \mathbb{T}$ 
and $f_{1}, \dots, f_{k} : \R \to \R$ continuous and bounded, that the map defined by
\[
F(\xi, \ell) := 
  (M^{\phi}(\xi,\ell)(t) - M^{\phi}(\xi,\ell)(s))
  \prod_{i=1}^{k} f_{i}(M^{\phi}(\xi, \ell)(s_{i}))
\] 
satisfies $\E F(\nu^{*}, L^{*}) = 0$. By Lemma~\ref{Tight_Lem_WeakConvHelper} and the boundedness 
and continuity of the $f_{i}$'s
\[
\E F(\nu^{*}, L^{*}) = \lim_{k \to \infty} \E F(\nu^{N_{k}}, L^{N_{k}}).
\]
However, from Proposition~\ref{Finite_Prop_Evolution}, we have that $M^{\phi}(\nu^{N_{k}}, L^{N_{k}})$ is a 
martingale since
\begin{equation}
\label{eq:Exist_MartingaleIto}
M^{\phi}(\nu^{N_{k}}, L^{N_{k}})(t)
  = \int^{t}_{0} \nu^{N_{k}}(\sigma(s, \cdot, L^{N_{k}}_{s}) \rho(s, L^{N_{k}}) \phi) dW_{s} 
  + I^{N_{k}}_{t}(\phi),
\end{equation}
therefore $\E F(\nu^{N_{k}}, L^{N_{k}}) = 0$ and so $M^{\phi}(\nu^{*}, L^{*})$ is a martingale.

For $S^{\phi}$, define the map 
\[
G(\xi, \ell) := 
  (S^{\phi}(\xi,\ell)(t) - S^{\phi}(\xi,\ell)(s))
  \prod_{i=1}^{k} f_{i}(S^{\phi}(\xi, \ell)(s_{i})).
\] 
By applying It\^o's formula to (\ref{eq:Exist_MartingaleIto}), we have
\[
S^{\phi}(\nu^{N_{k}}, L^{N_{k}})(t)
  = S^{\phi}(\nu^{N_{k}}, L^{N_{k}})(0)
  + \textrm{martingale term}
  + 2 [I^{N_{k}}(\phi)]_{t}.
\]
So be the boundedness of the $f_{i}$ and Proposition~\ref{Finite_Prop_Vanish}
\[
\E G(\nu^{N_{k}}, L^{N_{k}}) = O(1/N_{k}),
\]
so $\E G(\nu^{*}, L^{*}) = 0$ and $S^{\phi}(\nu^{*}, L^{*})$ is a martingale. The work for $C^{\phi}$ follows 
similarly, so we omit it. The result is then complete by Lemma~\ref{Tight_Lem_MGapp}, and the continuity of 
$t \mapsto \nu^{*}_{t}$ follows by the fact that the right-hand side of the evolution equation in Theorem~\ref{Intro_Thm_Exist} 
is continuous.
\end{proof}

\section{The kernel smoothing method}
\label{Sect_Kernel}

The kernel smoothing method converts a measure into an approximating family of functions and, by 
establishing uniform results on the functions, enables us to show the existence of a density for 
the measure. In the next section we will use this to 
prove Theorem~\ref{Intro_Thm_Unique}. Let $\zeta$ be a finite signed-measure and $p_{\e}$ the Gaussian 
heat kernel
\[
p_{\e}(x) := (2 \pi \e)^{-1/2} \exp\{ -x^{2} / 2\e  \},
  \qquad x \in \R. 
\]
Begin by noting the familiar fact that $\zeta$ can be approximated by its convolution with $p_{\e}$: For every continuous and bounded $\phi : \R \to \R$
\begin{equation}
\label{eq:Kernel_BasicSmooth}
\int_{\R} \phi(x) (\zeta * p_{\e}) (x) dx \rightarrow \zeta(\phi) 
	= \int_{\R} \phi(x)\zeta(dx),
\end{equation}
as $\e \rightarrow 0$, and
\begin{equation}
\label{eq:KSM_Tbar}
\bar{T}_{\e} \zeta (x) := (p_{\e} * \zeta)(x) = \int_{\R} p_{\e}(x - y) \zeta(dy) 
\end{equation}
is a $C^{\infty}(\R)$ function. We will sometimes abuse notation and write $\bar{T}_{\e}\phi = p_{\e} * \phi$ 
when $\phi : \R \to \R$ is a function. With $(\cdot, \cdot)_{2}$ denoting the usual $L^{2}(\R)$ inner product, we have 
\begin{equation}
\label{eq:Unique_L2Switch}
(\phi, \bar{T}_{\e} \zeta)_{2} = \zeta(\bar{T}_{\e} \phi).
\end{equation}

Our first observation is that $\bar{T}_{\e}$ is a contraction on $L^{2}(\R)$:

\begin{prop}[Contraction]
\label{KSM_Prop_Contract}
Let $f \in L^{2}(\R)$. Then $\Vert \bar{T}_{\e}f  \Vert_{2}  \leq \Vert f  \Vert_{2}$, where $\Vert \cdot \Vert_{2}$ 
is the $L^{2}$ norm on \R. 
\end{prop}

\begin{proof}
The Cauchy--Schwarz inequality gives
\[
|\bar{T}_{\e}f(x)|^{2}
  = \Big| \int_{\R} p_{\e}(x - y) f(y) dy \Big|^{2}
  \leq \int_{\R} p_{\e}(x - y)dy 
  \cdot \int_{\R} p_{\e}(x - y) f(y)^{2} dy.
\]
The first integral on the right-hand side integrates to one, then integrating over $x \in \R$ completes the proof.
\end{proof}

We now give a condition which shows how to recover the existence of a density via kernel smoothing.

\begin{prop}
\label{KSM_Prop_liminfL2}
Suppose that $\zeta$ is a finite signed measure and 
\[
\liminf_{\e \to 0} \Vert \bar{T}_{\e}\zeta \Vert_{2} < \infty.
\]
Then $\zeta$ has an $L^{2}(\R)$ density, i.e.\ there exists $f \in  L^{2}(\R)$ such that $\zeta(\phi) = (f, \phi)_{2}$, 
for every $\phi \in L^{2}(\R)$. Furthermore $\Vert \bar{T}_{\e}\zeta \Vert_{2} \to \Vert f \Vert_{2}$ in \R.
\end{prop}

\begin{proof}
The hypothesis gives a bounded sequence $(\bar{T}_{\e_{n}} \zeta)_{n \geq 1}$ in $L^{2}(\R)$, with $\e_{n} \to 0$. 
By \cite[App.~D, Thm.~3]{evans2010partial}, we can extract a weakly convergent subsequence 
\[
(\bar{T}_{\e_{n_{k}}}, \phi)_{2} 
	\to (f, \phi)_{2},
	\qquad\textrm{for every } \phi \in L^{2}(\R), 
\] 
for some $f \in L^{2}(\R)$. But by (\ref{eq:Kernel_BasicSmooth}) we conclude that $\zeta(\phi) = (f, \phi)_{2}$ 
for all $\phi \in \S$, and this gives the first result since \S\ is dense in $L^{2}(\R)$. 

We now have that $\bar{T}_{\e} \zeta = \bar{T}_{\e} f$, therefore by Proposition~\ref{KSM_Prop_Contract} 
\[\limsup_{\e \to 0} \Vert \bar{T}_{\e} \zeta \Vert_{2} \leq \Vert f  \Vert_{2}.\] By (\ref{eq:Kernel_BasicSmooth}) we also have
\[
|(f, \phi)_{2}|
	= \lim_{\e \to 0} |(\bar{T}_{\e} \zeta, \phi)_{2}|
	\leq \liminf_{\e \to 0} \Vert \bar{T}_{\e} \zeta \Vert_{2} \Vert \phi \Vert_{2},
		\qquad\textrm{for all } \phi \in\S,
\]
so $ \Vert f  \Vert_{2} \leq \liminf_{\e \to 0} \Vert \bar{T}_{\e} \zeta \Vert_{2}$, which completes the proof. 
\end{proof}

\subsection*{Smoothing in $H^{-1}$ and the anti-derivative}

The material above will be used to establish a preliminary regularity result (Proposition~\ref{Unique_Prop_L2Reg}) in 
Section~\ref{Sect_Uniqueness}. However, for the main uniqueness proof we will work in  a space of lower regularity and on 
the half-line. Recall that the first Sobolev space with Dirichlet boundary condition, $H^{1}_{0}(0,\infty)$, is defined to be the 
closure of $C^{\infty}_{0}(0, \infty)$ under the norm
\[
\Vert f \Vert_{H^{1}(0,\infty)}
	:= (\Vert f \Vert_{L^{2}(0,\infty)}^{2} + \Vert \partial_{x}f \Vert_{L^{2}(0,\infty)}^{2})^{1/2}.
\]
The dual of $H^{1}_{0}(0,\infty)$ will be denoted by $H^{-1}$ and its norm by
\[
\Vert \zeta \Vert_{-1}
	:= \sup_{\Vert \phi \Vert_{H^{1}(0,\infty)} = 1} |\zeta(\phi)|.
\]
This is a natural space for us to work in due to the following.

\begin{prop}
\label{KSM_Prop_SignedMeasure}
If $\zeta$ is a finite signed measure, then $\zeta \in H^{-1}$. 
\end{prop}

\begin{proof}
First observe that $|\zeta(\phi)| \leq |\zeta| \Vert \phi \Vert_{\infty}$, for every $\phi \in C^{\infty}_{0}(0,\infty)$. 
Morrey's inequality \cite[Sec.~5.6, Thm.~4]{evans2010partial} gives a universal constant, $C > 0$, such that 
$\Vert \phi \Vert_{\infty} \leq C \Vert \phi \Vert_{H^{1}}$, and this completes the proof. 
\end{proof}

To work on the half-line we will use the absorbing heat kernel defined, as in the proof of Proposition~\ref{Prob_Prop_Boundary}, by 
\begin{equation}
\label{eq:KSM_Absorbing}
G_{\e}(x, y) := p_{\e}(x - y) - p_{\e}(x + y),
	\qquad \textrm{for } x,y> 0
\end{equation}
and define 
\[
T_{\e} \zeta (x) 
	:= \int^{\infty}_{0} G_{\e}(x,y) \zeta(dy).
\]
Notice that $G_{\e}(x, 0) = 0$ for every $x$, so $y \mapsto G_{\e}(x, y)$ is an element of \C, and also notice 
that $T_{\e}\zeta(0) = 0$. For $T_{\e}\zeta$ to approximate $\zeta$, we need $\zeta$ to be supported on $[0,\infty)$:

\begin{prop}
\label{KSM_BasicSmooth}
If $\zeta$ is supported on $[0,\infty)$, then 
\[
(T_{\e} \zeta, \phi)_{2}
	 \to \zeta(\phi), 
\]
as $\e \to 0$, for every $\phi$ continuous, bounded and supported on $(0,\infty)$: 
\end{prop}
\begin{proof}
Let $\tilde{\phi}(x):=\phi(-x)$, then from (\ref{eq:Kernel_BasicSmooth})
\[
(T_{\e}\zeta, \phi)_2
	= (\bar{T}_{\e}\zeta, \phi)_2
	- (\bar{T}_{\e}\zeta, \tilde{\phi})_2
	\to \zeta(\phi)
	- \zeta(\tilde{\phi}).
\]
But by the hypotheses $\zeta(\tilde{\phi}) = 0$, as required. 
\end{proof}

To access the $H^{-1}$ norm, we will use the \emph{anti-derivative} defined by 
\[
\anti f(x) := -\int^{\infty}_{x} f(y)dy,
	\qquad \textrm{for } f:\R \to \R \textrm{ integrable}.
\]
Notice that $\dx \anti f = f$, and if $\dx f$ is also integrable, then $\anti \dx f = f$ too. The result we will use 
in Section~\ref{Sect_Uniqueness} is the following.

\begin{prop}
\label{KSM_Prop_liminfH-1}
If $\zeta \in H^{-1}$, then $\Vert \zeta \Vert_{-1} \leq \liminf_{\e \to 0} \Vert \anti T_{\e} \zeta \Vert_{L^{2}(0,\infty)}$. 
\end{prop}

\begin{proof}
First notice that for fixed \e\
\[
\int^{\infty}_{0} \int^{\infty}_{0} (p_{\e}(x - y) + p_{\e}(x + y)) dx |\zeta|(dy) < \infty,
\]
so $T_{\e} \zeta$ is integrable and hence $\anti T_{\e} \zeta$ is well-defined. Integration by parts gives
\[
(\anti T_{\e} \zeta, \dx \phi)_{L^{2}(0,\infty)} 
	= (T_{\e} \zeta, \phi)_{L^{2}(0,\infty)}
	= \zeta(T_{\e}\phi),
\]
for $\phi \in C^{\infty}(0,\infty)$. Therefore by Proposition~\ref{KSM_BasicSmooth} we have 
\begin{align*}
|\zeta(\phi)|
	&= \lim_{\e \to 0} |(\anti T_{\e} \zeta, \dx \phi)_{2}| \\
	&\leq \liminf_{\e \to 0} \Vert \anti T_{\e} \zeta \Vert_{2} \Vert \phi \Vert_{2} \\
	&\leq \liminf_{\e \to 0} \Vert \anti T_{\e} \zeta \Vert_{2} \Vert \phi \Vert_{H^{1}},
\end{align*}
which gives the result.  
\end{proof} 

\section{Uniqueness of solutions; Proof of Theorem \ref{Intro_Thm_Unique}}
\label{Sect_Uniqueness}

In this section we will prove Theorem~\ref{Intro_Thm_Unique}. Therefore take $\nu$, $\tilde{\nu}$ and $W$ as in the 
statement with $(\nu^{N_{k}}, W)_{k \geq 1} \Rightarrow (\nu, W)$ along some subsequence. Let $L_{t} = 1 - \nu_{t}(0, \infty)$ 
and $\tilde{L}_{t} = 1 - \tilde{\nu}_{t}(0, \infty)$. The first step will be to show that $\nu$ has some $L^{2}$ regularity 
(Proposition~\ref{Unique_Prop_L2Reg}), which is due to a comparison with $\bar{\nu}^{N_{k}}$ from (\ref{eq:Finite_nub}) 
and from the dynamics of Proposition~\ref{Finite_Prop_EvoNub}. We then use this fact, along with energy estimates in 
$H^{-1}$, to complete the proof. Several technical lemmas are used throughout this section, however, to aid readability, 
their full statements and proofs are deferred until Section~\ref{Sect_Lemmas}.

\subsection*{$L^{2}$-regularity} 

The result we will prove in this subsection is the following:

\begin{prop}[$L^{2}$-regularity]
\label{Unique_Prop_L2Reg}
With $\nu$ as introduced at the start of Section~\ref{Sect_Uniqueness}, 
\[
\sup_{s \in [0,T]} \sup_{\e > 0} \Vert T_{\e}\nu_{s} \Vert_{2}^{2} < \infty, 
	\qquad \textrm{with probability 1}.
\]
\end{prop} 

We would like to work with some process $\bar{\nu}$ defined analogously to (\ref{eq:Finite_nub}) that would satisfy 
the bound $\nu_{t}(S) \leq \bar{\nu}_{t}(S)$, for every $t\in[0,T]$ and $S \subseteq \R$. At this stage, however, we 
are dealing only with weak limit points, so must recover the required process through a limiting procedure on 
$(\bar{\nu}^{N})_{N \geq 1}$:

\begin{lem}[Whole space SPDE]
\label{Unique_Lem_WholeSpace}
On a sufficiently rich probability space, there exists $(\nu^{*}, \bar{\nu}^{*}, W)$ such that $(\nu^{*}, W)$ is equal in 
law to $(\nu, W)$, ${\nu}^{*}_{t}(S) \leq \bar{\nu}^{*}_{t}(S)$, for every $t \in [0,T]$ and $S\subseteq \R$, and 
$\bar{\nu}^{*}$ satisfies the limit SPDE on the whole space: 
\begin{align*}
\bar{\nu}^{*}_{t}(\phi) 
	&= \nu_{0}(\phi)
	+ \int^{t}_{0} \bar{\nu}^{*}_{s}( \mu(s, \cdot, L_{s}) \dx \phi ) ds
	+ \frac{1}{2} \int^{t}_{0} \bar{\nu}^{*}_{s}( \sigma^{2}(s, \cdot) \dxx \phi ) ds \\  
	&\quad + \int_{0}^{t} \bar{\nu}^{*}_{s}( \sigma(s, \cdot)\rho(s, L_{s}) \dx \phi ) dW_{s},
	\qquad \textrm{with \ \ } L^{*}_{t} = 1 - \nu^{*}_{t}(0,\infty),
\end{align*}
for every $t \in [0,T]$ and $\phi \in \S$, together with condition~(\ref{Not_Def_RegCond5}) of 
Assumption~\ref{Not_Def_RegCond} and the two-sided tail bound
\[
\E \bar{\nu}^{*}_{t}((-\infty, -\lambda) \cup (\lambda, \infty)) = o(e^{-\alpha \lambda}),
	\qquad \textrm{as } \lambda \to +\infty,
\]
for every $\alpha > 0$.
\end{lem}

\begin{proof}
Notice that in Proposition~\ref{Tight_prop_tight} we have carried out sufficient work to prove $(\bar{\nu}^{N})_{N \geq 1}$ 
is tight on $(D_{\Sdual}, \M)$, hence $(\nu^{N}, \bar{\nu}^{N}, W)_{N \geq 1}$ is tight. We can therefore conclude that 
there is a subsequence $(N_{k_{r}})_{r\geq 1}$ for which $(\nu^{N_{k_{r}}}, \bar{\nu}^{N_{k_{r}}}, W)_{r \geq 1}$ 
converges in law. Any realisation of this limit must have a marginal law that agrees with the law of $(\nu, W)$. As the work 
in Propositions~\ref{Tight_Prop_Measure} and~\ref{Tight_Prop_EvoEqn} is unchanged for $\bar{\nu}^{N}$ in place of 
$\nu^{N}$, we conclude that $\bar{\nu}^{*}$ is probability-measure-valued and, due to Proposition~\ref{Finite_Prop_EvoNub}, 
that $\bar{\nu}^{*}$ satisfies the limit SPDE on the whole space. Finally, we note that for every $\phi \in \S$ with $\phi \geq 0$ 
we have $\nu^{N_{k_{r}}}_{t}(\phi) \leq \bar{\nu}^{N_{k_{r}}}_{t}(\phi)$, therefore 
\[
\P(\nu^{*}_{t}(\phi) > \bar{\nu}^{*}_{t}(\phi) )
	\leq \liminf_{r \to \infty} \P(\nu^{N_{k_{r}}}_{t}(\phi) > \bar{\nu}^{N_{k_{r}}}_{t}(\phi) ) = 0,
\]
for every $\phi \in \S, \phi \geq 0$, by \cite[Thm.~2.1]{billingsley1999}. This inequality holds for all $t$ by the continuity of $\nu^{*}$ and $\bar{\nu}^{*}$ 
(which follows from being solutions to the limit SPDE) and suffices to give the required dominance.  
Condition~(\ref{Not_Def_RegCond5}) of Assumption~\ref{Not_Def_RegCond} is satisfied by $\bar{\nu}^{*}$ because the 
proof of Corollary~\ref{Prob_Cor_SpatialConc} uses only the behaviour of $\bar{\nu}^{N}$. Likewise, the two-sided tail 
estimate is satisfied due to the same work as in Proposition~\ref{Prob_Prop_Tail}.   
\end{proof}

Our strategy is to use the kernel smoothing method with $L^{2}$-energy estimates on the SPDE satisfied by $\bar{\nu}^{*}$. 
This is possible because we do not have to take boundary effects into account, which is the main difficulty in the uniqueness 
proof that will follow. The following lemma relates $\bar{\nu}^{*}$ to Proposition~\ref{Unique_Prop_L2Reg}.

\begin{lem}
With $\nu$ and $\bar{\nu}^{*}$ as above and $\bar{T}_{\e}$ as in (\ref{eq:KSM_Tbar}), if 
\[
\liminf_{\e \to \infty} \E [\sup_{s \in [0,T]}  \Vert \bar{T}_{\e}\bar{\nu}^{*}_{s} \Vert_{2}^{2}\,] < \infty
\]
then Proposition \ref{Unique_Prop_L2Reg} holds. 
\end{lem}

\begin{proof}
Since $\nu^{*} \leq \bar{\nu}^{*}$, $\liminf_{\e \to \infty} \E [\sup_{s \in [0,T]}  \Vert \bar{T}_{\e}\nu^{*}_{s} 
\Vert_{2}^{2}\,] < \infty$. We would first like to deduce that this fact also holds for $\bar{T}_{\e} \nu$, but since 
the map $\nu_{t} \mapsto \Vert \bar{T}_{\e} \nu_{t} \Vert_{2}$ might not be continuous on \Sdual, more care must be taken.

By fixing $\{ \phi_{i} \}_{i \geq 1}$ to be the Haar basis of $L^{2}(\R)$ we have
\begin{align}
\label{eq:Unique_Lem_L2_1}
\E \sup_{t \in [0,T]} \Vert \bar{T}_{\e} \nu_{t} \Vert_{2}^{2}
	&= \E \sup_{t \in [0,T]} \lim_{k \to \infty} \sum_{i = 1}^{k} (\bar{T}_{\e} \nu_{t} , \phi_{i})^{2}_{2} \\
	&\leq \liminf_{k \to \infty} \E \sup_{t \in [0,T]} \sum_{i = 1}^{k} \nu_{t}(\bar{T}_{\e} \phi_{i})^{2}. \nonumber
\end{align}
by (\ref{eq:Unique_L2Switch}) and Fatou's Lemma. Since each $\phi_{i}$ is compactly supported, we have that 
$\bar{T}_{\e} \phi_{i} \in \S$, therefore $\nu_{t}(\bar{T}_{\e} \phi_{i})$ is equal in law to $\nu^{*}_{t}(\bar{T}_{\e} 
\phi_{i})$, so by \cite[Lem.~13.4.1]{whitt2002}
\[
\sup_{t \in [0,T]} \sum_{i = 1}^{k} \nu_{t}(\bar{T}_{\e} \phi_{i})^{2}
	 =_{\textrm{law}} \sup_{t \in [0,T]} \sum_{i = 1}^{k} \nu^{*}_{t}(\bar{T}_{\e} \phi_{i})^{2}.
\]
Returning to (\ref{eq:Unique_Lem_L2_1}), we now have that
\[
\E \sup_{t \in [0,T]} \Vert \bar{T}_{\e} \nu_{t} \Vert_{2}^{2}
	\leq \liminf_{k \to \infty} \E \sup_{t \in [0,T]} \sum_{i = 1}^{k} \nu^{*}_{t}(\bar{T}_{\e} \phi_{i})^{2}
	\leq \E \sup_{t \in [0,T]} \Vert \bar{T}_{\e} \nu^{*}_{t} \Vert_{2}^{2}.
\]
By noting that $0 \leq T_{\e} \nu_{t} \leq \bar{T}_{\e} \nu_{t}$ and applying Fatou's Lemma once more we arrive at:
\begin{align*}
\E [ \, \liminf_{\e \to \infty} \sup_{s \in [0,T]} \Vert T_{\e}\nu_{s} \Vert_{2}^{2} \,]
	&\leq \E [ \, \liminf_{\e \to \infty} \sup_{s \in [0,T]} \Vert \bar{T}_{\e}\nu_{s} \Vert_{2}^{2} \,] \\
	&\leq \liminf_{\e \to \infty} \E\sup_{t \in [0,T]} \Vert \bar{T}_{\e} \nu^{*}_{t} \Vert_{2}^{2}
	 < \infty.
\end{align*}

We now have that $ \liminf_{\e \to \infty} \Vert T_{\e}\nu_{s} \Vert_{2} < \infty$, for every $s \in [0,T]$, with probability 1. 
Proposition~\ref{KSM_Prop_liminfL2} implies that $\nu_{t}$ has an $L^{2}(\R)$-density, $V_{t}$, for every $t$ and that 
\[
 \Vert V_{s} \Vert_{2} 
	\leq \liminf_{\e \to 0} \Vert T_{\e}\nu_{s} \Vert_{2}
	\leq \liminf_{\e \to \infty} \sup_{s \in [0,T]} \Vert T_{\e}\nu_{s} \Vert_{2},
\] 
therefore $\sup_{s \in [0,T]} \Vert V_{s} \Vert_{2} < \infty$, with probability 1. Then by Proposition~\ref{KSM_Prop_Contract}
\[
\sup_{s \in [0,T]} \sup_{\e > 0} \Vert T_{\e} \nu_{s} \Vert_{2} 
	\leq \sup_{s \in [0,T]}  \Vert  V_{s} \Vert_{2}  
	< \infty,
\]
almost surely, as required. 
\end{proof}

As an immediate consequence of the final part of the previous proof and of the forthcoming proof of Proposition \ref{Unique_Prop_L2Reg}, we have the existence of a density process for $\nu$:

\begin{cor}[$L^{2}(\R)$-regularity]
\label{Unique_Cor_L2reg}
With probability 1, for every $t \in [0,T]$ there exists $V_{t} \in L^{2}(\R)$ such that $V_{t}$ is supported on $[0,\infty)$ 
and is a density of $\nu_{t}$, i.e.~
\[
\nu_{t}(\phi) = \int_{0}^{\infty} \phi(x) V_{t}(x) dx,
	\qquad \textrm{for every } \phi \in L^{2}(\R). 
\]
Furthermore $\sup_{t \in [0,T]} \Vert V_{t} \Vert_{2} < \infty$, with probability 1. 
\end{cor}

\begin{rem}
\label{Unique_Rem_Crude}
We might hope that this argument could be used to prove uniqueness. However, notice that we have no control over 
$\nu - \tilde{\nu}$, as all we have are upper bounds on solutions. 
\end{rem}

\begin{proof}[Proof of Proposition \ref{Unique_Prop_L2Reg}]
Fix $x \in \R$ and set the function $y \mapsto p_{\e}(x - y) \in \S$ into the SPDE from Lemma \ref{Unique_Lem_WholeSpace} to get
\begin{align*}
d\bar{T}_{\e}\bar{\nu}^{*}_{t}(x)
	&= \bar{\nu}^{*}_{t}(\mu_{t}(y)\partial_{y} p_{\e}(x - y))dt 
	+ \frac{1}{2} \bar{\nu}^{*}_{t}(\sigma_{t}(y)^{2}\partial_{yy} p_{\e}(x - y))dt  \\
	& \qquad + \bar{\nu}^{*}_{t}(\sigma_{t}(y)\rho_{t} \partial_{y} p_{\e}(x - y)) dW_{t}  \\
	&= -\dx\bar{\nu}^{*}_{t}(\mu_{t}p_{\e}(x - \cdot))dt
	+ \frac{1}{2} \dxx \bar{\nu}^{*}_{t}(\sigma^{2}_{t} p_{\e}(x - \cdot)) dt \\
	& \qquad - \rho_{t}\dx \bar{\nu}^{*}_{t}(\sigma_{t} p_{\e}(x - \cdot))dW_{t},
\end{align*}
with the short-hand from Remark~\ref{Finite_Rem_ShortHand}. We would like to move the diffusion coefficients out of 
the integral against $\bar{\nu}^{*}$, and to do so we use Lemma~\ref{Lemmas_Lem_Switch2ndOrder}:
\begin{align*}
d\bar{T}_{\e}\bar{\nu}^{*}_{t}
	&= -(\mu_{t} \dx \bar{T}_{\e}\bar{\nu}^{*}_{t}
	- \dx \mu_{t} \bar{\mathcal{H}}^{\mu}_{t, \e}
	+ \bar{\mathcal{E}}^{\mu}_{t,\e} ) dt \\
	& \qquad + \frac{1}{2} \dx (\sigma^{2}_{t} \dx \bar{T}_{\e}\bar{\nu}^{*}_{t} 
	- \dx \sigma_{t}^{2} \bar{\mathcal{H}}^{\sigma^{2}}_{t, \e}
	+ \bar{\mathcal{E}}^{\sigma^{2}}_{t,\e} ) dt  \\
	&\qquad
	-\rho_{t}(\sigma_{t} \dx \bar{T}_{\e}\bar{\nu}^{*}_{t}
	- \dx \sigma_{t} \bar{\mathcal{H}}^{\sigma}_{t, \e}
	+ \bar{\mathcal{E}}^{\sigma}_{t,\e} ) dW_{t},
\end{align*}
where $\bar{\mathcal{H}}$ is as defined in Lemma~\ref{Lemmas_Lem_Switch2ndOrder} and the dependence 
on $x$ is omitted for clarity. Applying It\^o's formula to $(\bar{T}_{\e}\bar{\nu}^{*}_{t}(x))^{2}$ gives
\begin{align*}
d(\bar{T}_{\e}\bar{\nu}^{*}_{t})^{2}
	&= -2\bar{T}_{\e}\bar{\nu}^{*}_{t}(\mu_{t} \dx \bar{T}_{\e}\bar{\nu}^{*}_{t}
	- \dx \mu_{t} \bar{\mathcal{H}}^{\mu}_{t, \e}
	+ \bar{\mathcal{E}}^{\mu}_{t,\e} ) dt \\
	& \qquad + \bar{T}_{\e}\bar{\nu}^{*}_{t} \dx (\sigma^{2}_{t} \dx \bar{T}_{\e}\bar{\nu}^{*}_{t} 
	- \dx \sigma_{t}^{2} \bar{\mathcal{H}}^{\sigma^{2}}_{t, \e}
	+ \bar{\mathcal{E}}^{\sigma^{2}}_{t,\e} ) dt  \\
	&\qquad
	-2\rho_{t}\bar{T}_{\e}\bar{\nu}^{*}_{t}(\sigma_{t} \dx \bar{T}_{\e}\bar{\nu}^{*}_{t}
	- \dx \sigma_{t} \bar{\mathcal{H}}^{\sigma}_{t, \e}
	+ \bar{\mathcal{E}}^{\sigma}_{t,\e} ) dW_{t} \\
	& \qquad + \rho_{t}^{2}(\sigma_{t} \dx \bar{T}_{\e}\bar{\nu}^{*}_{t}
	- \dx \sigma_{t} \bar{\mathcal{H}}^{\sigma}_{t, \e}
	+ \bar{\mathcal{E}}^{\sigma}_{t,\e} )^{2} dt.
\end{align*}

Our strategy is to integrate over $x \in \R$, take a supremum over $t \in [0,T]$ and then take an expectation over 
the previous equation. For the first task we appeal to Lemma~\ref{Lemmas_Lem_Switch2ndOrder},
Lemma~\ref{Lemmas_Lem_StochasticFubini} and Young's inequality with free parameter $\eta > 0$ to obtain 
\begin{align*}
 \big\Vert \bar{T}_{\e} \nub_{t}^{*} \big\Vert^{2}_{2}
  &\leq \big\Vert \bar{T}_{\e} \nu_{0} \big\Vert^{2}_{2}
   + c_{\eta} \int^{t}_{0} \big\Vert \bar{T}_{\e} \nub_{s}^{*} \big\Vert_{2}^{2}ds + c_{\eta} \int_{0}^{t} 
   \big\Vert \bar{T}_{2\e} \nub_{s}^{*} \big\Vert_{2}^{2}ds  \\
   & \qquad + c_{\eta} \int_{0}^{t} \big\Vert \bar{\mathcal{E}}_{s, \e}^{\mu} \big\Vert_{2}^{2} + \big\Vert 
   \bar{\mathcal{E}}_{s, \e}^{\sigma^{2}} \big\Vert_{2}^{2} + \big\Vert \bar{\mathcal{E}}_{s, \e}^{\sigma} 
   \bigr\Vert_{2}^{2} ds  \\
  & \qquad - \int_{0}^{t} \int_{\R} [\,\sigma^{2}_{s} \cdot(1 -(1 + \eta)\rho_{s}^{2}) -\eta -\eta \mu_{s}^{2}\,] 
  (\dx \bar{T}_{\e} \nub^{*}_{s})^{2}dx ds   \\
  & \qquad - 2 \int_{0}^{t} \int_{\R} \rho_{s} \bar{T}_{\e} \nub_{s}^{*} (\sigma_{s} \dx \bar{T}_{\e} \nub_{s}^{*} 
  + \dx \sigma_{s} \bar{\mathcal{H}}_{s,\e} + \bar{\mathcal{E}}^{\sigma}_{s, \e} ) dx dW_{s}  
\end{align*}
where $c_{\eta} > 0$ is a constant depending only on $\eta$. Considering the third line, by Assumption~\ref{Not_Def_Coefficients} 
it is possible to choose $\eta > 0$ small enough so that
\begin{equation}
\label{eq:Unique_RhoLess1}
\sigma^{2}_{s}(x)(1 -(1 + \eta)\rho_{s}^{2}) -\eta -\eta \mu_{s}(x)^{2} \geq 0,
	\qquad \textrm{for all } x \in \R, s \in [0,T],
\end{equation}
therefore 
\begin{align*}
 \big\Vert \bar{T}_{\e} \nub_{t}^{*} \big\Vert^{2}_{2}
  &\leq \big\Vert \bar{T}_{\e} \nu_{0} \big\Vert^{2}_{2}
   + c_{\eta} \int^{t}_{0} \big\Vert \bar{T}_{\e} \nub_{s}^{*} \big\Vert_{2}^{2}ds + c_{\eta} \int_{0}^{t} 
   \big\Vert \bar{T}_{2\e} \nub_{s}^{*} \big\Vert_{2}^{2}ds  \\
   & \qquad + c_{\eta} \int_{0}^{t} \big( \big\Vert \bar{\mathcal{E}}_{s, \e}^{\mu} \big\Vert_{2}^{2} + 
   \big\Vert \bar{\mathcal{E}}_{s, \e}^{\sigma^{2}} \big\Vert_{2}^{2} + 
   \big\Vert \bar{\mathcal{E}}_{s, \e}^{\sigma} \bigr\Vert_{2}^{2} \big) ds  \\
  & \qquad - 2 \int_{0}^{t} \int_{\R} \rho_{s} \bar{T}_{\e} \nub_{s}^{*} (\sigma_{s} \dx \bar{T}_{\e} \nub_{s}^{*} + 
  \dx \sigma_{s} \bar{\mathcal{H}}_{s,\e} + \bar{\mathcal{E}}^{\sigma}_{s, \e} ) dx dW_{s}.
\end{align*}
Using Lemma~\ref{Lemmas_Lem_B--D--G} to take a supremum over $t$ and then expectation gives
\begin{multline*}
 \E \sup_{s \in [0,t]}\big\Vert \bar{T}_{\e} \nub_{s}^{*} \big\Vert^{2}_{2}
  \leq \big\Vert \bar{T}_{\e} \nu_{0} \big\Vert^{2}_{2}
   + c_{1} \E \int^{t}_{0} \big\Vert \bar{T}_{\e} \nub_{s}^{*} \big\Vert_{2}^{2}ds + c_{1} \E \int_{0}^{t} 
   \big\Vert \bar{T}_{2\e} \nub_{s}^{*} \big\Vert_{2}^{2}ds  \\
    \qquad + c_{1} \E \int_{0}^{t} \big( \big\Vert \bar{\mathcal{E}}_{s, \e}^{\mu} \big\Vert_{2}^{2} + 
    \big\Vert \bar{\mathcal{E}}_{s, \e}^{\sigma^{2}} \big\Vert_{2}^{2} + 
    \big\Vert \bar{\mathcal{E}}_{s, \e}^{\sigma} \bigr\Vert_{2}^{2} \big) ds,
\end{multline*}
where $c_{1} > 0$ is a numerical constant. 

Taking $\liminf$ as $\e \to 0$ over the previous inequality and applying Proposition~\ref{KSM_Prop_Contract} 
(to $V_{0} \in L^{2}$) and Lemma~\ref{Lemmas_Lem_Switch2ndOrder} yields
\begin{align*}
f(t) &:= \liminf_{\e \to 0} \E \sup_{s \in [0,t]} \Vert \bar{T}_{\e} \bar{\nu}^{*}_{s} \Vert_{2}^{2} \\
	&\leq c_{1} \Vert V_{0} \Vert_{2}^{2}
	+ 2c_{1} \liminf_{\e \to 0} \E \int^{t}_{0} \Vert \bar{T}_{\e} \bar{\nu}^{*}_{s} \Vert_{2}^{2} ds \\
	&\leq c_{1} \Vert V_{0} \Vert_{2}^{2}
	+ 2c_{1}tf(t).
\end{align*}
Hence for $t < 1/4c_{1}$ we have $f(t) \leq 2c_{1} \Vert V_{0} \Vert_{2}^{2}$. The proof is completed by 
propagating the argument onto $[1/4c_{1}, 2/4c_{1}]$ by the same work as above but started from $s = 1/4c_{1}$, 
rather than $s = 0$. This gives
\begin{multline*}
\liminf_{\e \to 0} \E \,[ \sup_{s \in [(4c_{1})^{-1}, 2(4c_{1})^{-1}]} \Vert \bar{T}_{\e} \bar{\nu}^{*}_{s} \Vert_{2}^{2}\,] \\
	\leq 2c_{1} \liminf_{\e \to 0} \E \,[ \sup_{s \in [0, (4c_{1})^{-1}]} \Vert \bar{T}_{\e} \bar{\nu}^{*}_{s} \Vert_{2}^{2}\,]
	\leq (2c_{1})^{2},
\end{multline*}
and so in general
\[
\liminf_{\e \to 0} \E \,[ \sup_{s \in [k(4c_{1})^{-1}, (k+1)(4c_{1})^{-1}]} \Vert \bar{T}_{\e} \bar{\nu}^{*}_{s} \Vert_{2}^{2}\,]
	\leq (2c_{1})^{k+1},
	\qquad \textrm{for } k \geq 0.
\]
Since the largest such $k$ we need to take is $k_{0} := 4c_{1}T$, the simple bound
\[
f(T) \leq \liminf_{\e \to 0} \E  \sum_{k = 0}^{k_{0} - 1}  \sup_{s \in [k(4c_{1})^{-1}, (k+1)(4c_{1})^{-1}]} 
	\Vert \bar{T}_{\e} \bar{\nu}^{*}_{s} \Vert_{2}^{2}
	\leq \sum_{k = 0}^{k_{0} - 1} (2c_{1})^{k + 1}< \infty
\]
completes the proof.
\end{proof}

\subsection*{Resuming the uniqueness proof}

Returning to proof of Theorem~\ref{Intro_Thm_Unique}, notice that for a fixed $x > 0$, the function $y \mapsto G_{\e}(x,y)$ 
from (\ref{eq:KSM_Absorbing}) is an element of \C. Setting into the SPDE for $\nu$ gives
\begin{multline*}
d \nu_{t}(G_{\e}(x, \cdot ))
  = \nu_{t}(\mu_{t} \partial_{y} G_{\e}(x, \cdot )) dt
  + \frac{1}{2} \nu_{t}(\sigma^{2}_{t} \partial_{yy}  G_{\e}(x, \cdot )) dt\\
  + \rho_{t} \nu_{t}(\sigma_{t} \partial_{y}  G_{\e}(x, \cdot )) dW_{t},
\end{multline*}
and by applying Lemma~\ref{Lemmas_Lem_SwitchDeriv}
\begin{align*}
d T_{\e} \nu_{t}(x)
  &= -\dx \nu_{t}(\mu_{t}  G_{\e}(x, \cdot )) dt
  + \frac{1}{2} \dxx \nu_{t}(\sigma^{2}_{t}  G_{\e}(x, \cdot )) dt\\
  &\qquad- \rho_{t} \dx \nu_{t}(\sigma_{t}   G_{\e}(x, \cdot )) dW_{t}    
    - 2\dx \nu_{t}(\mu_{t}  p_{\e}(x + \cdot )) dt\\
  &\qquad- 2\rho_{t}\dx \nu_{t}(\sigma_{t}  p_{\e}(x + \cdot )) dW_{t}.
\end{align*} 
To introduce the anti-derivative we integrate the above equation over $x > 0$ and 
apply Lemma~\ref{Lemmas_Lem_StochasticFubini} to switch the time and space integrals. (Note: 
Lemma~\ref{Lemmas_Lem_StochasticFubini}
is stated for $\bar{\nu}^{*}$, however the proof only relies on the tail bound from Assumption~\ref{Not_Def_RegCond} 
condition~(\ref{Not_Def_RegCond3}), which is satisfied by $\nu$ and $\tilde{\nu}$.) We arrive at 
\begin{align*}
d \anti T_{\e} \nu_{t}(x)
  &= - \nu_{t}(\mu_{t}  G_{\e}(x, \cdot )) dt
  + \frac{1}{2} \dx \nu_{t}(\sigma^{2}_{t}  G_{\e}(x, \cdot )) dt \\
  & \qquad -  \rho_{t} \nu_{t}(\sigma_{t}   G_{\e}(x, \cdot )) dW_{t}    
   - 2 \nu_{t}(\mu_{t}  p_{\e}(x + \cdot )) dt \\
  & \qquad - 2 \rho_{t} \nu_{t}(\sigma_{t}  p_{\e}(x + \cdot )) dW_{t},
\end{align*}
which, after rewriting using the notation from Lemma~\ref{Lemmas_Lem_Switch1stOrder}, becomes
\begin{align}
\label{eq:Unique_Anti1}
d \anti T_{\e} \nu_{t}
  &= - (\mu_{t}T_{\e} \nu_{t} + \mathcal{E}^{\mu}_{t, \e}) dt
  + \frac{1}{2} \dx ( \sigma^{2}_{t}T_{\e}\nu_{t} + \mathcal{E}^{\sigma^{2}}_{t, \e}) dt \\
  & \qquad -  \rho_{t} (\sigma_{t} T_{\e} \nu_{t} + \mathcal{E}^{\sigma}_{t, \e}) dW_{t}    
 - 2 \nu_{t}(\mu_{t}  p_{\e}(x + \cdot )) dt \nonumber \\
  & \qquad - 2 \rho_{t} \nu_{t}(\sigma_{t}  p_{\e}(x + \cdot )) dW_{t}. \nonumber
\end{align}

We will now introduce the simplifying notation $o_{\mathrm{sq}}(1)$ to denote any family of $L^{2}(0,\infty)$-valued processes, 
$\{(f_{t, \e})_{t \in [0,T]}\}_{\e > 0}$, satisfying 
\[
\E \int^{T}_{0} \Vert f_{t, \e} \Vert_{L^2(0,\infty)}^{2} dt \to 0,
	\qquad \textrm{as } \e \to 0. 
\]
Thus a formal linear combination of $o_{\mathrm{sq}}(1)$ terms is of order $o_{\mathrm{sq}}(1)$. 
Therefore (\ref{eq:Unique_Anti1}) can be written (using  Lemma~\ref{Lemmas_Lem_Switch1stOrder}) as
\begin{align}
\label{eq:Unique_Anti2}
d \anti T_{\e} \nu_{t}
  &= - \mu_{t}T_{\e} \nu_{t} dt
  + \frac{1}{2} \dx ( \sigma^{2}_{t}T_{\e}\nu_{t} + \mathcal{E}^{\sigma^{2}}_{t, \e}) dt
  -  \sigma_{t}\rho_{t} T_{\e} \nu_{t} dW_{t} \\
  &\qquad +  o_{\mathrm{sq}}(1) dt + o_{\mathrm{sq}}(1) dW_{t}  \nonumber   \\
  & \qquad - 2 \nu_{t}(\mu_{t}  p_{\e}(x + \cdot )) dt
  - 2 \rho_{t} \nu_{t}(\sigma_{t}  p_{\e}(x + \cdot )) dW_{t}, \nonumber
\end{align}
and we claim that the integrands in the final two terms are also of order $o_{\mathrm{sq}}(1)$. This claim is in fact 
the critical boundary result from \cite{bush11}, but here we only need first moment estimates:

\begin{lem}[Boundary estimate]
\label{Unique_Lem_Boundary}
We have 
\[
\E \int^{T}_{0} \int^{\infty}_{0} \Big( \int^{\infty}_{0} p_{\e} (x + y) \nu_{t}(dy) \Big)^{2} dx dt	\to 0,
	\qquad \textrm{as } \e \to 0,
\]
hence $\nu_{t}(\mu_{t}  p_{\e}(x + \cdot )) = o_{\mathrm{sq}}(1)$ and $\nu_{t}(\sigma_{t}  p_{\e}(x + \cdot )) 
= o_{\mathrm{sq}}(1)$. 
\end{lem}

\begin{proof}
Begin by noting that 
\begin{align*}
|\nu_{t}(p_{\e}(x + \cdot))|
  &\leq e^{-x^{2} / \e} \int_{0}^{\infty} p_{\e}(y)\nu_{t}(dy) \\
  &\leq c_{1} e^{-x^{2} / \e} \e^{-1/2} [\, \nu_{t}(0,\e^{\eta}) + \exp\{-\e^{2\eta -1} / 2\} \,],
\end{align*}
for $\eta \in (0, \frac{1}{2})$ a free parameter and $c_{1} > 0$ a universal constant. Squaring and integrating over $x > 0$ gives
\[
\int_{0}^{\infty} |\nu_{t}( p_{\e}(x + \cdot))|^{2} dx
  \leq c_{2} \e^{-1/2} [\, \nu_{t}(0,\e^{\eta})^{2} + \exp\{-\e^{2\eta -1}\} \,],
\]
with $c_{2} > 0$ another numerical constant. Condition~(\ref{Not_Def_RegCond4}) of Assumption~\ref{Not_Def_RegCond} 
and the fact that $\nu_{t}(S)^{2} \leq \nu_{t}(S)$, since $\nu_{t}$ is a sub-probability measure, allows us to write
\[
\E \int^{T}_{0} \int_{0}^{\infty} |\nu_{t}( p_{\e}(x + \cdot))|^{2} dx
  = O(\e^{\eta(1 + \beta)-1/2})
  + O(\e^{-1/2}\exp\{-\e^{2\eta -1}\}),
\]
which vanishes if we choose $\eta$ to satisfy
\[
\frac{1}{2(1 + \beta)} < \eta < \frac{1}{2}, 
\]
and this completes the proof.
\end{proof}

With Lemma~\ref{Unique_Lem_Boundary}, we can now reduce (\ref{eq:Unique_Anti2}) to 
\begin{align}
\label{eq:Unique_Anti3}
d \anti T_{\e} \nu_{t}
  &= - \mu_{t}T_{\e} \nu_{t} dt
  + \frac{1}{2} \dx ( \sigma^{2}_{t}T_{\e}\nu_{t} + \mathcal{E}^{\sigma^{2}}_{t, \e}) dt
  -  \sigma_{t}\rho_{t} T_{\e} \nu_{t} dW_{t} \\
  & \qquad + o_{\mathrm{sq}}(1) dt + o_{\mathrm{sq}}(1) dW_{t} \nonumber,
\end{align}
and this equation is also satisfied by $\tilde{\nu}$, as so far all we have used is Assumption~\ref{Not_Def_RegCond}. 
Writing $\Delta := \nu - \tilde{\nu}$ and $\delta^{g}_{t}(x) := g(t,x,L_{t}) - g(t, x, \tilde{L}_{t})$, taking the difference 
of (\ref{eq:Unique_Anti3}) for $\nu$ and $\tilde{\nu}$ yields
\begin{align*}
d \anti T_{\e} \Delta_{t}
  &= - (\tilde{\mu}_{t}T_{\e} \Delta_{t} + \delta^{\mu}_{t} T_{\e} \nu_{t}) dt
  + \frac{1}{2} \dx ( \sigma^{2}_{t}T_{\e}\Delta_{t} + \mathcal{E}^{\sigma^{2}}_{t, \e} -
   \tilde{\mathcal{E}}^{\sigma^{2}}_{t, \e}) dt  \\
  &\qquad -\sigma_{t}(\tilde{\rho}_{t} T_{\e} \Delta_{t} + \delta^{\rho}_{t}T_{\e}\nu_{t}) dW_{t}
  + o_{\mathrm{sq}}(1) dt + o_{\mathrm{sq}}(1) dW_{t},
\end{align*}
where $\tilde{\mathcal{E}}^{\sigma^{2}}_{t, \e}$ is as in Lemma~\ref{Lemmas_Lem_Switch1stOrder}, but with $\nu$ 
replaced by $\tilde{\nu}$. Applying It\^o's formula to the square $(\anti T_{\e} \Delta_{t})^{2}$ gives
\begin{align}
\label{eq:Unique_Anti4}
d (\anti T_{\e} \Delta_{t})^{2}
  &= - 2\anti T_{\e} \Delta_{t}(\tilde{\mu}_{t}T_{\e} \Delta_{t} + \delta^{\mu}_{t} T_{\e} \nu_{t}) dt \\
  & \qquad + \anti T_{\e} \Delta_{t} \dx ( \sigma^{2}_{t}T_{\e}\Delta_{t} + \mathcal{E}^{\sigma^{2}}_{t, \e} -
   \tilde{\mathcal{E}}^{\sigma^{2}}_{t, \e}) dt  \nonumber \\
  &\qquad -2\anti T_{\e} \Delta_{t}\sigma_{t}(\tilde{\rho}_{t} T_{\e} \Delta_{t} + \delta^{\rho}_{t}T_{\e}\nu_{t}) dW_{t} \nonumber \\
  &\qquad + (\tilde{\rho}_{t} T_{\e} \Delta_{t} + \delta^{\rho}_{t}T_{\e}\nu_{t})^{2} dt \nonumber \\
  &\qquad+ \anti T_{\e} \Delta_{t} \cdot o_{\mathrm{sq}}(1) dt 
  + \anti T_{\e} \Delta_{t} \cdot o_{\mathrm{sq}}(1) dW_{t} 
  + o_{\mathrm{sq}}(1)^{2} dt. \nonumber
\end{align}
Note that the initial condition for this equation is zero because $\nu$ and $\tilde{\nu}$ have the same initial condition. 

Since the work in establishing the bounds in Lemma~\ref{Lemmas_Lem_StochasticFubini} only uses the tail estimate
(\ref{Not_Def_RegCond3}) of Assumption~\ref{Not_Def_RegCond}, they remain valid and so, together with 
Lemma~\ref{Lemmas_Lem_AntiTail}, the stochastic integrals in (\ref{eq:Unique_Anti4}) are martingales for fixed $x$ and $\e$. 
Therefore first taking an expectation and then integrating over $x > 0$ and using Young's inequality with free parameter 
$\eta > 0$ produces a constant $c_{\eta} > 0$ such that 
\begin{align}
\label{eq:Unique_Anti5}
\E \left\Vert \anti T_{\e} \Delta_{t} \right\Vert_{2}^{2}
   &\leq c_{\eta} \E \int^{t}_{0} \left\Vert \anti T_{\e} \Delta_{s} \right\Vert_{2}^{2} ds  
  + c_{\eta} \E \int_{0}^{t} \big\Vert (|\delta^{\mu}_{s}| + |\delta^{\rho}_{s}|) |T_{\e} \tilde{\nu}_{s}| \big\Vert_{2}^{2} ds    \\
  &\quad - \E \int^{t}_{0} \int^{\infty}_{0} [\, \sigma^{2}_{s} ( 1 - (1+\eta)\tilde{\rho}^{2}_{s}) - \eta - 
  \eta \tilde{\mu}_{s}^{2}\,]|T_{\e}\Delta_{s}|^{2}dx ds \nonumber \\
  &\quad + o(1),\nonumber
\end{align}
where the terms involving $o_{\mathrm{sq}}(1)$ have collapsed to order $o(1)$. Also notice that (\ref{eq:Unique_Anti5}) 
remains valid if $t$ is a stopping time. 

If it was the case that $\E \int^{t}_{0} \Vert T_{\e}\Delta_{s}\Vert_{2}^{2}ds = 0$, then by 
Proposition~\ref{KSM_Prop_liminfL2} we would have $\Delta = 0$ on $[0,t]$, and so would have completed the proof for this 
value of $t$. It is therefore no loss of generality to assume that this value is bounded away from zero for all $\e > 0$ sufficiently 
small. Then by taking $\eta > 0$ we can find a positive value $c_{0} >0$ such that
\begin{align}
\label{eq:Unique_Anti6}
\E \left\Vert \anti T_{\e} \Delta_{t} \right\Vert_{2}^{2}
  &\leq c \E \int^{t}_{0} \left\Vert \anti T_{\e} \Delta_{s} \right\Vert_{2}^{2} ds \\
  & \qquad + c \E \int_{0}^{t} \big\Vert (|\delta^{\mu}_{s}| + |\delta^{\rho}_{s}|) |T_{\e} \tilde{\nu}_{s}| \big\Vert_{2}^{2} ds
  - c_{0}    
  + o(1),\nonumber 
\end{align}
for $c > 0$ constant. We now want to introduce a comparison between solutions in the $\delta$ terms, and to do so we consider 
two cases.

\subsection*{\emph{Case 1}: Globally Lipschitz coefficients} 
First consider the simpler case where $\mu$ and $\rho$ are Lipschitz in the loss variable, rather than piecewise Lipschitz. 
Therefore we have $|\delta^{g}_{t}| \leq C|L_{t} - \tilde{L}_{t}|$, so the inequality in (\ref{eq:Unique_Anti6}) becomes
\begin{align*}
\E \left\Vert \anti T_{\e} \Delta_{t} \right\Vert_{2}^{2}
  &\leq c_{1} \E \int^{t}_{0} \left\Vert \anti T_{\e} \Delta_{s} \right\Vert_{2}^{2} ds  \\ 
  &\qquad + c_{1} \E \int_{0}^{t} |L_{s} - \tilde{L}_{s}|^2 \big\Vert  T_{\e} \tilde{\nu}_{s} \big\Vert_{2}^{2} ds    
  - c_{0} + o(1),
\end{align*}
with $c_{1} > 0$ constant.  

To bound the second term above, we introduce the stopping times 
\[
t_{n} := \inf\{t > 0:  \sup_{s \in [0,T]} \sup_{\e > 0} \Vert T_{\e} \tilde{\nu}_{s} \Vert_{2}^{2} > n \} \wedge T.
\]
From Proposition~\ref{Unique_Prop_L2Reg} we know that $t_{n} \to T$ as $n \to \infty$, with probability 1. 
Since (\ref{eq:Unique_Anti5}) is valid for stopping times we have 
\begin{align*}
\E \left\Vert \anti T_{\e} \Delta_{t \wedge t_n} \right\Vert_{2}^{2}
  &\leq c_{1} \E \int^{t \wedge t_n}_{0} \left\Vert \anti T_{\e} \Delta_{s} \right\Vert_{2}^{2} ds  \\
  & \qquad + c_{1}n \E \int_{0}^{t \wedge t_n} |L_{s} - \tilde{L}_{s}|^2ds    
  - c_{0}  + o(1) \\
  &\leq c_{1} \E \int^{t}_{0} \left\Vert \anti T_{\e} \Delta_{s\wedge t_n} \right\Vert_{2}^{2} ds   \\
 & \qquad + c_{1}n \E \int_{0}^{t} |L_{s\wedge t_n} - \tilde{L}_{s\wedge t_n}|^2 ds    
  - c_{0}  + o(1).
\end{align*}
By using the integrating factor $e^{-c_{1}t}$ we obtain
\[
\E \left\Vert \anti T_{\e} \Delta_{t \wedge t_n} \right\Vert_{2}^{2}
	\leq c_{1} n e^{c_{1} T} \E \int^{t}_{0} |L_{s\wedge t_n} - \tilde{L}_{s\wedge t_n}|^2 ds - c_{0}' ,
\]
and applying Fatou's lemma and Propositions~\ref{KSM_Prop_SignedMeasure} and~\ref{KSM_Prop_liminfH-1} gives
\[
\E \left\Vert  \Delta_{t \wedge t_n} \right\Vert_{-1}^{2}
	\leq c_{1} n e^{c_{1} T} \E \int^{t}_{0} |L_{s\wedge t_n} - \tilde{L}_{s\wedge t_n}|^2 ds - c_{0}',
\]
where $c_{0}' = c_{0} e^{-c_{1}T} > 0$. 
 
Finally we apply Lemma~\ref{Lemmas_Lem_LossInH1} to the above inequality to reintroduce $\Delta$ to the right-hand side. 
With fixed $\alpha > 0$ we have
\[
\E \left\Vert  \Delta_{t \wedge t_n} \right\Vert_{-1}^{2}
	\leq c_{2}(\delta^{-1} + \lambda) \E \int^{t}_{0} \left\Vert  \Delta_{s \wedge t_n} \right\Vert_{-1}^{2} ds
	+ c_{2} \delta + c_{\alpha}e^{-\alpha \lambda}
	-c_{0}',
\]
where $c_{2} > 0$ does not depend on $\alpha$ (but does depend on $n$). Now fix $\delta = c_{0}' / c_{2}$ so that we have 
\[
\E \left\Vert  \Delta_{t \wedge t_n} \right\Vert_{-1}^{2}
	\leq c_{3}(1 + \lambda) \E \int^{t}_{0} \left\Vert  \Delta_{s \wedge t_n} \right\Vert_{-1}^{2} ds
	+ c_{\alpha}e^{-\alpha \lambda}
\]
with $c_{3} > 0$ independent of $\alpha$. By using the integrating factor $e^{-c_{3}(1 + \lambda)t}$ we deduce
\[
\E \left\Vert  \Delta_{t \wedge t_n} \right\Vert_{-1}^{2}
	\leq c_{\alpha}e^{c_{3}(1 +\lambda) t - \alpha \lambda},
\]
so setting $\alpha = 2c_{3}t$ and sending $\lambda \to \infty$ gives $\E \left\Vert  \Delta_{t \wedge t_n} \right\Vert_{-1}^{2} 
= 0$. Therefore $\nu = \tilde{\nu}$ on $[0,t_{n}]$, and since $t_n \to T$ we have Theorem~\ref{Intro_Thm_Unique} in Case 1. 

\subsection*{\emph{Case 2}: Piecewise Lipschitz coefficients}

To extend the argument to the general case, we use a stopping argument and consider the system only on time 
intervals where the loss processes are in the same interval $[\theta_{i}, \theta_{i + 1})$ --- recall 
Assumption~\ref{Not_Def_Coefficients}. 

Define the stopping times
\[
T_{0} := \inf\{ t > 0 : L_{t} \geq \theta_{1} \} \wedge T
\qquad
\tilde{T}_{0} := \inf\{ t > 0 : \tilde{L}_{t} \geq \theta_{1} \} \wedge T
\]
and $S_{0} = T_{0} \wedge \tilde{T}_{0}$. For the reason immediately proceeding (\ref{eq:Unique_Anti4}), the 
argument in Case 1 can be replicated on $[0,S_{0})$ by replacing $t$ by $t \wedge S_{0}$, since before $S_{0}$, 
the coefficients can be compared using the Lipschitz property on $[\theta_{0}, \theta_{1})$. Therefore we conclude 
$\nu_{t} = \tilde{\nu}_{t}$ for $t \leq S_{0}$, which forces $L_{t} = \tilde{L}_{t}$ for $t \leq S_{0}$ and thus 
$T_{0} = S_{0} = \tilde{T}_{0}$.

We can then repeat the argument for the interval $[S_{0}, S_{1})$, since $\Delta_{S_{0}} = 0$ (by continuity of 
$\nu$ and $\tilde{\nu}$), where
\[
T_{1} := \inf\{ t > S_{0} : L_{t} \geq \theta_{2} \} \wedge T
\qquad
\tilde{T}_{1} := \inf\{ t > S_{0} : \tilde{L}_{t} \geq \theta_{2} \} \wedge T
\] 
and $S_{1} = T_{1} \wedge \bar{T}_{1}$. Continuing upto $S_{k}$ covers all the $[\theta_{i}, \theta_{i + 1})$ 
intervals, and this completes the proof, since $L$ and $\tilde{L}$ are increasing (Assumption~\ref{Not_Def_RegCond}, 
condition~(\ref{Not_Def_RegCond1})) so $[0,T] \subseteq \cup_{i = 0}^{k-1} [S_{i}, S_{i+1})$. \qed

\section{Technical lemmas}
\label{Sect_Lemmas}

This section collects all the technical lemmas that were used in Section~\ref{Sect_Uniqueness}, and should be read only
 as a reference.

\begin{lem}
\label{Lemmas_Lem_Switch1stOrder}
Let $g_{s}(x) = g(s, x, L_{s})$ where $g$ is one of $\mu$, $\sigma$ or $\sigma^{2}$ and $L_{s} = 1- \nu_{s}(0,\infty)$. 
Define the error term
\[
\mathcal{E}^{g}_{t, \e}(x) 
	:= \nu_{t}(g_{t}(\cdot) G_{\e}(x , \cdot)) 
	- g_{t}(x)  T_{\e}\nu_{t}(x).
\]
Then 
\[
\E \int^{T}_{0} \Vert \mathcal{E}^{g}_{t, \e} \Vert_{L^{2}(0,\infty)}^{2} dt \to 0,
	\qquad \textrm{as } \e \to 0.
\]
\end{lem}

\begin{proof}
Let $\lambda = \lambda(\e) \rightarrow \infty$, as $\e \rightarrow 0$, be a function that we will specify later. For any $x >0$
\begin{align*}
|\mathcal{E}^{g}_{t, \e}(x)|
  &\leq \Vert \dx g \Vert_{\infty} \int_{0}^{\infty} |x-y|p_{\e}(x-y) \nu_{t}(dy)\\
  &\leq c_{1} \e^{\eta - \frac{1}{2}} \nu_{t}(x - \e^{\eta}, x + \e^{\eta}) 
  + c_{1} \e^{-1/2} \exp\{-\e^{2 \eta - 1} / 2\},
\end{align*}
with $c_{1} > 0$ a universal constant, and where the second line follows by splitting the integral on $|y - x| < \e^{\eta}$ 
and its complement. By considering the range $x < \lambda$ and using condition~(\ref{Not_Def_RegCond5}) of 
Assumption~\ref{Not_Def_RegCond}
\begin{equation}
\label{eq:Kerenl_Switch1}
\E \int^{T}_{0} \left\Vert \mathcal{E}^{g}_{t, \e} \right\Vert_{L^{2}(-\lambda, \lambda)}^{2} dt
  = \lambda(\e) O(\e^{(2 + \delta) \eta - 1} + \e^{-1} \exp\{-\e^{2 \eta - 1}\} )
  = \lambda(\e) O(\e^{\gamma}),
\end{equation}
for some $\delta, \gamma > 0$, by fixing $\eta$ in the range 
\[
\frac{1}{2 + \delta} < \eta < \frac{1}{2}.
\] 

Now consider the range $x \geq \lambda$. Decomposing the $y$-integral on the range $y < x / 2$ and its complement gives 
\[
|\mathcal{E}^{g}_{t, \e}(x)| 
  \leq 2 \left\Vert g \right\Vert_{\infty} \int_{0}^{\infty} p_{\e}(x - y) \nu_{t}(dy)
  \leq c_{2} p_{\e}(x/2)
  + c_{2} \e^{-1/2} \nu_{t} (|x|/2, +\infty),
\]
with $c_{2} > 0$ another universal constant. Therefore 
\begin{align}
\label{eq:Kerenl_Switch2}
\E \int^{T}_{0} & \left\Vert \mathcal{E}^{g}_{t, \e}  \right\Vert_{L^{2}((-\lambda, \lambda)^{c})}^{2} dt \\
  &= O\Big(\e^{-1/2}e^{-\lambda(\e)^{2}/8\e} \int^{\infty}_{-\infty} p_{\e}(x/2)dx  + \e^{-1}\int_{\lambda(\e)}^{\infty} e^{-x} dx \Big) \nonumber \\
  &= O(\e^{-1} e^{-\lambda(\e)}). \nonumber
\end{align}
Summing (\ref{eq:Kerenl_Switch1}) and (\ref{eq:Kerenl_Switch2}) and fixing $\lambda(\e) = \log(\e^{-2})$ completes the proof.
\end{proof}

\begin{lem}
\label{Lemmas_Lem_Switch2ndOrder}
Let $g_{s}(x) = g(s, x, L^{*}_{s})$ where $g$ is one of $\mu$, $\sigma$ or $\sigma^{2}$ and $L^{*}_{s} = 
1- \bar{\nu}^{*}_{s}(0,\infty)$. Define the error term
\begin{gather*}
\bar{\mathcal{E}}^{g}_{t, \e}(x) 
	:= \dx\bar{\nu}^{*}_{t}(g_{t} p_{\e}(x - \cdot)) 
	- g_{t}(x) \dx \bar{T}_{\e}\bar{\nu}^{*}_{t}(x)
	+ \dx g_{t}(x) \bar{\mathcal{H}}^{g}_{t, \e}(x)
	\\ \textrm{where} \qquad \bar{\mathcal{H}}^{g}_{t, \e}(x) :=  \bar{\nu}^{*}_{t}((x - y)\dx p_{\e}(x - \cdot)).
\end{gather*}
Then 
\[
\E \int^{T}_{0} \Vert \bar{\mathcal{E}}^{g}_{t, \e} \Vert_{L^{2}(\R)}^{2} dt \to 0,
	\qquad \textrm{as } \e \to 0
\]
and there exists a numerical constant $c > 0$ such that 
\[
|\bar{\mathcal{H}}^{g}_{t, \e}(x)| \leq c \bar{T}_{2\e}\bar{\nu}^{*}_{t}(x),
	\qquad \textrm{for all } t \in [0,T], x \in \R \textrm{ and } \e > 0.
\] 
\end{lem}

\begin{proof}
Interchanging differentiation and integration with respect to $\bar{\nu
}^{*}_{t}$ gives
\[
\bar{\mathcal{E}}^{g}_{t, \e}(x) 
	= \int_{\R}[g_{t}(y) - g_{t}(x) + (y - x)\dx g_{t}(x)]\dx p_{\e}(x - y)\bar{\nu}^{*}_{t}(dy).
\]
By bounding with the second-order derivative and using $\dx p_{\e}(x -y) = -2\e^{-1}(x-y)p_{\e}(x - y)$ gives
\[
|\bar{\mathcal{E}}^{g}_{t, \e}(x)|
	\leq \frac{1}{2} \int_{\R} |\dxx g_{t}(x)| |x - y|^{3} \e^{-1} p_{\e}(x-y) \bar{\nu}^{*}_{t} (dy).
\]
We therefore have the same order of \e\ as in Lemma~\ref{Lemmas_Lem_Switch1stOrder}, so the first result follows 
by the same work. For the second result, notice that
\[
|z \dx p_{\e}(z)| 
	= \frac{1}{\sqrt{2\pi \e}} \e^{-1} z^{2} e^{-z^{2} / 2\e}
	= \sqrt{2}\e^{-1} z^{2} e^{-z^{2} / 4\e} p_{2\e}(z),
\]
and $\sup_{z \in \R} z^{2} e^{-z^{2} / 4\e} = \e$. 
\end{proof}

\begin{lem}[Stochastic Fubini]
\label{Lemmas_Lem_StochasticFubini}
For all $n, m \geq 0$, $\e > 0$ and $t \in \timeInt$
\[
\int_{\R} \Bigl(  \int^{t}_{0} \E [ |\dx^{n} \bar{T}_{\e} \nub^{*}_{s}(x) \cdot \dx^{m} \bar{T}_{\e} \nub^{*}_{s}(x)|^{2} ] 
ds \Bigr)^{1/2} dx < \infty,
\]
hence the stochastic Fubini theorem \cite[1.4]{veraar2012} gives
\begin{multline*}
\int_{\R} \int_{0}^{t}  g_{t}(x) \cdot \dx^{n} \bar{T}_{\e} \nub^{*}_{s}(x) \cdot \dx^{m} \bar{T}_{\e} \nub^{*}_{s}(x) 
dW_{s} dx \\
  = \int_{0}^{t} \int_{\R} g_{t}(x) \cdot \dx^{n}\bar{T}_{\e} \nub^{*}_{s}(x) \cdot \dx^{m} \bar{T}_{\e} \nub^{*}_{s}(x) 
  dx dW_{s}
\end{multline*}
whenever $\sup_{t \in  [0,T], x \in \R}|g_{t}(x)| < \infty$.
\end{lem}

\begin{proof}
By applying Young's inequality and concavity of $z \mapsto \sqrt{z}$, it suffices to show that 
\[
\int_{\R} \Bigl( \int^{t}_{0} \E [ |\dx^{n} \bar{T}_{\e} \nub^{*}_{s}(x)|^{4} ] ds \Bigr)^{1/2} dx < \infty.
\]
First notice that 
\[
\dx^{n} \bar{T}_{\e} \nub^{*}_{s}(x)
  = \nub^{*}_{s}(\dx^{n} p_{\e}(x - \cdot))
  =  \nub^{*}_{s}(P_{n}(\e^{-1}(x - \cdot))p_{\e}(x - \cdot) ),
\]
where $P_{n}$ is a polynomial of degree $n$. Since $\nub^{*}_{s}$ is a probability measure, H\"older's inequality gives
\begin{equation}
\label{eq:Lemmas_InterchangeFinite1}
\E [|\dx^{n} \bar{T}_{\e} \nub^{*}_{s}(x)|^{4}]
  \leq \E \int_{\R} |P_{n}(\e^{-1}(x - y))|^{4}p_{\e}(x - y)^{4} \nub^{*}_{s}(dy).
\end{equation}
For any value of $x$, the integrand above is bounded (recall that \e\ is fixed). Hence it suffices to bound the right-hand side 
of (\ref{eq:Lemmas_InterchangeFinite1}) in terms of $x$ only for large values of $|x|$. Splitting the $y$-integral on the 
region $|y| < x/2$ and its complement gives the bound
\begin{multline*}
\E [|\dx^{n} \bar{T}_{\e} \nub^{*}_{s}(x)|^{4}] \\
  \leq c_{\e} \E \nub^{*}_{s}((x/2, +\infty) \cup (-\infty, -x/2))
  + c_{\e} \exp\{ - x^{2}/2\e \}
  = O(e^{-x}),
\end{multline*}
where $c_{\e}$ and the $O$ depend only on \e\ and where we have used the tail estimate from 
Lemma~\ref{Unique_Lem_WholeSpace}. This suffices to complete the proof. 
\end{proof}

\begin{lem}[An integration-by-parts calculation]
\label{Lemmas_Lem_EasyIBP}
Let $f, g \in C^{1}(\R)$ be bounded with bounded first derivatives. Assume also that these functions and their 
first derivatives vanish at $\pm \infty$. Then 
\[
\int_{\R} g(x) f(x) \dx f(x) dx 
  = - \frac{1}{2} \int_{\R} \dx g(x) f(x)^{2} dx.
\]
\end{lem}

\begin{proof}
Integration by parts. 
\end{proof}

\begin{lem}
\label{Lemmas_Lem_B--D--G}
There exists a constant $c > 0$ such that 
\begin{multline*}
\E \sup_{u \in [0,t]} \Big| 2 \int_{0}^{u} \int_{\R} \rho_{s}\bar{T}_{\e} \nub_{s}^{*} (\sigma_{s} \dx \bar{T}_{\e} 
\nub_{s}^{*} + \dx \sigma_{s} \bar{\mathcal{H}}_{s, \e} + \bar{\mathcal{E}}_{s, \e}^{\sigma} ) dx dW_{s} \Big|    \\
  \leq \frac{1}{2} \E \sup_{s \in [0,t]} \big\Vert \bar{T}_{\e} \nub_{s}^{*} \big\Vert_{2}^{2}
  + c \E \int^{t}_{0} \big\Vert \bar{T}_{\e} \nub_{s}^{*} \big\Vert_{2}^{2} ds
  + c \E \int^{t}_{0} \big\Vert \bar{\mathcal{E}}_{s, \e}^{\sigma} \big\Vert_{2}^{2} ds
\end{multline*}
for all $t \in [0,T]$.
\end{lem}

\begin{proof}
By a similar analysis to (\ref{eq:Lemmas_InterchangeFinite1}) we know that, for every fixed \e, the integrand above is a 
rapidly decaying function of $x$, hence the stochastic integral is a martingale, so the Burkholder--Davis--Gundy
inequality~\cite[Thm.~IV.42.1]{rogerswilliams2} gives a universal constant, $c_{1}>0$, for which the left-hand side above is 
bounded by
\[
2c_{1} \E \Big[ \Big( \int^{t}_{0} \Big( \int_{\R} \rho_{s} \bar{T}_{\e} \nub_{s}^{*} (\sigma_{s} \dx \bar{T}_{\e} \nub_{s}^{*} 
+ \dx \sigma_{s} \bar{\mathcal{H}}_{s, \e} + \bar{\mathcal{E}}^{\sigma}_{s, \e}  ) dx \Big)^{2} ds \Big)^{1/2} \Big].
\]
By Lemma~\ref{Lemmas_Lem_EasyIBP}, this is equal to a constant multiple of 
\[
\E \Big[ \Big( \int^{t}_{0} \Big( \int_{\R} \rho_{s} \bar{T}_{\e} \nub_{s}^{*} (- \dx \sigma_{s} \bar{T}_{\e} \nub_{s}^{*} 
+ \dx \sigma_{s} \bar{\mathcal{H}}_{s,\e} + \bar{\mathcal{E}}^{\sigma}_{s, \e}) dx \Big)^{2} ds \Big)^{1/2} \Big],
\]
which, by H\"older's inequality, is bounded by a constant multiple of 
\begin{multline*}
\E \Big[ \Big( \int^{t}_{0}  \big\Vert \bar{T}_{\e} \nub_{s}^{*} \big\Vert_{2}^{2} \big\Vert - \dx \sigma_{s}  
\bar{T}_{\e} \nub_{s}^{*} + \dx \sigma_{s} \bar{\mathcal{H}}^{\sigma}_{s, \e} + \bar{\mathcal{E}}^{\sigma}_{s, \e}
 \big\Vert_{2}^{2} ds \Big)^{1/2} \Big]  \\
  \leq \E \Big[ \sup_{s \in [0,t]} \big\Vert \bar{T}_{\e} \nub_{s}^{*} \big\Vert_{2} \Big( \int^{t}_{0} \big\Vert - \dx \sigma_{s}  
  \bar{T}_{\e} \nub_{s}^{*} + \dx \sigma_{s} \bar{\mathcal{H}}^{\sigma}_{s, \e} + \bar{\mathcal{E}}^{\sigma}_{s, \e} 
  \big\Vert_{2}^{2} ds \Big)^{1/2} \Big].
\end{multline*}
The result then follows by applying Young's inequality with parameter $1/2$ and using the boundedness of the coefficients.
\end{proof}

\begin{lem}[Switching derivatives]
\label{Lemmas_Lem_SwitchDeriv}
For all $x,y \in \R$ and $\e > 0$ we have
\begin{enumerate}[(i)]
\item $\partial_{y} G_{\e}(x, y) = - \dx G_{\e}(x, y) - 2 \dx p_{\e}(x + y)$,
\item $\partial_{yy} G_{\e}(x, y) = \dxx G_{\e}(x, y)$.
\end{enumerate}
\end{lem}

\begin{proof}
An easy calculation. 
\end{proof}

\begin{lem}
\label{Lemmas_Lem_AntiTail}
For all $x > 0$, $t \in [0,T]$ and $\e > 0$
\[
|\anti T_{\e} \Delta_{t}(x)|
	\leq \nu_{t}(x/2, +\infty) + \tilde{\nu}_{t}(x/2, +\infty)
	+ e^{-x^{2} / 8\e}.
\]
\end{lem}

\begin{proof}
Split the integral 
\[
\anti T_{\e} \nu_{t}(x)
	 = -\int^{\infty}_{x} \int^{\infty}_{0} G_{\e}(y, z) \nu_{t} (dz) dy
\]
at $z < x/2$ and its complement to obtain
\begin{align*}
|\anti T_{\e} \nu_{t}(x)|
	&\leq \frac{1}{\sqrt{2\pi\e}} \int_{x}^{\infty} e^{-(y - x/2)^{2} / 2\e} dy
	+ \nu_{t}(x/2, +\infty) \\
	&\leq e^{-x^{2} / 8\e} + \nu_{t}(x/2, +\infty).
\end{align*}
The triangle inequality completes the result.
\end{proof}

\begin{lem}
\label{Lemmas_Lem_LossInH1}
Let $\nu$, $\tilde{\nu}$, $L$, $\tilde{L}$ and $\Delta$ be as in Section~\ref{Sect_Uniqueness}. 
For every $\alpha > 0$ there exists a constant $c_{\alpha} > 0$ such that
\[
\E \int^{t}_{0} |L_{s} - \tilde{L}_{s}|^{2} ds
	\leq c(\delta^{-1} + \lambda)\E \int^{t}_{0} \Vert \Delta_{s} \Vert_{-1}^{2} ds 
	+ c \delta 
	+ c_{\alpha} e^{-\alpha \lambda}.
\]
for all $t \in [0,1]$, $0<\delta<1$ and $\lambda \geq 1$, where $c>0$ is a constant that does not depend on $\alpha$. 
\end{lem}

\begin{proof}
For $0 <\delta < 1$ and $\lambda \geq 1$, let $\phi_{\delta, \lambda} \in H^{1}_{0}(0, \infty)$ be any cut-off function satisfying
\[
\phi_{\delta, \lambda}(x)\begin{cases}
=0, & \textrm{if }x=0\\
\in(0,1), & \textrm{if }0<x<\delta\\
=1, & \textrm{if }\delta\leq x\leq\lambda\\
\in(0,1), & \textrm{if }\lambda<x<\lambda + 1\\
=0, & \textrm{if }x\geq  \lambda + 1,
\end{cases}
\]
$\left\Vert \dx \phi_{\delta, \lambda} \right\Vert_{L^{\infty}(0,\delta)} \leq c_{1} \delta^{-1}$ and $\left\Vert \dx 
\phi_{\delta, \lambda} \right\Vert_{L^{\infty}(\lambda,\lambda+1)} \leq c_{1}$, for some constant $c_{1}>0$. Then 
\[
\left\Vert \phi_{\delta, \lambda} \right\Vert_{H^{1}_{0}}^{2}
  \leq \int_{0}^{\lambda + 1} dx
  + \int_{0}^{\delta} c_{1}^{2} \delta^{-2} dx
  + \int_{\lambda}^{\lambda + 1} c_{1}^{2} dx
  = c_{2} (\delta^{-1} + \lambda),
\]
for $c_{2}> 0$ a constant. Therefore
\begin{align*}
|L_{t}  - \tilde{L}_{t}|
	&= |\nu_{t}(0,\infty) - \tilde{\nu}_{t}(0,\infty)|  \\
	&\leq |\nu_{t}(\phi_{\delta, \lambda}) - \tilde{\nu}_{t}(\phi_{\delta, \lambda})|
	+ |\nu_{t}(0, \delta)| + |\tilde{\nu}_{t}(0, \delta)| \\
	& \qquad + |\nu_{t}(\lambda, +\infty)| + |\tilde{\nu}_{t}(\lambda, +\infty)|  \\
	&\leq c_{2}^{1/2}(\delta^{-1} + \lambda)^{1/2}\Vert \nu_{t} - \tilde{\nu}_{t} \Vert_{-1}
	+ |\nu_{t}(0, \delta)| + |\tilde{\nu}_{t}(0, \delta)| \\
	& \qquad + |\nu_{t}(\lambda, +\infty)| + |\tilde{\nu}_{t}(\lambda, +\infty)|
\end{align*}
and so the result follows from conditions~(\ref{Not_Def_RegCond3}) and~(\ref{Not_Def_RegCond4}) of 
Assumption~\ref{Not_Def_RegCond} (and that $|\nu_{t}(S)|^{2} \leq |\nu_{t}(S)|$ for all $S \subseteq \R$). 
\end{proof}

The following result will be used in Section~\ref{Sect_MV}.

\begin{lem}[Interchanging stochastic integration and conditional expectation]
 \label{Lemmas_Lem_Interchange}
Suppose we are working on a probability space with filtration $\{\mathcal{F}_{t}\}$ and $W$ is a standard Brownian motion 
with natural filtration $\{\mathcal{F}_{t}^{W}\}$. Let $H$ be a real-valued $\{\mathcal{F}_{t}\}$-adapted process with 
\[
\mathbf{E}\int_{0}^{T}H_{s}^{2}ds<\infty.
\]
Then, with probability 1, 
\[
\mathbf{E}\left[\left.\int_{0}^{t}H_{s}dW_{s}\right|\mathcal{F}_{t}^{W}\right]=\int_{0}^{t}
\mathbf{E}\left[\left.H_{s}\right|\mathcal{F}_{s}^{W}\right]dW_{s}
\]
and 
\[
\mathbf{E}\left[\left.\int_{0}^{t}H_{s}dW_{s}^{1}\right|\mathcal{F}_{t}^{W}\right]=0
\]
for every $t\in\left[0,T\right].$
\end{lem}

\begin{proof}
As we can multiply $H_{s}$ by $\mathbf{1}_{s<t}$, it suffices to
take $t=T$. First, suppose that $H$ is a basic process, that is
\[
H_{u}=Z\mathbf{1}_{s_{1}<u\leq s_{2}},
\]
where $s_{1}<s_{2}\leq T$ are real numbers and $Z$ is $\mathcal{F}_{s_{1}}$-measurable.
Then 
\begin{align*}
\mathbf{E}\left[\left.\int_{0}^{T}H_{s}dW_{s}\right|\mathcal{F}_{T}^{W}\right] & =\mathbf{E}\left[\left.Z\left(W_{s_{2}}-
W_{s_{1}}\right)\right|\mathcal{F}_{T}^{W}\right]\\
 & =\mathbf{E}\left[\left.Z\right|\mathcal{F}_{s_{1}}^{W}\right]\left(W_{s_{2}}-W_{s_{1}}\right)\\
 & =\int_{0}^{T}\mathbf{E}\left[\left.Z\right|\mathcal{F}_{s}^{W}\right]\mathbf{1}_{s_{1}<s\leq s_{2}}dW_{s}\\
 & =\int_{0}^{T}\mathbf{E}\left[\left.H_{s}\right|\mathcal{F}_{s}^{W}\right]dW_{s}
\end{align*}
and 
\begin{align*}
\mathbf{E}\left[\left.\int_{0}^{T}H_{s}dW_{s}^{1}\right|\mathcal{F}_{t}^{W}\right] & 
=\mathbf{E}\left[\left.Z\left(W_{s_{2}}^{1}-W_{s_{1}}^{1}\right)\right|\mathcal{F}_{t}^{W}\right]\\
 & =\mathbf{E}\left[\left.\mathbf{E}\left[\left.Z\left(W_{s_{2}}^{1}-W_{s_{1}}^{1}\right)\right|\sigma\left(\mathcal{F}_{T}^{W},
 \mathcal{F}_{s_{1}}\right)\right]\right|\mathcal{F}_{t}^{W}\right]\\
 & =\mathbf{E}\left[\left.Z\mathbf{E}\left[\left.\left(W_{s_{2}}^{1}-W_{s_{1}}^{1}\right)\right|\sigma\left(\mathcal{F}_{T}^{W},
 \mathcal{F}_{s_{1}}\right)\right]\right|\mathcal{F}_{t}^{W}\right]\\
 & =\mathbf{E}\left[\left.Z\mathbf{E}\left[W_{s_{2}}^{1}-W_{s_{1}}^{1}\right]\right|\mathcal{F}_{T}^{W}\right]=0,
\end{align*}
where we have used the fact that $W_{s_{2}}^{1}-W_{s_{1}}^{1}$ is
independent of $\sigma\!\left(\mathcal{F}_{T}^{W}\!,\mathcal{F}_{s_{1}}\!\right)$
since $W^{1}$ and $W$ are independent and $W^{1}$ has independent
increments.  So the result holds in this case and immediately extends
to linear combinations of basic processes. The usual density argument
then allows us to extend the result to all required $H$. 
\end{proof}

\section{Stochastic McKean--Vlasov problem; Proof of Theorem \ref{Intro_Thm_MV}}
\label{Sect_MV}

This section presents a short proof of Theorem~\ref{Intro_Thm_MV}. Take a strong solution $(\nu, W)$ to the limit SPDE 
(Remark~\ref{Intro_Rem_StrongSolutions}), an independent Brownian motion $W^{\perp}$ and define $X$ by
\[
\begin{cases}
X_{t} 
	= X_{0}
	+ \int_{0}^{t} \mu(s,X_{s}, L_{s}) ds
	+ \int_{0}^{t} \sigma(s,X_{s})\rho(s,L_{s}) dW_{s} \\
	\qquad \qquad + \int_{0}^{t} \sigma(s,X_{s})(1-\rho(s,L_{s})^{2})^{\frac{1}{2}} dW^{\perp}_{s} \\
\tau=\inf \{ t>0 : X_{t} \leq 0 \}.
\end{cases}
\]
(It is possible to find such an $X$ by standard diffusion theory, since $t \to L_{t} = 1 - \nu_{t}(0,\infty)$ 
is given and fixed.) Let $\tilde{\nu}$ be the conditional law of $X$ given $W$ killed at zero, that is 
\[
\tilde{\nu}_{t}(S) := \P(X_{t} \in S; t < \tau | W).
\]
We will have the existence statement of Theorem~\ref{Intro_Thm_MV} if we can prove $\nu = \tilde{\nu}$. 

Applying It\^o's formula to $\phi(X_{t})$ as in the proof of Proposition~\ref{Finite_Prop_Evolution} gives
\begin{multline*}
\phi(X_{t}) \1_{t < \tau}
  = \phi(X_{0}) 
  + \int_{0}^{t} (\mu_{s} \dx \phi)(X_{s}) \1_{s < \tau} ds
  + \frac{1}{2} \int_{0}^{t} (\sigma_{s}^{2} \dxx \phi) (X_{s}) 1_{s < \tau} ds   \\
   \qquad 
   + \int_{0}^{t} (\sigma_{s} \rho_{s} \dx \phi)(X_{s}) \1_{s < \tau} dW_{s} 
  + \int_{0}^{t} (\sigma_{s} (1-\rho_{s}^{2})^{\frac{1}{2}} \dx \phi) (X_{s}) \1_{s < \tau} dW^{\perp}_{s}.
\end{multline*}
Take a conditional expectation with respect to $W$ by applying Lemma~\ref{Lemmas_Lem_Interchange} (and using 
that $L$ is $\sigma(W)$-measurable) to get
\begin{align*}
\tilde{\nu}_{t}(\phi) 
	&= \nu_{0}(\phi)
	+ \int^{t}_{0} \tilde{\nu}_{s}( \mu(s, \cdot, L_{s}) \dx \phi ) ds
	+ \frac{1}{2} \int^{t}_{0} \tilde{\nu}_{s}( \sigma^{2}(s, \cdot, L_{s}) \dxx \phi ) ds \\  
	&\quad + \int_{0}^{t} \tilde{\nu}_{s}( \sigma(s, \cdot)\rho(s, L_{s}) \dx \phi ) dW_{s},
	\qquad \textrm{with \ \ } L_{t} = 1 - \nu_{t}(0,\infty).
\end{align*}
Now, $\nu$ also satisfies this equation, however in both cases the coefficients depend only on $L$. Therefore we 
can regard $L$ as fixed and $\nu$ and $\tilde{\nu}$ as solving the limit SPDE in the special case when coefficients 
do not depend on the loss-variable. This is a much easier linear problem and Theorem~\ref{Intro_Thm_Unique} is 
certainly sufficient to conclude $\nu = \tilde{\nu}$, as required.

We have also just shown that if $(X, W)$ solves the McKean--Vlasov problem in Theorem~\ref{Intro_Thm_MV}, then its 
conditional law $\nu = \tilde{\nu}$ solves the limit SPDE. By Theorem~\ref{Intro_Thm_Unique}, this fixes the law of $\nu$, 
hence we have the uniqueness statement too. \qed 

\section{Open problems}
\label{Sect_Open}

We end by giving some open problems arising from our model and its related extensions:

\begin{enumerate}[(i)]

\item As indicated at the end of Section~\ref{Sect_Intro}, the most important practical question is how do we numerically
approximate $\nu$ from a given realisation of $W$? This leads to the further questions of how do we combine these 
approximations to get an estimator for $\E \Psi(L)$, where $\Psi : D_{\R} \to \R$ is some pay-off function, and how do we 
calibrate the model to any data on traded prices for options with payoff $\Psi(L)$?

Our proposed algorithm for the first problem is as follows. Here, we discretise the time variable and treat the outputs of 
the following subroutines as functions on $[0,\infty)$ --- in practise we would also need a discretisation scheme for the 
spatial variable too, but we will not consider that problem here. Fix a precision level $\delta > 0$ and assume we are given 
a piecewise constant or piecewise linear approximation to a Brownian trajectory $t \mapsto w_{t}$ to precision at least 
$\delta$ (generating such a path contributes negligible computational cost in this algorithm) and an initial density 
$V^{(0)}$. Set $L^{(0)} = 0$. 
For $1 \leq n \leq T / \delta - 1$, form $V^{(n)}$ recursively by setting $V^{(n)} = u_{\delta}$ where $u$ solves the 
deterministic linear PDE
\begin{align}
\label{eq:Open_Algo1}
du_{t}(x) 
	&= -\mu(t, x, L^{(n-1)}) \partial_{x} u_{t}(x) dt
	+ \frac{1}{2}\sigma(t,x)\rho(t, L^{(n-1)})\partial_{xx} u_{t}(x)dt  \\
	&\qquad- \sigma(t,x)\sqrt{1-\rho(t, L^{(n-1)})^{2}}\partial_{x} u_{t}(x) dw_{t},
		\qquad \textrm{with } u_{t}(0) = 0, \nonumber
\end{align}
for $t \in [0,\delta]$ and $x > 0$. Set $L^{(n)} = 1 - \int^{\infty}_{0} V^{(n)}(x)dx$ (calculated using some quadrature 
routine). Our approximation to the density process, $V$, of $\nu$ and the loss process, $L$, are given by piecewise 
interpolation of $\{V^{(n)}\}_{n}$ and $\{L^{(n)} \}_{n}$:
\begin{align*}
\tilde{V}_{t} 
	&:= (1-\mathrm{frac}\{s\}) V^{([s])} + \mathrm{frac}\{s\} V^{([s] + 1)}
	\\
	\tilde{L}_{t} 
	&:= (1-\mathrm{frac}\{s\}) L^{([s])} + \mathrm{frac}\{s\} L^{([s] + 1)},
\end{align*}
where $s := t / \delta$, $[s]$ is the floor of $s$ and $\mathrm{frac}\{s\} = s - [s]$. 

In the case when $\sigma$ and $\mu$ are constant and $\rho$ depends only on the loss variable and $w$ is 
given as a piecewise constant interpolation of $W$ with precision $\delta$, the solution to (\ref{eq:Open_Algo1}) 
can be written explicitly in terms of the Brownian transition kernel. A numerical solution can then be found by 
quadrature. (This instance of the algorithm was used to produce Figure~\ref{Intro_Fig_LossDepHeat}.) If these 
assumption do not hold, then further approximations may be necessary. In \cite{giles2012} (\ref{eq:Open_Algo1}) 
is solved (for the constant coefficient case) by finite element methods and the scheme is proven to converge when 
the system is considered on the whole space. The authors conjecture and provide numerical evidence for a 
convergence rate for the scheme on the half-line with space-time discretisation. A first open problem is to verify 
that the piecewise-constant time-discretisation, $\tilde{V}$, above converges in law to the solution $\nu$ of limit 
SPDE as $\delta \to 0$. A harder problem is to establish the rate of convergence, in some appropriate norm, 
averaged over realisations of $W$. 

Returning to the task of calculating the pay-off $\E\Psi(L)$, we have the estimator
\[
\mathcal{E}_{m, \delta}
	:= \frac{1}{m} \sum_{i = 1}^{m} \Psi(\tilde{L}_{w^{i}, \delta})
\]
where $\{w^{i}\}_{1 \leq i \leq m}$ are independent standard Brownian motions and $\tilde{L}_{w, \delta}$ 
denotes the approximation to the loss function using the algorithm above with precision $\delta$ and Brownian 
trajectory $w$. As the Monte Carlo routine depends on $\delta$, a natural variance reduction technique is to 
use multi-level Monte Carlo as in \cite{giles2012}. Another potentially useful technique is to alter the drift 
coefficient in (\ref{eq:Intro_Coefficients}) using Girsanov's theorem to produce a reweighted estimator. In the 
case when the pay-off function, $\Psi$, is supported on large losses, and hence is sensitive only to rare events, 
changing the measure to one under which the particles have a large negative drift and multiplying by the appropriate 
Radon--Nikodym derivative is a form of importance sampling. A simpler observation in this scenario is that if the 
systemic Brownian motion has a realisation that has followed a largely increasing path on $[0,T]$, then although 
that realisation is likely to contribute little to $\mathcal{E}_{m, \delta}$, the negative of this realisation is likely to 
give a heavy contribution. Hence the simple antithetic sampling routine in which we draw $2m$ samples of the 
common Brownian motion in pairs $(w, -w)$ is a candidate for variance reduction. An open problem is to verify the 
usefulness of these techniques either numerically or analytically. 

\item Following on from the previous point, a natural extension to the model is to replace the systemic Brownian 
motion term in (\ref{eq:Intro_Coefficients}) with a L\'evy process. This would allow the possibility of generating 
extreme losses. Mathematically we expect to arrive at a non-linear SPDE driven by a L\'evy process on the 
half-line --- see, for example, \cite{kim2014}.

\item Another possibility for generating large systemic losses is to incorporate a contagion term in the particle 
dynamics along the lines of \cite{dirt_annalsAP_2015, dirt_SPA_2015}. For simplicity, consider the model where 
particles move according to the dynamics
\begin{align}
\label{eq:Open_FiniteContagion}
X^{i,N}_{t} &= X^{i}_{0} + W^{i}_{t} - \alpha L^{N}_{t}
	\\ \tau^{i} &= \inf\{ t > 0 : X^{i,N}_{t} \leq 0 \} \nonumber
	\\ L^{N}_{t} &= \sum_{i=1}^{N} \1_{ \tau^{i} \leq t}, \nonumber
\end{align}
with $\alpha > 0$. Whenever a particle hits the origin, every other particle jumps by size $\alpha / N$ towards the 
boundary. This can begin an avalanche effect where a default causes many other entities to default. Convergence 
of a finite particle system to a limiting McKean--Vlasov equation is shown in \cite{dirt_SPA_2015}, and it is known 
that for small values of $\alpha$ the solution is unique. For large values of $\alpha$ the limiting system undergoes 
a jump, whereby a macroscopic proportion of mass is lost in an infinitesimal period of time. It remains a challenge 
to prove uniqueness of solutions in this regime and to characterise a critical value of $\alpha$. From our perspective, 
a natural extension is to consider the system with a common Brownian noise term between particles. 

\end{enumerate}

\subsubsection*{Acknowledgement}
The authors thank the anonymous referees for their helpful corrections. We are grateful to Andreas Sojmark for his very thorough reading and suggestions for improvements. SL thanks Christoph Reisinger and Francois Delarue for discussions on this material.


  \bibliographystyle{plain}

 \end{document}